\documentclass{amsart}
\usepackage[top=2cm, bottom=2.8cm, left=2cm, right=2cm]{geometry}
\usepackage[english]{babel}
\usepackage{amsmath}
\usepackage{mathtools}
\usepackage{amsthm}
\usepackage{amsfonts}
\usepackage{amssymb}
\usepackage{graphicx}
\usepackage{mathrsfs}
\usepackage{a4wide,color,enumerate,mathrsfs}
\usepackage[normalem]{ulem}
\usepackage{amsmath,amssymb,epsfig,bbm}
\usepackage{pdfsync}
\usepackage{tikz-cd}
\usepackage{textcomp}
\numberwithin{equation}{section}
\usepackage[colorlinks=true,citecolor=green,linkcolor=red]{hyperref}
\usepackage{adjustbox}
\usepackage{caption}
\usepackage{tcolorbox}
\usepackage{ esint }
\newcommand{\N}{\mathbb{N}}

\newcommand{\R}{\mathbb{R}}
\renewcommand{\S}{\mathbb{S}}
\newcommand{\Z}{{\rm Z}}
\newcommand{\HH}{\mathscr{H}}

\newcommand{\MM}{\mathscr{M}}

\newcommand{\PP}{\mathscr{P}}

\renewcommand{\phi}{\varphi}
\newcommand{\cQ}{{\ensuremath{\mathcal Q}}}

\newcommand{\mm}{{\mbox{\boldmath$m$}}}

\newcommand{\oo}{{\mbox{\boldmath$o$}}}

\newcommand{\soo}{{\mbox{\scriptsize\boldmath$o$}}}

\newcommand{\dD}{{\mbox{\boldmath$D$}}}

\newcommand{\DDelta}{{\mbox{\boldmath$\Delta$}}}
\newcommand{\sfd}{{\sf d}}
\newcommand{\sfe}{{\sf e}}
\newcommand{\restr}[1]{\lower3pt\hbox{$|_{#1}$}}

\newcommand{\la}{{\big\langle}}                  % brackets
\newcommand{\ra}{{\big\rangle}}
\newcommand{\eps}{\varepsilon}  
\newcommand{\nchi}{{\raise.3ex\hbox{$\chi$}}}
\newcommand{\weakto}{\rightharpoonup}
\newtheorem{theorem}{Theorem}[section]

\newtheorem{corollary}[theorem]{Corollary}
\newtheorem{lemma}[theorem]{Lemma}
\newtheorem{proposition}[theorem]{Proposition}
\newtheorem{definition}[theorem]{Definition}

\theoremstyle{definition}
\newtheorem{remark}[theorem]{Remark}

\newcommand{\Per}[1]{{\rm Per}(#1)}

\newcommand{\Haus}[1]{{\mathscr H}^{#1}}     % Misura di Hausdorff
\newcommand{\Leb}[1]{{\mathscr L}^{#1}}      % Misura di Lebesgue

\newcommand{\LIP}{\mathrm{LIP}}
\newcommand{\Lip}{\mathrm{Lip}}
\newcommand{\lip}{\mathrm{lip}}

\newcommand{\diam}{\mathrm{diam}}

\newcommand{\fr}{\hfill$\blacksquare$}   %quadratino nero alla fine del remark, se non vi piace, la cosa migliore e' `svuotare' la macro, cosi' non bisogna intervenire sul testo
\newcommand{\rcd}[2]{\mathrm{RCD}(#1,#2)}
\newcommand{\RCD}{\mathrm{RCD}}
\newcommand{\cd}{\mathrm{CD}}
\renewcommand{\mm}{\mathfrak m}

\renewcommand{\limsup}{\varlimsup}
\renewcommand{\liminf}{\varliminf}
\renewcommand{\d}{{\rm d}}

\newcommand{\X}{{\rm X}}
\newcommand{\Y}{{\rm Y}}

\newcommand{\Xdm}{(\X,\sfd,\mm)}
\newcommand{\supp}{\text{supp}}
\newcommand{\rmCh}{{\rm Ch}}

\newcommand{\eucl}{{\sf Eucl}}
\newcommand{\vol}{{\rm Vol}}
\newcommand{\aopt}{A^{\rm opt}}
\renewcommand{\Cap}{{\rm Cap}}
\newcommand{\sca}{{\rm S}}

\makeatletter
	\def\l@subsection{\@tocline{2}{0pt}{2.5pc}{5pc}{}}
	\makeatother

\address{SISSA, Via Bonomea 265, 34136 Trieste}
\author[]{Francesco Nobili}
\email{francesco.f.nobili@jyu.fi}
\author[]{Ivan Yuri Violo}
\email{ivan.y.violo@jyu.fi}

\title[]{Rigidity and almost rigidity of  Sobolev inequalities on compact
 spaces with lower Ricci curvature bounds}
\begin{document}

\begin{abstract}
We prove that if $M$ is a closed $n$-dimensional Riemannian manifold, $n \ge 3$, with ${\rm Ric}\ge n-1$ and for which the optimal constant in the critical Sobolev inequality equals the one of the $n$-dimensional sphere $\S^n$, then $M$ is isometric to $\S^n$. An almost-rigidity result is also established, saying that if equality is almost achieved, then $M$ is close in the measure Gromov-Hausdorff sense to a spherical suspension. These statements are obtained in the  $\RCD$-setting  of (possibly non-smooth)  metric measure spaces satisfying synthetic lower Ricci curvature bounds.

An independent result of our analysis is the characterization of the best constant in the Sobolev inequality on any compact $\cd$ space, extending to the non-smooth setting a classical result by Aubin. Our arguments are based on a new concentration compactness result for mGH-converging sequences of $\RCD$ spaces and on a P\'olya-Szeg\H{o} inequality of Euclidean-type in $\cd$ spaces.

As an application  of the technical tools developed  we prove both an existence result for the Yamabe equation and  the continuity  of the generalized Yamabe constant under measure Gromov-Hausdorff convergence, in the $\RCD$-setting.
\end{abstract}

\maketitle

%\newpage
\setcounter{tocdepth}{2}
\tableofcontents
\medskip
\section{Introduction}
The standard Sobolev inequality in sharp form reads as
\begin{equation} \|u\|_{L^{p^*}(\R^n)} \le \eucl(n,p) \|\nabla u\|_{L^p(\R^n)}, \qquad \forall u \in W^{1,p}(\R^n),\label{eq:EuclSob}\end{equation}
where $p \in (1,n)$, $p^*:=\frac{pn}{n-p}$ is the Sobolev conjugate exponent and $\eucl(n,p)$ is the smallest positive constant for which the inequality \eqref{eq:EuclSob} is valid. Its precise value (see \eqref{eq:euclnp} below) was computed independently by Aubin \cite{Aubin76-2} and Talenti \cite{Talenti76} (see also \cite{C-ENV04}).

In the setting of compact Riemannian manifolds, the presence of constant functions in the Sobolev space immediately shows that an inequality of the kind of \eqref{eq:EuclSob} must fail. Yet, Sobolev embeddings are certainly valid also in this context and they can be expressed by calling into play the full Sobolev norm:
\begin{equation}\label{eq:ManifoldSob}\tag{{\color{red}$\star$}}
	\|u\|_{L^{p^*}(M)}^p\le A \|\nabla u\|_{L^p(M)}^p + B \|u\|_{L^p(M)}^p , \qquad \forall u \in W^{1,p}(M),
\end{equation}
where $M$ is a compact $n$-dimensional Riemannian manifold and $A,B>0.$
From the presence of the two parameters $A,B$,  it is not straightforward which is the notion of best constants in this case. The issue of defining and determining the best  constants in \eqref{eq:ManifoldSob} has been the central role of the celebrated $AB$-program, we refer to \cite{Hebey99} for a thorough presentation of this topic (see also \cite{DH02}). The starting point of this program is the definition of the following two  different notions of ``best Sobolev constants'':
\[
\begin{array}{cc}
\alpha_p(M)\coloneqq \inf \{ A \ : \  \eqref{eq:ManifoldSob} \text{ holds for some $B$}  \}, &\quad  \beta_p(M)\coloneqq \inf \{ B \ : \  \eqref{eq:ManifoldSob} \text{ holds for some $A$} \}. 
\end{array}
\]
Then the first natural problem is to determine the value of $\alpha_p(M)$ and $\beta_p(M)$. It is rather easy to see  that  $$\beta_p(M) = {\sf Vol}(M)^{p/p^*-1},$$ indeed  constant functions give automatically $\beta_p(M) \ge {\sf Vol}(M)^{p/p^*-1}$, while the other inequality follows from the Sobolev-Poincar\'e inequality (see, e.g. \cite[Sec 4.1]{Hebey99}). It is instead  more subtle to determine whether $\beta_p(M)$ is attained, in the sense that the infimum in its definition is actually a minimum. This is true for $p=2$ and due to Bakry \cite{Bakry94} (see also Proposition \ref{prop:sobolev embedding general}), but actually false for $p>2$ (see e.g. \cite[Prop. 4.1]{Hebey99}). 

 Concerning instead the value of $\alpha_p(M)$, it turns out to be related to the sharp constant in the Euclidean Sobolev inequality \eqref{eq:EuclSob}. More precisely  Aubin in \cite{Aubin76-2} (see also \cite{Hebey99}) showed  that on any compact $n$-dimensional Riemannian manifold $M$ with $n \ge 2$, we have
\begin{equation}\label{eq:alfa manifold}
	\alpha_{p}(M)= \eucl(n,p)^p\, \qquad \forall p \in(1,n).
\end{equation}
We point out that it is a hard task to show that $\alpha_p(M)$ is attained, namely that there exists some $B>0$ for which \eqref{eq:ManifoldSob} holds with $A=\alpha_p(M)$ and $B$. This has been verified for $p=2$ in \cite{HV96},  answering affirmatively to a conjecture of Aubin.

On the other hand, knowing the value of $\beta_p(M)$ (and that it is attained for $p=2$), we can define a further notion of  optimal-constant $A$, ``relative" to  $B=\beta_2(M)$. More precisely we define
$$\aopt_{2^*}(M)\coloneqq \vol(M)^{1- 2/{2^*}} \cdot  \inf \{A \ : \ \eqref{eq:ManifoldSob} \text{ for $p=2$ holds  with $A$ and }B=\vol(M)^{ 2/{2^*}-1} \}  .$$ 

For the sake of generality we will actually consider $\aopt$ also in the so-called \emph{subcritical case}, meaning that we enlarge the class of Sobolev inequalities and consider for every $q \in (2,2^*]$
\begin{equation}\label{eq:critical sobolev manifold}\tag{{\color{red}$\star \star$}}
	\|u\|_{L^{q}(M)}^2\le A   \|\nabla u\|_{L^2(M)}^2 +\vol(M)^{ 2/{q}-1}  \|u\|_{L^2(M)}^2 , \qquad \forall u \in W^{1,2}(M),
\end{equation}
for some constant $A\ge0$. Then we define 
$$\aopt_q(M)\coloneqq \vol(M)^{1- 2/q} \cdot  \inf \{A \ : \ \eqref{eq:critical sobolev manifold} \text{ holds}\}  .$$ 
Note that the infimum above is always a minimum and that $\vol(M)^{ 2/{q}-1}$ is the ``minimal B'' that we can take in \eqref{eq:critical sobolev manifold}.

\begin{remark}\rm 
We bring to the attention of the reader  the renormalization factor  $\vol(M)^{1- 2/q} $ in the definition of $\aopt_q(M).$ This is usually not  present in the literature concerning the $AB$-program (see e.g. \cite{Hebey99}), however this choice will allow us to have cleaner inequalities. This  also makes $\aopt_q$ invariant under rescalings of the volume measure of $M.$ \fr  
\end{remark}

One of the main questions that we will investigate in this note  concerns the value of $\aopt_q(M).$ So far $\aopt_q(M)$  is   known explicitly only in the case of $\S^n$ and was firstly computed by Aubin in \cite{Aubin76-3} in the case of $q=2^*$ and  by Beckner in \cite{Beckner93} for a general $q$:
\begin{equation}\label{eq:aubin}
\aopt_q(\S^n)=\frac{q-2}{n}, \qquad   \forall n\ge 3.
\end{equation}
Aubin also exhibited a family of non-constant functions that achieve  equality in \eqref{eq:critical sobolev manifold} with $A=\aopt_{2^*}(\S^n)$. For a general manifold $M$ instead it can be proved that
\begin{equation}\label{eq:upper bound curvature manifold}
\aopt_q(M)\le C(K,D,N),
\end{equation}
where $K\in \R$ is a lower bound on the Ricci curvature of $M$, $N$ is an upper bound on the dimension and  $D \in \R^+$ an upper bound on its diameter. This follows from the Sobolev-Poincar\'e inequality combined with an inequality by Bakry (see e.g. \cite[Theorem 4.4]{DH02} and also Section \ref{sec:upper bound A}). On the other hand, for positive Ricci curvature we have the following celebrated comparison result originally proven in  \cite{Said83} (see also \cite{Ledoux00,BGL14} for the case of a general $q$):
\begin{theorem}\label{thm:Aopt upper}
	Let $M$ be an $n$-dimensional Riemannian manifold, $n\ge 3$, with ${\rm Ric }\ge n-1$. Then, for every $q\in(2,2^*]$, it holds
	\begin{equation}\label{eq:Aopt upper}
		\aopt_q(M)\le \aopt_q(\S^n).
	\end{equation}
\end{theorem}
One of the main consequence of the results
in this note is the characterization of the equality in \eqref{eq:Aopt upper}, in particular we show:

\begin{theorem}\label{thm:rigidity manifold}
    Equality in \eqref{eq:Aopt upper} holds for some $q\in(2,2^*]$ if and only if $M$ is isometric to $\S^n.$
\end{theorem}
It is important to point out that the novelty of the above result is that it covers the case $q=2^*.$ Indeed, for $q<2^*$, Theorem \ref{thm:rigidity manifold} was already established (see e.g. \cite[Remark 6.8.5]{BGL14}) and follows from an improvement (only for $q<2^*$) of \eqref{eq:Aopt upper} due to \cite{Fontenas97} involving the spectral gap (see Remark \ref{rmk:sharper bound}  for more details). On the other hand, up to our knowledge, this is the first time that it appears  in the critical case $q=2^*.$

It is also worth to compare  Theorem \ref{thm:rigidity manifold} with the rigidity result in \cite{Ledoux99} for the Sobolev inequality on manifolds with non-negative Ricci curvature (and later improved in \cite{Xia01}, see also \cite{BK21}). In \cite{Ledoux99} it is proved that if \eqref{eq:EuclSob} is valid on a non-compact manifold with non-negative Ricci curvature, then the manifold must be  the Euclidean space. Here instead we consider compact manifolds and the rigidity is obtained in comparison with the Sobolev inequality on the sphere. For this reason, our arguments will also be substantially different from the ones in \cite{Ledoux99,Xia01}. Nevertheless, we will also deal with the former types of rigidity in Corollary \ref{thm:sharpSobgrowthtop} below.

Theorem \ref{thm:rigidity manifold} will be proved in the context of  metric measure spaces with synthetic Ricci curvature bounds. One of the main reasons to approach the problem in this more general setting is that it will allow us  to characterize also the ``almost-equality'' in  \eqref{eq:Aopt upper} (see Theorem \ref{thm:almrigidity} below). Indeed, as we will see, in this case we need to compare the manifold $M$ to a class of singular spaces, rather than to the round sphere.

\subsection{Best constant in the Sobolev inequality on compact $\cd$ spaces}\label{sec:alfa intro}

The notion of metric measure spaces with synthetic Ricci curvature bounds originated in the independent seminal works of \cite{Sturm06-1,Sturm06-2} and \cite{Lott-Villani09}, where the celebrated \emph{curvature-dimension condition} $\cd(K,N)$ was introduced. Here $K\in \R$ is a lower bound for the Ricci curvature  and $N\in[1,\infty]$ is an upper bound on the dimension. The definition is given via optimal transport, by requiring some convexity properties of entropy functionals (see Definition \ref{def:CDKN} below).

The proof of the rigidity (and almost rigidity) of  $\aopt_q$ in the case $q=2^*$, will force us to study also the value of $\alpha_p$  in the context of $\cd$-spaces. The connection of this with the proof of Theorem \ref{thm:rigidity}  will be explained towards the end of Section \ref{sec:strategy}, where we provide a sketch of the proof yielding the main rigidity theorem. 

Let then $\Xdm$ be a $\cd(K,N)$ space  with $N\in(1,\infty)$. For any $p\in (1,N)$ set $p^*\coloneqq \frac{Np}{N-p}$ and, in the same fashion of \eqref{eq:ManifoldSob}, we consider:
\begin{equation}\label{eq:alfa sobolev}
	\|u\|_{L^{p^*}(\mm)}^p\le A  \||D u|\|_{L^p(\mm)}^p + B  \|u\|_{L^p(\mm)}^p , \qquad \forall u \in W^{1,p}(\X).
\end{equation}
We are then interested in the minimal $A$ for which \eqref{eq:alfa sobolev} holds. In other words we set (with the usual convention that the inf is $\infty$ when no $A$ exists):
\begin{equation}\label{eq:alfa}
\alpha_p(\X)\coloneqq \inf \{ A \ : \  \eqref{eq:alfa sobolev} \text{ holds for some $B$}  \}.
\end{equation}
We will be able to compute the value of $\alpha_p(\X)$ for every compact $\cd(K,N)$ space $\X$, extending the result of Aubin for Riemannian manifolds (see \eqref{eq:alfa manifold} above). Before passing to the actual statement, it is useful to explain first the intuition behind it and the geometrical meaning of the constant $\alpha_p(\X).$ The rough idea is that its value is tightly linked to the  local structure of the space. Indeed, the key observation is that $\alpha_p(\X)$ is invariant under rescaling of the form $(\X,\sfd/r,\mm/r^N)$. For example, since manifolds are locally Euclidean, it is not  surprising that in \eqref{eq:alfa manifold} the optimal Euclidean-Sobolev constant appears. On the other hand,  $\cd(K,N)$ spaces  have a more singular local behavior and additional parameters must be taken into account. In particular the value of $\alpha_p(\X)$ turns out to be related to the Bishop-Gromov density:
\[
(0,+\infty]\ni \theta_{N}(x)\coloneqq \lim_{r \to 0^+}\,\frac{\mm(B_r(x))}{\omega_N r^N}, \quad x \in \X,
\]
where $\omega_N$ is the volume of the Euclidean unit ball (see \eqref{eq:omegasigma} for non integer $N$). Our result is then the following:
\begin{theorem}\label{thm:alfa}
	Let $\Xdm$ be a compact $\cd(K,N)$ space for some $K \in \R$ and $N \in(1,\infty)$. Then for every $p \in(1,N)$
	\begin{equation}\label{eq:alfa rcd}
		\alpha_{p}(\X)=\left(\frac{\eucl(N,p)}{\min_{x\in \X} \theta_N(x)^{\frac1N}}\right)^p.
	\end{equation}
\end{theorem}
We point out that, since $\X$ is compact, $\min_{x\in \X} \theta_N(x)$ always exists because $\theta_N$ is lower semicontinuous (see Section \ref{sec:cd rcd}).

\begin{remark}
	Note that if $\X$ is a $n$-dimensional Riemannian manifold, $\theta_n(x)=1$ for every $x \in \X$, hence in this case \eqref{eq:alfa rcd} (with $N=n$) is exactly Aubin's result in \eqref{eq:alfa manifold}.  Recall also that here $N$ needs not to be an integer and thus $\eucl(N,p)$ has to be defined for arbitrary $N\in(1,\infty)$ (see \eqref{eq:euclnp}). \fr
\end{remark}
\begin{remark}
	We are not assuming $\Xdm$ to be renormalized. In particular observe that if we rescale the reference measure $\mm$ as $c\cdot \mm$, then $\alpha_p$ gets multiplied by $c^{-p/N}$, which is in accordance with the scaling in \eqref{eq:alfa rcd}.  \fr
\end{remark}
\begin{remark}
	Theorem \ref{thm:alfa} gives non-trivial information even in the  ``collapsed'' case, i.e. when $\theta_N=+\infty$ in a set of positive (or even full) measure. Indeed, to have $\alpha_p(\X)>0$ it is sufficient that $\theta_N(x)<+\infty$ at a single point $x \in \X$. As an example, consider the model space $([0,\pi],|.|,\sin^{N-1}\Leb 1)$ which is $\cd(N-1,N)$ with $\theta_N(x)<+\infty$  only for $x \in \{0,\pi\}.$ \fr
\end{remark}
 Theorem \ref{thm:alfa} will be proved in two steps, by the combination of an upper bound (Theorem \ref{thm:Amin}), obtained via local Sobolev inequalities (Theorem \ref{thm:local sobolev intro}), and a lower bound (Theorem \ref{thm:A lower bound}) derived with a blow-up analysis.

We end this part  with a question that naturally arises  from the validity of Theorem \ref{thm:alfa}:

\vspace{\medskipamount}

\noindent{\bf Question:} Let $\Xdm$ be a compact $\cd (K,N)$ (or $\rcd NK$) space with $N \in(1,\infty)$ and suppose that $\alpha_{2^*}(\X)\in(0,\infty).$ Is there a constant $\overline B<+\infty$ such that 
\begin{equation}\label{eq:conjecture}
	\|u\|_{L^{2^*}(\mm)}^p\le \alpha_{2}(\X)  \||D u|\|_{L^2(\mm)}^2 + \overline B  \|u\|_{L^2(\mm)}^2 , \qquad \forall u \in W^{1,2}(\X)\,\, ?
\end{equation}
This has positive answer  in the smooth setting  \cite{HV96}. However in \cite[Proposition 5.1]{Hebey99} it is shown that on a Riemannian manifold  $M$ of dimension $n\ge 4,$ the scalar curvature of $M$ is bounded above by $c_n\overline B$, for a dimensional constant $c_n>0$. This  points to a negative answer, since we are assuming only a  Ricci lower bound on the space, however it is not  clear to us how to prove or disprove \eqref{eq:conjecture}.\medskip

\subsection{Main rigidity and almost rigidity results in compact $\RCD$ spaces}

Even if some of our results will hold for the general class of $\cd(K,N)$ spaces, our main focus will be the smaller class of  spaces satisfying the \emph{Riemannian curvature-dimension condition} $\rcd{K}{N}$, which adds to the $\cd$ class the linearity of the heat flow (see Definition \ref{def:RCDKN} below). This notion appeared first in the infinite dimensional case ($N=\infty$) in \cite{AGS14}  (see also \cite{AGMR15} in the case of $\sigma$-finite reference measure) while, in the finite dimensional case ($N<\infty$), it was introduce in \cite{Gigli12}. We also mention the slightly weaker  $\RCD^*(K,N)$ condition (coming from the \emph{reduced curvature-dimension condition} $\cd^*(K,N)$ introduced in \cite{BacherSturm10}) which has been proved in \cite{EKS15,AmbrosioMondinoSavare13-2} to be equivalent to  the validity of a weak $N$-dimensional Bochner-inequality (see also \cite{AmbrosioGigliSavare12} for the same result in the infinite dimensional case). We recall that in the compact case (or more generally for finite reference measure) which will be the main setting of this note, the $\RCD^*(K,N)$ and the $\RCD(K,N)$ conditions turn out to be perfectly equivalent after the work in \cite{CM16}. The main advantage for us to work in the $\RCD$ class, as opposed to the more general $\cd$ class, is that it enjoys  rigidity and stability properties that are analogous to the Riemannian manifolds setting.

To state our main results for metric measure spaces we need to define first the notion of optimal constant in the Sobolev inequality in the non-smooth setting. Given a (compact) $\RCD(K,N)$ space  (or more generally a $\cd(K,N)$ space) $\Xdm$, for some $K \in \R$,  $N\in (2,\infty)$, we set $2^*\coloneqq 2N/(N-2)$ and consider the analogous of \eqref{eq:critical sobolev manifold}:
\begin{equation}\label{eq:critical sobolev}
	\|u\|_{L^{q}(\mm)}^2\le A \||D u|\|_{L^2(\mm)}^2 + \mm(\X)^{2/q-1}  \|u\|_{L^2(\mm)}^2 , \qquad \forall u \in W^{1,2}(\X),
\end{equation}
for $q \in (2,2^*]$ and a constant $A\ge0$. Then we define $$\aopt_q(\X)\coloneqq \mm(\X)^{1- 2/q} \cdot \inf  \{A \ : \ \eqref{eq:critical sobolev} \text{ holds}\}, $$ with the convention that $\aopt_q(\X)=\infty$ when no $A$ exists. Note that $\aopt_q(\X)$, when is finite, is actually a minimum. Observe also that, as in the smooth case, there is a renormalization factor $\mm(\X)^{1- 2/q}$ in the definition. However, being not restrictive, we will mainly work asking $\mm(\X)=1$ so that the value of $\aopt_q(\X)$ is equivalent to the  non-renormalized one.

Remarkably in this more general framework, a comparison analogous to \eqref{eq:Aopt upper} holds.
\begin{theorem}[\cite{CM17}]\label{thm:Aopt upper cd}
	Let $\Xdm$ be an essentially non-branching $\cd(N-1,N)$ space, $N \in (2,\infty)$. Then, for every $q\in(2,2^*]$
	\begin{equation}\label{eq:Aopt upper cd}
		\aopt_q(\X)\le \frac{q-2}{N}.
	\end{equation}
\end{theorem}
The essentially nonbranching condition is a technical property of mass transportation that, roughly said, requires a suitable nonbranching property of transportation geodesics. It was introduced in \cite{RS14} where it was shown that it is satisfied in the $\RCD(K,N)$-class. We also mention that Theorem \ref{thm:Aopt upper cd} in the $\RCD$ case was previously obtained in \cite{Profeta15}. Observe also that, whenever $N$ is an integer and thanks to \eqref{eq:aubin}, for a $N$-dimensional Riemannian manifolds \eqref{eq:Aopt upper cd} is exactly \eqref{eq:Aopt upper} and in particular Theorem \ref{thm:Aopt upper cd} generalizes Theorem \ref{thm:Aopt upper}.

We can now state our main rigidity result in the setting of metric measure spaces.

%The study of Sobolev (or more generally, geometric) sharp inequalities comes very naturally with stability questions. This is also the case of the Theorem below, where we investigate the shape of spaces attaining almost best Sobolev constant.
\begin{theorem}[Rigidity of $\aopt_q$]\label{thm:rigidity}
	Let $\Xdm$ be an $\RCD(N-1,N)$ space for some $N\in(2,\infty)$ and let $q\in(2,2^*]$. Then, equality holds in \eqref{eq:Aopt upper cd} if and only if $\Xdm$ is isomorphic to a spherical suspension, i.e. there exists an $\RCD(N-2,N-1)$ space $(\Z,\sfd_\Z,\mm_\Z)$ such that $\Xdm \simeq [0,\pi] \times_{\sin}^{N-1} \Z.$
\end{theorem}
Differently from the smooth case, in the more abstract setting of $\RCD$ spaces the above result is instead new for all $q$. As anticipated above, we can also prove an ``almost-rigidity'' statement linked to the almost-equality case in \eqref{eq:Aopt upper cd} (see Section \ref{sec:mgh} for the notion of measure-Gromov-Hausdorff convergence and distance $\sfd_{mGH}.$).
\begin{theorem}[Almost-rigidity of  $\aopt_q$]\label{thm:almrigidity}
		For every $N\in(2,\infty)$, $q \in(2,2^*]$ and every  $\eps>0$, there exists $\delta:= \delta(N,\eps,q)>0$ such that the following holds. Let $\Xdm$ be an $\RCD(N-1,N)$ space with $\mm(\X)=1$ and suppose that 
	\[ 
	\aopt_q(\X)\ge \frac{(q-2)}{N}-\delta,
	\]
	Then, there exists a spherical suspension $(\Y,\sfd_\Y,\mm_\Y)$ (i.e. there exists an $\RCD(N-2,N-1)$ space $(\Z,\sfd_\Z,\mm_\Z)$  so that $\Y$ is isomorphic as a metric measure space to $[0,\pi] \times^{N-1}_{\sin} \Z$) such that
	\[ \sfd_{mGH}( (\X,\sfd,\mm),(\Y,\sfd_\Y,\mm_\Y)) <\eps.\]
\end{theorem}
\begin{remark} \rm We briefly point out two important facts concerning the two above statements.
\begin{itemize}
    \item[$i)$]	 In the smooth setting,  for $q <2^*$, the almost rigidity follows ``directly'' from the sharper version of \eqref{eq:Aopt upper} cited above (see Remark \ref{rmk:sharper bound} for the explicit statement) and using the almost-rigidity of the $2$-spectral gap \cite{CM17,CMS19}. Nevertheless, we are not aware of any such statement in the literature and anyhow, our proof does not rely on any improved version of \eqref{eq:Aopt upper}.
    \item[$ii)$] The key feature of Theorem \ref{thm:rigidity} and Theorem \ref{thm:almrigidity} is that they include the ``critical'' exponent. Indeed, the difference between the ``subcritical'' case $q<2^*$ and $q=2^*$ is not only technical but a major issue linked to the lack of compactness in the Sobolev embedding. As it will be clear in the sequel, the proof of the critical case requires several additional arguments that constitute the heart of this note.\fr
\end{itemize}
\end{remark}

The almost-rigidity result contained in Theorem \ref{thm:almrigidity} will be actually a consequence of a stronger statement, that is the continuity of $\aopt_q$ under measure Gromov-Hausdorff convergence. More precisely we will prove the following:
\begin{theorem}[Continuity of $\aopt_q$ under mGH-convergence]\label{thm:mGHAopt}
Let $(\X_n,\sfd_n, \mm_n)$, $n \in \N\cup \{\infty\}$, be a sequence of compact $\RCD(K,N)$-spaces with $\mm_n(\X_n)=1$ and for some $K\in \R$, $N \in(2,\infty)$ so that $\X_n\overset{mGH}{\to} \X_\infty$. Then, $A_q^{\rm opt}(\X_\infty)=\lim_n \aopt_q(\X_n)$, for every $q \in (2,2^*]$.
\end{theorem}

\subsection{Additional results and application to the Yamabe equation}
\hphantom{1cm}\medskip

\noindent\emph{Euclidean-type P\'olya-Szeg\H{o} inequality on ${\cd(K,N)}$ spaces.}
We will develop  a  P\'olya-Szeg\H{o} inequality (see Section \ref{sec:euclidean polya}), which is roughly a Euclidean-variant of the P\'olya-Szeg\H{o} inequality for $\cd(K,N)$ spaces, $K>0$, derived in \cite{MS20}. The main feature of this inequality is that it holds on arbitrary $\cd(K,N)$ spaces, $K \in \R$, but  assumes the validity of an isoperimetric inequality of the type
\[ 
\Per{E} \ge {\sf C_{Isop}} \mm(E)^{\frac{N-1}N}, \qquad \forall E \subset \Omega \text{ Borel},
\]
for some $\Omega\subset \X$ open and where  ${\sf C_{Isop}}$ is a positive constant independent of $E$.  For our purposes this P\'olya-Szeg\H{o} inequality will be used to derive local Sobolev inequalities of Euclidean-type (see Theorem \ref{thm:local sobolev intro}), however it allows us to obtain also sharp Sobolev inequalities under Euclidean-volume growth assumption.

\medskip
\noindent\emph{Sharp and rigid Sobolev inequalities under Euclidean-volume growth.} As a by-product of our analysis, we achieve sharp Sobolev inequalities on ${\cd (0,N)}$ spaces with Euclidean-volume growth. We recall that a ${\cd (0,N)}$ space $\Xdm$ has Euclidean-volume growth if
\[
{\sf AVR}(\X)\coloneqq \lim_{R\to +\infty} \frac{\mm(B_R(x_0))}{\omega_NR^N}>0,
\] 
for some (and thus any) $x_0 \in \X$. We will prove the following.
\begin{theorem}\label{thm:sharpSobgrowth}
	Let $\Xdm$ be a $\cd(0,N)$ space for some $N \in (1,\infty)$ and with Euclidean volume growth. Then, for every $p \in (1,N)$, it holds
	\begin{equation}\label{eq:sharp sobolev growth}
		\|u\|_{L^{p^*}(\mm)}\le \eucl(N,p){\sf AVR(\X)}^{-\frac{1}{N}}\|| D u|\|_{L^p(\mm)}, \qquad \forall u \in \LIP_c(\X).
	\end{equation}
Moreover \eqref{eq:sharp sobolev growth} is sharp.
\end{theorem}
This extends a result recently derived in \cite{BK21} in the case of Riemannian  manifolds and answers positively to a question posed in \cite[Sec. 5.2]{BK21}.

Combining Theorem \ref{thm:sharpSobgrowth} with the volume rigidity for 
\emph{non-collapsed} $\RCD$ spaces  in \cite{GDP17} and the results in \cite[Appendix A]{Cheeger-Colding97I} (see also  \cite[Theorem 3.5]{KapMon21}) we immediately get  the following topological rigidity which extends to the non-smooth setting the results for Riemannian manifolds in \cite{Ledoux99,Xia01}.
Recall that an $\RCD(K,N)$ space $\Xdm$ is said to be non-collapsed (see Definition \ref{def:ncRCD}) if $\mm=\Haus{N}$, the $N$-dimensional Hausdorff measure   (this notion has been introduced in \cite{GDP17}, see also \cite{Kita17} and inspired by \cite{Cheeger-Colding97I}).

\begin{corollary}[Topological-rigidity of Sobolev embeddings]\label{thm:sharpSobgrowthtop}
For every $N \in \N$, $p\in(1,N)$ and $\eps>0$ there exists $\delta>0$ such that the following holds.
	Let $(\X,\sfd,\Haus{N})$ be an $\RCD(0,N)$ space  with Euclidean volume growth and such that
	\begin{equation}\label{eq:sharp sobolev growth top}
		\|u\|_{L^{p^*}(\mm)}\le (\eucl(N,p)+\delta)\|| D u|\|_{L^p(\mm)}, \qquad \forall u \in \LIP_c(\X).
	\end{equation}
Then $\X$ is homeomorphic to $\R^N$ and $\sfd_{GH}(B_r(x),B_r(0^N))\le \eps r$ for every $x \in \X$ and $r>0.$
\end{corollary}
To deduce the above result, a lower bound on the optimal constant in \eqref{eq:sharp sobolev growth} is actually sufficient (see Theorem \ref{thm:sobolev implies growth}).

\medskip
\noindent\emph{Concentration compactness and mGH-convergence.}
As often happens for almost-rigidity results in ${\rm RCD}$ spaces, Theorem \ref{thm:almrigidity} will be proved by compactness. However, in the case $q=2^*$ we have a strong lack of compactness, hence for the proof we will need an additional tool, which is a concentration compactness result under mGH-convergence of compact ${\rm RCD}$-spaces. In particular, we will prove a concentration-compactness dichotomy principle (see Lemma \ref{lem:conccompII} and Theorem \ref{thm:CC_Sob} below) in the spirit of \cite{Lions85} (see also the monograph \cite{Struwe08}), but under varying underlying measure. As far as we know, this is the first result of this type dealing with varying spaces and we believe it to be interesting on its own.
\medskip

\noindent\emph{Existence for the Yamabe equation and mGH-continuity of Yamabe constant on $\RCD$ spaces}
As an application of Theorem \ref{thm:alfa} we show that on a compact $\RCD(K,N)$ space a (non-negative and non-zero) solution to the so-called Yamabe equation
\begin{equation}\label{eq:yamabe intro}
    -\Delta u+\sca \, u =\lambda u^{2^*-1}, \qquad \text{for }\lambda \in \R,\  \sca \in L^p(\mm),\,p>N/2,
\end{equation}
exists provided 
\[\lambda_\sca(\X)\coloneqq \inf_{u \in W^{1,2}(\X)\setminus \{0\}} \frac{\int |D u|^2+ \sca |u|^2 \,\d \vol}{\|u\|_{L^{2^*}(M)}^2}<\frac{\min \theta_N^{N/2}}{\eucl(N,2)^2},\]
where $\lambda_\sca$ is called generalized Yamabe constant (see Theorem \ref{thm:4.3}). This extends a classical result on smooth Riemannian manifolds (see Section \ref{sec:scalarPDE} for more details and references).

We also show the continuity of the generalized Yamabe constant under measure Gromov-Hausdorff convergence. More precisely for a sequence $\X_n$ of compact $\RCD(K,N)$ spaces such that $\X_n \overset{mGH}{\rightarrow}\X_\infty$ with $\X_\infty$ a compact $\RCD(K,N)$ space, we show that
$$\lim_n \lambda_{\sca_n}(\X_n)=\lambda_{\sca}(\X_\infty),$$
where $\sca_n$ converges $L^p$-weak to $\sca$ for some $p>N/2$. See Theorem \ref{thm:GHstablelambda} for a precise statement and Section \ref{sec:mgh} for the definition of $L^p$-weak convergence with varying spaces. This result extends and sharpens an analogous statement proved for Ricci-limits in \cite{Honda18}, where an additional boundedness assumption on the sequence $\lambda_{\sca_n}(\X_n)$ is required.

\subsection{Proof-outline of the rigidity of $\aopt_q$}\label{sec:strategy}

Here we explain the scheme of the proof of the rigidity  result in Theorem \ref{thm:rigidity}.

We consider only the case $q = 2^*$, since it is the most interesting one and we also restrict to
the case of manifolds, which already contains all the main ideas.

Let $M$ be a  compact $n$-manifold $M$, with ${\rm Ric}\ge n-1$ and $\aopt_{2^*}(M)=\aopt_{2^*}(\S^n)$, $n\ge 3$. This is equivalent to the existence of a sequence $(u_i)\subset W^{1,2}(M)$ of non-constant functions satisfying $\|  u_i \|^2_{L^{2^*}(\mm)}=1$ and
	\begin{equation} \cQ(u_i):=\frac{\|  u_i \|^2_{L^{2^*}}- \vol(M)^{-2/n}\|  u_i \|^2_{L^2} }{\vol(M)^{-2/n}\| \nabla   u_i\|^2_{L^2} } \rightarrow A^{\rm opt}_{2^*}(\S^n). \label{eq:extremal}
	\end{equation}
	In a nutshell, the strategy of the proof consists in a fine investigation of these sequences.

We will show (Theorem \ref{thm:CC_Sob}) that $(u_i)$ up to a subsequence can have only  three  possible behaviors. For each case the conclusion will follow applying a different rigidity theorem.

\noindent\textsc{Case 1} (Convergence to extremal). The sequence $u_i $ converges in $L^{2^*}$ to a \emph{non constant} extremal function $u$ such that $\cQ(u)= \aopt_{2^*}(M)=\aopt_{2^*}(\S^n)$. This forces the monotone rearrangement $u^*: \S^n\to \R$ (as defined in Section \ref{sec:polya}) to achieve equality in the P\'olya-Szeg\H{o} inequality. Since $u$ is assumed not constant, the rigidity case of the P\'olya-Szeg\H{o} inequality (see Theorem \ref{thm:Polyarigidity}) ensures $M=\S^n$.

\noindent\textsc{Case 2}  (Convergence to constant). The sequence $u_i $ converges  in $L^{2^*}$ to a \emph{constant} function $u\equiv c$. Up to renormalization (of the volume measure), it can be assumed that $\int u_i=1$ and $u\equiv 1$. In this case the rigidity follows exploiting that the linearization of the Sobolev inequality is the Poincaré inequality.  More precisely we write $u_i= 1+v_i,$ so that $v_i\coloneqq u_i-1$ has zero mean. Then it can be shown that:
\[ \frac{2^*-2}{n}=\aopt_{2^*}(\S^n) = \lim_{i\to \infty} \cQ (u_i) = \lim_{i\to \infty }\frac{(2^*-2)\|v_i\|_{L^2}^2}{\| \nabla v_i\|_{L^2}^2} \le \frac{2^*-2}{\lambda_1(M)}, \]
where $\lambda_1(M)$ is the spectral gap. This forces $\lambda_1(M)=n$ and the conclusion follows by classical Obata's rigidity theorem.

\noindent\textsc{Case 3}  (Concentration in a single point). The sequence $u_i$ vanishes, i.e. $\|u_i\|_{L^2}\to 0$ (in fact the following concentration happens:  $|u_i|^{2^*}\rightharpoonup \delta_{p}$ for some point $p\in M$).  Here is where the constant $\alpha_{2}(M)$ enters into play. Indeed, by definition of $\alpha_2(M)$, for every $\eps>0$ there exists $B_\eps$ such that 
\[
	1=\|u_i\|_{L^{2^*}}^2\le (\alpha_2(M)+\eps)  \|\nabla  u_i\|_{L^2}^2 + B_\eps  \|u_i\|_{L^2}^2, \qquad \forall \,i\in \N.
\]
Moreover, from $\|u_i\|_{L^2}\to 0$ we must have $\liminf _i \|\nabla  u_i\|_{L^2}^2>0 $.  Combining these two observation we obtain that
\[
\limsup_i \frac{\|u_i\|_{L^{2^*}}^2}{ \|\nabla  u_i\|_{L^2}^2}\le (\alpha_2(M)+\eps).
\]
By assumption $\cQ(u_i)\to \aopt_{2^*}(\S^n)$, which implies
\[
\liminf_i \frac{\|u_i\|_{L^{2^*}}^2}{ \|\nabla  u_i\|_{L^2(M)}^2}\ge \vol(M)^{-2/n}\aopt_{2^*}(\S^n).
\]
Therefore $\alpha_2(M)\ge \vol(M)^{-2/n}\aopt_{2^*}(\S^n).$ However combining \eqref{eq:alfa manifold} with $$ \eucl(n,2)^2=\aopt_{2^*}(\S^n)\vol(\S^n)^{-2/n},$$ we have $\alpha_2(M)=\aopt_{2^*}(\S^n)\vol(\S^n)^{-2/n}$, that coupled with the previous observation yields $
\vol(M)\ge  \vol(\S^n).$
This  and the Bishop-Gromov volume ratio implies that $\vol(M)=\vol(\S^n)$, which forces $\diam(M)=\pi$ and the required rigidity follows from Cheng's diameter rigidity theorem.
\medskip

\noindent \textbf{Structure of the paper}. This note is organized as follows:

We begin in Section \ref{sec:preliminaries} with the necessary preliminaries concerning Sobolev calculus on metric measure spaces and the main properties of $\cd / \RCD$ spaces. 

Section \ref{sec:upper bound alpha} is devoted to show the upper bound of $\alpha_p(\X)$ in \eqref{eq:alfa rcd}. This upper bound is  obtained from a class of local Euclidean Sobolev inequalities (Theorem \ref{thm:local sobolev intro}). To prove these inequalities we  develop, in the general framework of $\cd(K,N)$ spaces, a Euclidean P\'olya-Szeg\H{o} inequality (Section \ref{sec:euclidean polya}), which is then  coupled with a local isoperimetric inequality of Euclidean type (Theorem \ref{theorem:almost euclidean isop}).

Section \ref{sec:lower bound alfa} is devoted to achieve the lower bound of $\alpha_p(\X)$ in \eqref{eq:alfa rcd} and, combined with the previous section, the proof of Theorem \ref{thm:alfa}. Here we also derive, as an application, sharp Sobolev inequalities on $\cd(0,N)$ spaces (Section \ref{sec:sharp sobolev}).

In Section \ref{sec:Aopt}, we consider three different geometric bounds on the optimal constant $\aopt_q$ in the  Sobolev inequality \eqref{eq:critical sobolev}: an upper bound depending on the Ricci curvature bounds (Section \ref{sec:upper bound A}), a lower bound in terms of the first eigenvalue (Section \ref{sec:lower bound aopt}) and a lower bound in terms of the diameter (Section \ref{Secdiameterbounds}).

In Section \ref{sec:rigidity}, we prove our main rigidity result on $\aopt_q$, namely Theorem \ref{thm:rigidity}. To this aim we  develop  a concentration compactness dichotomy principle under mGH-convergence, in the $\RCD(K,N)$ setting (Theorem \ref{thm:CC_Sob}). The second ingredient for the rigidity is instead a quantitative linearization lemma for the Sobolev inequality that we prove in Section \ref{sec:linearization}.

In Section \ref{sec:almrigidity}, we  prove the main almost-rigidity result of this note stated in Theorem \ref{thm:almrigidity}. This will be obtained as a consequence of the continuity of the constant $\aopt_q$ under mGH-convergence (Section \ref{sec:continuity aopt}). For this result we will need to fully exploit the concentration compactness tools under mGH-convergence developed in the previous section.

Finally, in Section \ref{sec:scalarPDE} we conclude this note by studying the so-called generalize Yamabe equation on $\RCD(K,N)$ spaces. We will prove a classical existence result in Section \ref{sec:existence yamabe} while in Section \ref{sec:continuity yamabe} we will show a continuity result for the generalized Yamabe constant under mGH-convergence.

\section{Preliminaries}\label{sec:preliminaries}
\subsection{Basic notations}
We collect once and for all the key constants appearing in this note.

\begin{quote}
For all $ N \in [1,\infty), p \in (1,N)$, we define the generalized\footnote{For an integer $N$, $\omega_N$ is the volume of the unit ball in $\R^N$ and $\sigma_N$ is the volume of the $N$-sphere $\S^N.$} unit ball and unit sphere volumes by
\begin{equation}
\omega_N\coloneqq\frac{\pi^{N/2}}{\Gamma\left(N/2+1\right)}, \quad \sigma_{N-1}\coloneqq N\omega_{N},\label{eq:omegasigma}
\end{equation}
where $\Gamma$ is the Gamma-function, and the sharp Euclidean Sobolev constant by
\begin{equation}
\eucl(N,p)\coloneqq \frac 1N\Big(\frac{N(p-1)}{N-p}\Big)^{\frac{p-1}p}\Big(\frac{\Gamma(N+1)}{N\omega_N\Gamma(N/p)\Gamma (N+1-N/p)}\Big)^{\frac 1N}.\label{eq:euclnp}
\end{equation}
For $N>2$ and $p=2$, the above reduces to
\begin{equation}
\eucl(N,2)= \Big(\frac{4}{N(N-2)\sigma_N^{2/N}}\Big)^\frac{1}{2}.\label{eq:eucln2}
\end{equation}
We will sometimes need also the following identity: 
\begin{equation}\label{eq:integral sin}
    \int_0^\pi \sin^{N-1}(t)\, \d t = \frac{\sigma_{N}}{\sigma_{N-1}}, \quad \forall N >1.
\end{equation}
\end{quote}

Throughout this note a metric measure space will be a triple $\Xdm$, where
\[\begin{split}
(\X,\sfd)\qquad & 	\text{is a complete and separable metric space},\\
\mm\neq 0\qquad & \text{is non negative and boundedly finite Borel measure}.
\end{split}
\]
To avoid technicalities, we will work under the assumption that $\supp(\mm)=\X$. 

We will denote by  $\LIP(\X),\, \LIP_b(\X),\, \LIP_{bs}(\X)$, $\LIP_c(\X)$, $C(\X)$, $C_b(\X)$ and $C_{bs}(\X)$ respectively the spaces of Lipschitz functions, Lipschitz and bounded functions, Lipschitz functions with bounded support, Lipschitz functions with compact support, continuous functions, continuous and bounded functions and continuous functions with bounded support on $\X$. We will also denote  by $\LIP_c(\Omega)$ and $\LIP_{loc}(\Omega)$, for $\Omega\subset \X$ open, the spaces of Lipschitz functions with compact support and locally Lipschitz functions in $\Omega$. Moreover, if $f \in \LIP(\X)$, we denote by $\Lip(f)$ its Lipschitz constant, and we say that $f$ is $L$-Lipschitz, for $L>0$, if $\Lip(f)\le L$. Also, we recall the notion of local Lipschitz constant for a locally Lipschitz function $f$:
\[ \lip \ f(x) := \limsup_{y\rightarrow x} \frac{|f(y)-f(x)|}{\sfd(x,y)},\]
taken to be $0$ if $x$ is isolated.

We will denote by  $ \MM^+_b(\X)$ and $\PP(\X)$ respectively the  space of  Borel non-negative finite measure and Borel probability measures on $\X$. By  $\PP_2(\X)$, we denote the class of probability measures with finite second moment, that is the space of all $\mu \in \PP(\X)$ so that $\int \sfd^2(x,x_0)\, \d \mu(x)<\infty$ for some (and thus, any) $x_0 \in \X$. Given two complete metric spaces $(\X,\sfd),(\Y,\sfd_\Y)$ a Borel measure $\mu$ on $\X$ and a Borel map $\varphi \colon \X\to\Y$, the \emph{pushforward} of $\mu$ via $\varphi$, is the measure  $(\varphi)_\sharp \mu$ on $\Y$ defined by $(\varphi)_\sharp \mu(E):= \mu(\varphi^{-1}(E))$ for every $E\subset \Y$. Then two metric measure spaces $(\X_i,\sfd_i,\mm_i)_{i=1,2}$ are said to be \emph{isomorphic}, $\X_1 \simeq \X_2$ in short, if there exists an isometry $\iota : \X_1\to \X_2$ such that $(\iota)_\sharp\mm_1=\mm_2.$

%We will consider equipping these classes with the \emph{weak topology} in duality with $C_b(\X)$ and we recall that, since we are assuming $\X$ to be Polish, it is metrizable. Sometimes, we will deal with compact metric measure spaces $\X$, where evidently this topology reduces to the weak$^*$-topology. 

For $B \subset \X$ we will denote by diam$(B)$ the quantity $\sup_{x,y \in B}\sfd(x,y)$. We say that a metric measure space $\Xdm$ is locally doubling if for every $R>0$, there exists a constant $C:=C(R)$ so that
\[ \mm(B_{2r}(x)) \le C\mm(B_r(x)), \qquad \forall x \in \X, r \in (0,R).\]
 Whenever $C(R)$ can be taken independent of $R$ we say that $(\X,\sfd,\mm)$ is doubling. 
 
 A geodesic for us will denote a constant speed length-minimizing curve between its endpoints and defined on $[0,1]$, i.e. a curve $\gamma :[0,1]\to \X$ so that $\sfd(\gamma_t,\gamma_s)=|t-s|\sfd(\gamma_0,\gamma_1)$, for every $t,s \in [0,1]$. Also, we denote by ${\rm Geo}(\X)$ the set of all geodesics and call $\X$ a geodesic metric space, provided for any two couple of points, there exists a geodesic linking the two as already discussed. To conclude, we define the evaluation map $\sfe_t$, $t \in [0,1]$, as the assignment $\sfe_t \colon C([0,1],\X) \to \X$ defined via $\sfe_t(\gamma):= \gamma_t$.

\subsection{Calculus on metric measure spaces}
\subsubsection{Sobolev spaces}
We start recalling the notion of Sobolev spaces in a metric measure space. We refer to \cite{HKST15,Gigli14, GP20} for more details on this topic. 

The concept of Sobolev space for a metric measure space was introduced in the seminal works of Cheeger \cite{Cheeger99} and of Shanmugalingam \cite{Shan00}, while here we  adopt the approach  via Cheeger energy developed in \cite{AGS13} and proved there to be  equivalent with the notions in  \cite{Shan00,Cheeger99}.

Let $p \in (1,\infty)$ and $\Xdm$ be a metric measure spaces.  The $p$-Cheeger energy $\rmCh_p \colon L^p(\mm)\to [0,\infty]$ is defined as the convex and lower semicontinuous functional 
\[ \rmCh_p(f) := \inf \Big\{ \liminf_{n\to\infty} \int \lip^p  f_n\, \d \mm \colon (f_n) \subset L^p(\mm)\cap \LIP(\X), \lim_n\|f-f_n\|_{L^p(\mm)} =0 \Big\}.\]
The $p$-Sobolev space is then defined as the space $W^{1,p}(\X):= \{ \rmCh_p <\infty\}$  equipped with the norm $\| f\|^p_{W^{1,p}(\X)} := \|f\|_{L^p(\mm)}^p + \rmCh_p(f)$, which makes it  a Banach space. Under the assumption that $\Xdm$ is doubling, $W^{1,p}(\X)$ is reflexive as proven in \cite{ACDM15} and in particular the class $\LIP_{bs}(\X)$ is dense in $W^{1,p}(\X)$ (see also the more recent \cite{E-BS21}). Finally, exploiting the definition by relaxation given for the $p$-Cheeger energy, it can be proved (see \cite{AGS13}) that whenever $f \in W^{1,p}(\X)$, then there exists a minimal $\mm$-a.e.\ object $|Df|_p \in L^p(\mm)$ called minimal $p$-weak upper gradient so that
\[ \rmCh_p(f) := \int |Df|_p^p\, \d \mm.\]
In general, the dependence on $p$ of such object is hidden and not trivial (that is why we introduced the $p$-subscript in the object $|Df|_p$), as shown for example in the analysis \cite{DMS15}. Nevertheless, in this note, we are mainly concerned in working on a class of spaces, which will be later discussed, where such dependence is ruled out (see Remark \ref{rmk:p independence CD}). In this case, the subscript will be automatically omitted. 

We will often need to consider the case when $\Xdm$ is a weighted interval with a weight that is bounded away from zero, i.e. $\Xdm=([a,b],|.|,h \Leb 1)$, $a,b\in \R$ with $a<b$ where $h \in L^1([a,b])$ and for every $\eps>0$ there exists $c_{\eps}>0$ so that $h \ge c_\eps$ $\Leb 1$-a.e.\ in $[a+\eps,b-\eps].$ In this case, we denote by $W^{1,p}([a,b],|.|,h \Leb 1)$ the $p$-Sobolev space over the weighted interval according to the metric definition relying on the Cheeger energy, while simply write $W^{1,p}(a,b)$ for the classical definition via integration by parts. It can be shown that (for example using  \cite[Remark 4.10]{AGS13})
\begin{equation}\label{eq:sobolev ac}
    f \in W^{1,p}([a,b],|.|,h \Leb 1) \quad \iff \quad \begin{array}{l} f \in W_{loc}^{1,1}(a,b)\text{ with }  f,f'\in L^p(h\Leb 1),
    \end{array}
\end{equation}
in which case $ |Df|_p=|f'|$ $\Leb 1$-a.e..

%We list now some key calculus tools concerning the minimal weak upper gradient:
%\begin{itemize}
%   \item[] \textsc{Locality} $|Df|=|Dg|$ $\mm$-a.e. on $\{f=g\}$.;
%    \item[]\textsc{Chain rule} $\phi\circ f \in W^{1,p}(\X)$ and $|D(\phi\circ f)| \le \Lip(\phi)|Df|$ $\mm$-a.e.;
%    \item[] \textsc{Leibniz} $h f \in W^{1,p}(\X)$ and $|D(h f)| \le \Lip(h) f + h |Df|$ $\mm$-a.e..
%\end{itemize}
%for every $f,g \in W^{1,p}(\X)$, $\phi \in \LIP(\R)$ and $h \in \LIP_b(\X)$. 

For every $\Omega \subset \X$ we   define also the local Sobolev space $W^{1,p}_0(\Omega)\subset L^p(\Omega)$ as
\[
W^{1,p}_0(\Omega)\coloneqq \overline{ \LIP_c(\Omega)}^{W^{1,p}(\X)}. 
\]
From the previous discussion, if $\X$ is locally compact and  locally doubling, then $W^{1,p}_0(\X)=W^{1,p}(\X).$

Next, according to the definition given in \cite{Gigli12}, we say that $\Xdm$ is  \emph{infinitesimally Hilbertian} if $W^{1,2}(\X)$ is a Hilbert space. This property reflects that the underlying geometry looks Riemannian at small scales and can equivalently be characterized via the validity of the following parallelogram identity
\begin{equation}
 |D(f+g)|_2^2 + |D(f-g)|_2^2 = 2|Df|_2^2+ |Dg|_2^2 \qquad \mm\text{-a.e.}, \quad \forall \, f,g \in W^{1,2}(\X).\label{eq:parallelogram}
\end{equation}
This allows to give a notion of scalar product between gradients of  Sobolev functions
\begin{equation}\label{eq:scalare}
    L^1(\mm) \ni \la \nabla f, \nabla g\ra \coloneqq \lim_{\eps \to 0}\frac{|D(f+\eps g)|_2^2-|Df|_2^2}{\eps},\qquad \forall f,g \in W^{1,2}(\X),
\end{equation}
where the limit exists  and is bilinear on its entries, as it can be directly checked using \eqref{eq:parallelogram}. Notice that the symbol $\la \nabla f, \nabla g\ra $ is purely formal. Nevertheless, by introducing the right framework to discuss gradients $\nabla f$, it can be made rigorous (see  \cite{Gigli14}), but we will never need this fact. 

In the infinitesimal Hilbertian class, we can give a notion of a \emph{measure-valued Laplacian} via integration by parts. Since it will be enough for our purposes, we will only consider the compact case.
\begin{definition}[Measure-valued Laplacian, \cite{Gigli12}]
Let $\Xdm$ be a compact infinitesimally Hilbertian metric measure space. We say that $f \in W^{1,2}(\X)$ has a measure-valued Laplacian, and we write $f \in D(\DDelta)$, provided there exists a Radon measure $\mu$ such that
\[ \int g \, \d \mu = - \int \la \nabla f,\nabla g\ra \, \d \mm, \qquad \forall g \in \LIP(\X).\]
In this case the we will denote (the unique) $\mu$ by $\DDelta f$.
\end{definition}
From  the bilinearity of the pointwise inner product we see that $D(\DDelta)$ is a vector space and the assignment $f\mapsto \DDelta f$ is linear.

%We recall from \cite[Proposition 4.11]{Gigli12} a useful calculus rule associated to this object:
%\begin{itemize}
 %   \item[] \textsc{Chain rule} $\phi\circ f \in D(\DDelta)$ and $\DDelta (\phi\circ f) = \phi '\circ f \DDelta f + \phi ''\circ f |Df|^2 \mm.$
%\end{itemize}
%for every $f \in D(\DDelta)\cap \LIP(\X)$ and $\phi \colon f(\X) \to \R$ that is $C^2$. We did not report the weakest assumption on $f,\phi$ and $\X$. Nevertheless, we notice that being $f$ continuous, the expression $\phi '\circ f \DDelta f $ makes sense as a locally finite measure and similarly the second term as $\phi ''\circ f |Df|^2$ is bounded.

\subsubsection{Functions of bounded variations and sets of finite perimeter}
We  introduce the space of functions of  bounded variation and sets finite perimeter following  \cite{ADM14,Miranda03}.
\begin{definition}[BV-functions]\label{def:BV}
A function $f \in L^1(\mm)$ is of bounded variation, and we write $f \in BV(\X)$, provided there exists a sequence of locally Lipschitz functions $f_n \to f$ in $L^1(\mm)$ such that
\[ \limsup_{n\to \infty} \int \lip \ f_n \, \d \mm <\infty.\]
\end{definition}
By localizing this definition, we can define accordingly
\[ |\dD f|(A):= \inf \Big\{ \liminf_{n\to \infty} \int_A \lip \ f_n\, \d \mm \colon f_n \subset \LIP_{loc}(A), f_n \to f \text{ in } L^1(A) \Big\},\]
for every open $A\subset \X$. It turns out (see \cite{ADM14} and also \cite{Miranda03} for locally compact spaces) that the map $A\mapsto |\dD f|(A)$ is the restriction to open sets of a non-negative finite Borel measure called  the \emph{total variation} of $f$, which we will still denote by $|\dD f|$.

For every $f \in \LIP_{bs}(\X)$ we clearly have that $|\dD f|\le \lip \ f\mm$ and in particular that $|\dD f|\ll \mm$. In this case we call $|Df|_1$ the density of $|\dD f|$ with respect to $\mm$.

If we suitably modify Definition \ref{def:BV} for functions in $L^1_{loc}(\mm)$ we can choose $f=\nchi_E$ for any $E\subset \X$ Borel and define:
\begin{definition}[Perimeter and finite perimeter sets]
Let $E$ be Borel and $A$ open subset of $\X$. The perimeter of $E$ in $A$, written ${\rm Per}(E,A)$ is defined as
\[  {\rm Per}(E,A):= \inf \Big\{ \liminf_{n\to \infty} \int_A \lip \ u_n\, \d \mm \colon u_n \subset \LIP_{loc}(A), u_n \to \nchi_E \text{ in } L^1_{loc}(A) \Big\}.\]
Moreover, we say that $E$ is a set of finite perimeter if ${\rm Per}(E,\X)<\infty$.
\end{definition}
Again, (see, e.g. \cite{Ambrosio02,ADM14, Miranda03}), when $E$ has finite perimeter, it holds that $A\mapsto {\rm Per}(E,A)$ is the restriction of a non-negative finite Borel measure to open sets, which we denote by ${\rm Per}(E,\cdot)$. Moreover, as a common convention, when $A=\X$ we simply write ${\rm Per}(E)$ instead of ${\rm Per}(E,\X)$.

For a Borel set $E\subset \X$ of finite measure we also define its Minkowski content as:
\[
\mm^+(E)=\liminf_{\delta \to 0^+} \frac{\mm(E^\delta)-\mm(E)}{\delta},
\]
where $E^\delta\coloneqq\{x \in \X  \ : \ \sfd(x,E)<\delta\}$.  In general we only have $\Per{E}\le \mm^+(E)$.

We recall that the following \emph{coarea formula} is valid after \cite[Proposition 4.2]{Miranda03}.
\begin{theorem}[Coarea formula]
	Let $(\X,\sfd,\mm)$ be a locally compact metric measure space and $f \in BV(\X)$. Then the set $\{f>t\}$ is of finite perimeter for a.e.\ $t\in \R$ and given any Borel function $g:\X \to [0,\infty)$, it holds that
	\begin{equation}\label{eq:coarea}
		\int_{\{s\le u<t\}} g \, \d |\dD f| =\int_s^t \int g\, \d \Per{\{f>t\},\cdot }\, \d t, \qquad \forall s,t \in [0,\infty), \ s<t.
	\end{equation}
\end{theorem}

\subsection{$\cd(K,N)$ and $\RCD(K,N)$ spaces}

\subsubsection{Main definitions and properties}\label{sec:cd rcd}
In this note, as anticipated in the introduction, we will work in the general framework of metric measure spaces $(\X,\sfd,\mm)$ satisfying synthetic Ricci curvature lower bounds. For completeness, we briefly recall the  definition and  the key  properties that we will need.

The first notion of synthetic Ricci lower bounds was given independently in the seminal papers \cite{Lott-Villani09} and \cite{Sturm06-1,Sturm06-2} where the authors introduced the celebrated \emph{curvature dimension condition}. We report here its definition only in finite dimension $N \in [1,\infty)$, given in term of convexity properties of the $N$-\emph{R\'enyi-entropy}  functional $\mathcal{U}_N: \PP_2(\X) \to [-\infty,0]$ defined by
\[ 
\mathcal{U}_N(\mu|\mm) := -\int \rho^{1-\frac{1}{N}}\, \d \mm, \qquad \text{if }\mu= \rho \mm + \mu^s,
\]
where $\mu \in \PP_2(\X)$ and $\mu^s$ is singular with respect to $\mm$. In this note, since optimal transportation plays a minor role, we shall assume the reader to be familiar with Optimal Transport and the Wasserstein Space $(\PP_2(\X),W_2)$ and we refer to \cite{Villani09} for a systematic discussion (see also \cite{guide}).

We start recalling the definition of distortion coefficients. For every $K \in \R, N \in [0,\infty), t \in [0,1]$ set
\[
\sigma^{(t)}_{K,N}(\theta) := \begin{cases} 
\ +\infty, &\text{if } K\theta^2 \ge N\pi^2,\\
\ \frac{\sin(t\theta\sqrt{K/N})}{\sin(\theta\sqrt{K/N)}}, &\text{if } 0<K\theta^2<N\pi^2,\\
\ t, &\text{if } K\theta^2<0\text{ and } N=0 \text{ or if } K\theta^2=0, \\
\ \frac{\sinh(t\theta\sqrt{-K/N})}{\sinh(\theta\sqrt{-K/N)}}, &\text{if } K\theta^2\le 0 \text{ and }N>0.\\
\end{cases}
\]
Set also, for $N>1$, $\tau^{(t)}_{K,N}(\theta) := t^{\frac{1}{N}}\sigma^{(t)}_{K,N-1}(\theta)^{1-\frac{1}{N}}$ while $\tau^{(t)}_{K,1}(\theta) =t$ if $K\le 0$ and $\tau^{(t)}_{K,1}(\theta)=\infty$ if $K>0$.

\begin{definition}[$\cd(K,N)$-spaces]\label{def:CDKN}
Let $K \in\R$ and $N\in [1,\infty)$. A metric measure space $\Xdm$ satisfies the \emph{curvature dimension condition} $\cd(K,N)$ if, for every $\mu_0,\mu_1 \in \PP_2(\X)$ absolutely continuous with bounded supports, there exists a dynamical optimal transference plan $\pi \in \PP({\rm Geo}(\X))$ between $\mu_0,\mu_1$ so that: for every $t\in [0,1]$ and $N'\ge N$, we have $\mu_t:=({\sf e}_t)_\sharp \pi =\rho_t\mm$ and
\begin{equation}
\mathcal{U}_{N'}(\mu_t|\mm) \le - \int \Big( \tau^{(1-t)}_{K,N'}(\sfd(\gamma_1,\gamma_0))\rho_0(\gamma_0)^{-\frac{1}{N}} + \tau^{(t)}_{K,N'}(\sfd(\gamma_1,\gamma_0))\rho_1(\gamma_1))^{-\frac{1}{N}}\Big) \, \d \pi(\gamma).
\label{eq:CDqKN}
\end{equation}
\end{definition}

We recall the also the notion of one-dimensional model space for the $\cd(N-1,N)$ condition:

\begin{definition}[One dimensional model space]\label{def:model space}
    For every  $N>1$ we define  $I_N\coloneqq([0,\pi],|.|,\mm_N)$, where $|.|$ is the Euclidean distance restricted on $[0,\pi]$ and
\[\mm_{N}:= \tfrac{1}{c_{N}}\sin^{N-1}\Leb 1\restr{[0,\pi]}, \]
with $c_{N}:= \int_{[0,\pi]}\sin(t)^{N-1}\, \d t $.
\end{definition}

To encode a more ``Riemannian" behavior of the space, and to rule out Finsler spaces which are allowed by the $\cd$ condition, it was introduced in \cite{AGS14} the so-called $\RCD$ condition in the infinite dimensional case (see also \cite{GMS15} for the case of $\sigma$-finite reference measure). In this note however we will only work  in finite dimensional $\RCD$-spaces introduced in \cite{Gigli12}.
\begin{definition}[$\RCD(K,N)$-spaces]\label{def:RCDKN}
Let $K \in\R$ and $N\in [1,\infty)$. A metric measure space $\Xdm$ is an $\RCD(K,N)$-space, provided it is an infinitesimal Hilbertian $\cd(K,N)$-space.
\end{definition}
\begin{remark}\label{rmk:p independence CD}
Spaces satisfying the $\cd(K,N)$ (and thus also the $\RCD(K,N)$) condition, support a $(1,1)$-local Poincar\'e inequality  (see \cite{Rajala12}) and by the Bishop-Gromov inequality below they are  locally-doubling, therefore from the results in \cite{Cheeger99} we know that the minimal weak upper gradient is independent on the exponent $p$ (see also \cite{GH14}). For this reason, to lighten the notation, in this setting we will simply write $|Df|$ for $f \in W^{1,p}(\X)$ and call it simply minimal weak upper gradient of $f$. \fr
\end{remark}

We start by recalling some useful properties about these spaces that are going to be used in the sequel. 

On $\cd(K,N)$ spaces the Bishop-Gromov inequality holds (see \cite{Sturm06-2}):
\begin{equation}
\frac{\mm(B_R(x))}{v_{K,N}(R)}\le \frac{\mm(B_r(x))}{v_{K,N}(r)}, \quad \text{for any $0<r<R\le \pi \sqrt{\frac{N-1}{K^+}}$ and any $x \in \X$},\label{eq:Bishop}
\end{equation}
where  the quantities $v_{K,N}(r)$, $N\in[1,\infty)$ $K \in \R$ are defined as
\[
v_{K,N}(r)\coloneqq\sigma_{N-1}\int_0^r |s_{K,N}(t)|^{N-1}\, \d t,
\]
and  $s_{K,N}(t)$ is defined as $\sin\Big ( t\sqrt{\frac{K}{N-1}} \Big),$ if $K>0$, $ \sinh\Big ( t\sqrt{\frac{|K|}{N-1}} \Big)$, if  $K<0$ and $t$ if $K=0$. In particular $\cd(K,N)$ spaces are uniformly locally doubling and thus proper, i.e. closed and bounded sets are also compact. We also note that in the case $K=0$ this implies that the limit 
\[
{\sf AVR(\X)}\coloneqq\lim_{r\to +\infty}\frac{\mm(B_r(x))}{\omega_Nr^N}
\]
exists finite and does not depend on the point $x \in \X$. We call the quantity ${\sf AVR(\X)}$ \emph{asymptotic volume ratio} of $\X$ and if ${\sf AVR(\X)}>0$ we say that $\X$ has \emph{Euclidean-volume growth}. A key role in the note will be played by the following quantities:
\[
\theta_{N,r}(x)\coloneqq\frac{\mm(B_r(x))}{\omega_N r^N}, \quad 
\theta_{N}(x)\coloneqq\lim_{r \to 0^+} \theta_{N,r}(x), \qquad \forall \, r>0, x \in \X.
\]
Observe that the above limit exists thanks to the Bishop-Gromov inequality and  the fact that $\lim_{r \to 0^+} \frac{\omega_N r^N}{v_{K,N}(r)}=1$ for every $K\in \R$, $N \in[1,\infty),$ which in particular grants that 
\begin{equation}\label{eq:theta equivalent}
    \theta_{N}(x)=\lim_{r\to 0} \frac{\mm(B_r(x))}{v_{K,N}(r)}=\sup_{r>0} \frac{\mm(B_r(x))}{v_{K,N}(r)}.
\end{equation}
This and the fact that $\mm(\partial B_r(x))=0$ for every $r>0$ and $x\in\X$ (which follows from the Bishop-Gromov inequality), implies that $\theta_N(x)$ is a \emph{lower-semicontinuous} function of $x.$ Therefore, when $\X$ is compact, there exists  $\min_{x \in \X}\theta_N(x)$.

Next we recall the Brunn-Minkowski inequality.
\begin{theorem}[\cite{Sturm06-2}]
	Let $\Xdm$ be a $\cd(K,N)$ space with $N\in[1,\infty)$, $K \in \R$. For any couple of Borel sets $A_0,A_1 \subset \X$ it holds that 
	\begin{equation}\label{eq:brunn}
		\mm(A_t)^\frac1N\ge\sigma_{K,N}^{(1-t)}(\theta)\mm(A_0)^\frac1N+ \sigma_{K,N}^{(t)}(\theta)\mm(A_1)^\frac1N, \qquad \forall t \in[0,1],
	\end{equation}
	where $A_t\coloneqq \{\gamma_t \ : \ \gamma \text{ geodesic such that } \gamma_0\in A_0,\, \gamma_1\in A_1 \}$ and
$$
\theta:= 
\begin{cases}
	{\inf}_{(x_{0},x_{1}) \in A_{0} \times A_{1}} \sfd(x_{0},x_{1}), & \textrm{if } K \geq 0, \\
{\sup}_{(x_{0},x_{1}) \in A_{0} \times A_{1}} \sfd(x_{0},x_{1}), & \textrm{if } K < 0,
\end{cases}
$$
\end{theorem}
We remark that \eqref{eq:brunn} is actually weaker than the statement appearing in \cite{Sturm06-2} and it holds for the (a priori) larger class of ${\rm CD}^*(K,N)$ spaces (see \cite{BS10}).

 We report the Bonnet-Myers diameter-comparison theorem for $\cd$-spaces from \cite{Sturm06-2}:
\begin{equation}
    \text{$\Xdm$  is a $\cd(K,N)$ space}, \text{ for some }K>0 \quad \Rightarrow \quad \text{diam}(\X) \le \pi\sqrt{\tfrac{N-1}K},\label{eq:BonnetMyers}
\end{equation}

The Lichnerowitz $2$-spectral gap inequality  is valid also in the $\cd$-setting. To state it we recall the notion of  first non-trivial Neumann eigenvalue of the Laplacian (or $2$-spectral gap) in metric measure spaces.
\begin{definition}
	Let $(\X,\sfd,\mm)$ be a metric measure space with finite measure. We define the first non trivial $2$-eigenvalue $ \lambda^{1,2}(\X)$ as the non-negative number given by
	\begin{equation}\label{eq:spectral gap}
		 \lambda^{1,2}(\X)\coloneqq \inf \left \{ \frac{\int |D f|_2^2\, \d\mm}{\int |f|^2\, \d\mm}  \ : \ f \in \LIP(\X)\cap L^2(\mm),\, f \neq 0,\, \int f\, \d \mm=0   \right\}.
	\end{equation}
\end{definition}
Clearly, in light of \cite{AGS13}, in the above definition one can equivalently take the infimum among all $f \in W^{1,2}(\X)$. In the sequel will use this fact without further notice. 

Then the spectral-gap inequality as proven in \cite{LV09} (see also \cite{JZ16}) says that:
\[ \lambda^{1,2}(\X) \ge N, \qquad \text{for every } \cd(N-1,N)\text{-space $\X$},\]
with $N$ ranging in $(1,\infty)$.
%We remark that in \cite{CM17}, the above has been generalized to every $p \in (1,\infty)$ taking into account also an upper diameter bound.

We conclude this part  recalling some rigidity and stability statements for $\RCD(K,N)$ spaces and to this goal we need to define the notion of spherical suspension over a metric measure space. For any $N \in [1,\infty)$ the $N$-spherical suspension over a metric measure space $(\Z,\mm_\Z,\sfd_\Z)$ is defined  to be the space
$([0,\pi]\times_{\sin}^N \Z)\coloneqq \Z\times \left[0,\pi \right]/(\Z\times \left\{0,\pi\right\})$ endowed with the following distance and measure
\begin{align*}
	&\sfd((t,z),(s,z'))\coloneqq \cos^{-1}\big(\cos(s)\cos(t)+\sin(s)\sin(t)\cos\left({\sfd}_\Z(z,z')\wedge\pi\right)\big),\\
	&\mm\coloneqq \sin^{N-1}(t) \d t\,\otimes\, \mm_\Z.
\end{align*}
It turns out that the ${\rm RCD}$ condition is stable under the action of taking spherical suspensions, more precisely it has been proven in \cite{Ketterer14} that

\begin{equation}\label{eq:ketterer}
	\begin{split}
	&[0,\pi]\times_{\sin}^N \Z, N\ge 2\text{ is a } \RCD(N-1,N) \text{ space if and only if}\\
	&\diam(\Z)\le \pi \text{ and $\Z$ is an ${\rm RCD}(N-2,N-1)$ space,}
	\end{split}
\end{equation}

We can now recall the two main rigidity statements that we will use in the note: the maximal diameter theorem and the Obata theorem for $\RCD(K,N)$ spaces:
\begin{theorem}[\cite{Ketterer15}]\label{thm:obata}
    Let $\Xdm$ be an $\RCD(N-1,N)$ space with  and $N\in[2,\infty)$ and suppose that $\diam(\X)=\pi$. Then $\Xdm$ is isomorphic to a spherical suspension, i.e. there exists an $\RCD(N-2,N-1)$ space $(\Z,\sfd_\Z,\mm_\Z)$ with $\diam(\Z)\le \pi$ satisfying $\X \simeq [0,\pi ]\times_{\sin}^N \Z$.
\end{theorem}
\begin{theorem}[\cite{Ketterer14}]\label{thm:diameter rigidity}
    Let $\Xdm$ be an $\RCD(N-1,N)$ space with  and $N\in[2,\infty)$ and suppose that $\lambda^{1,2}(\X)=N$. Then $\Xdm$ is isomorphic to a spherical suspension, i.e. there exists an $\RCD(N-2,N-1)$ space $(\Z,\sfd_\Z,\mm_\Z)$ with $\diam(\Z)\le \pi$ satisfying $\X \simeq [0,\pi ]\times_{\sin}^N \Z$.
\end{theorem}

We end this part by recalling the definition of ``non-collapsed'' $\RCD$-spaces, which extends the notion of non-collapsed Ricci-limits introduced in \cite{Cheeger-Colding97I}.
\begin{definition}[\cite{GDP17}]\label{def:ncRCD}
 We say that $\Xdm$ is a \emph{non-collapsed} $\RCD(K,N)$ space, for some $K \in \R, N \in \N$, provided it is $\RCD(K,N)$ and $\mm = \HH^N$, where $\HH^N$ is the $N$-dimensional Hausdorff measure.
\end{definition}
This class of spaces enjoys extra regularity with respect to the general $\RCD$-class and are a suitable setting to derive the \emph{topological rigidity} results of this note. Here we just mention that   if $\theta_N$ is finite  $\mm$-a.e.\ (or equivalently if $\mm\ll \Haus{N}$), then up to a constant multiplicative factor, $\mm$ equals $\HH^N$ and the space is non-collapsed. This has been proved first in \cite{H19} for compact spaces and then in \cite{BGHZ21} in the general case solving a conjecture of \cite{GDP17} (see also \cite{Honda20Conj} for an account on further conjectures around this topic).

\subsubsection{Sobolev-Poincar\'e inequality on $\cd(K,N)$ spaces}
A well-established fact which  goes back to the seminal work \cite{HK00}, is that a $(1,p)$-Poincar\'e inequality on a doubling metric measure space, improves to a $(q,p)$-Poincar\'e inequality with $q>1.$ On $\cd(K,N)$ spaces this translates in the following result.  
\begin{theorem}[$(p^*,p)$-Poincar\'e inequality]\label{thm:improved poincaret}
	Let $\Xdm$ be a $\cd(K,N)$ space for some $N\in(1,\infty)$, $K \in \R$. Fix also $p \in(1,N)$ and $r_0>0$. Then, for every $B_r(x)\subset \X$ with $r\le r_0$ it holds
	\begin{equation}\label{eq:improved poincaret}
		\Big(\fint_{B_r(x)} |u-u_{B_r(x)}|^{p^*} \, \d \mm \Big)^\frac{1}{p^*} \le C(K,N,p,r_0)r \Big(\fint_{B_{2r}(x)} |D u|^{p}\, \d \mm \Big)^\frac{1}{p}, \qquad \forall \, u \in \LIP(\X),
	\end{equation}
	where $p^*\coloneqq pN/(N-p)$ and $u_{B_r(x)}\coloneqq\int_{B_r(x)}u\, \d \mm.$
\end{theorem}
\begin{proof}
	From  \cite{Rajala12} we have that $\X$ supports a strong $(1,1)$-Poincar\'e inequality, in particular it also supports a strong $(1,p)$-Poincar\'e inequality for every $p\in[1,\infty),$ by H\"older inequality. Moreover, for every $x_0 \in \X$, $r\le r_0$ and $x \in B_{r_0}(x_0)$, from the Bishop-Gromov inequality \eqref{eq:Bishop} it holds that
	\begin{equation*}\label{eq:measure-dimension}
		\frac{\mm(B_r(x))}{\mm(B_{r_0}(x_0))}\ge  C(K,N,r_0) \Big(\frac{r}{r_0}\Big)^N. 
	\end{equation*}
	Then \eqref{eq:improved poincaret} follows from \cite[Theorem 5.1]{HK00} (see also \cite[Theorem 4.21]{BB13}).
\end{proof}

We end this part recalling the sharp Sobolev-inequality on the $N$ model space $I_N$ (see Def. \ref{def:model space}) for $N \in (2,\infty)$ (see e.g. \cite{Ledoux00}):
\begin{equation}\label{eq:sobolev model}
    \|u\|_{L^{q}(\mm_N)}^2\le \frac{q-2}{N}\,  \||Du|\|_{L^2(\mm_N)}^2 +  \|u\|_{L^2(\mm_N)}^2 , \quad \forall u \in W^{1,2}([0,\pi],|.|,\mm_N),
\end{equation}
for every $ q \in (2,2^*]$, with $2^*=2N/(N-2)$. 
%\begin{remark}\label{rmk:beta 1}
%Notice that \eqref{eq:Sobolev embedding pre} implies the validity of \eqref{eq:general sobolev} for every $p \in(1,N)$ and $q \in(p,q^*)$. Indeed this follows applying the H\"older inequality on the left hand side and elevating the inequality to the $p$-the power.

%Even more, using the elementary inequality  $(a+b)^p\le C_\eps a^p+(1+\eps)b^p$, for all $a,b\ge 0$, shows that $\beta_{q,p}(\X)=\mm(\X)^{-p/N}$ (where   $\beta_{q,p}(\X)$ is as in \eqref{eq:alfabeta}).
%\end{remark}

\subsubsection{Convergence and compactness under mGH-convergence}\label{sec:mgh}
We recall here the notion of \emph{pointed-measure Gromov Hausdorff convergence} (pmGH convergence for short). Let us say that the definition we will adopt is not the classical one (see e.g. \cite{BBI01,Gromov07}), but it is equivalent in the case of a sequence of  uniformly locally doubling metric measure spaces, thanks to the results in \cite{GMS15}. It will be convenient to consider in this section the set $\bar \N := \N\cup \{\infty\}$. Recall also that a pointed metric measure space is a quadruple $(\X,\sfd,\mm,x)$ consisting of a metric measure space $\Xdm$ and a point $x \in \X$. 
\begin{definition}[Pointed measure Gromov-Hausdorff convergence]\label{def:pmGH}
	We say that the sequence  $(\X_n,\sfd_n,\mm_n,x_n)$,  $n \in \N$, of pointed metric measure spaces, \emph{pointed measure Gromov-Hausdorff}-converges (pmGH-converges in short) to    $(\X_\infty,\sfd_\infty,\mm_\infty,x_\infty)$, if 
	there exist isometric embeddings $\iota_n: \X_n \to (\Z,\sfd_\Z)$, $n \in \bar \N$,  into a common metric space $(\Z,\sfd_\Z)$ such that
	\[
	({\iota_n})_\sharp\mm_n \rightharpoonup 	({\iota_\infty})_\sharp\mm_\infty \text{ in duality with $C_{bs}(\Z)$ and } \iota_n(x_n)\to \iota_\infty(x_\infty).\]
\end{definition}
In the case of a  sequence of   uniformly locally doubling spaces (as  in the case of $\cd(K,N)$-spaces for fixed $K\in R,N<\infty$) we can also take $(\Z,\sfd_\Z)$ to be proper. Moreover, again for a class of uniformly locally doubling spaces, in \cite{GMS15} it is proven that the pmGH-convergence is metrizable with a distance which we call $\sfd_{pmGH}$. 

It will be also convenient to adopt, thanks to Definition \ref{def:pmGH}, the so-called \emph{extrinsic approach}, where the spaces $\X_n$ are identified as subsets of a common \emph{proper} metric space $(\Z,\sfd_\Z)$,  $X_n \subset \Z$,  $\supp(\mm_n)=\X_n$,  $\sfd_\Z \restr{\X_n\times \X_n} = \sfd_n$ for all $n \in \bar \N$, and $\sfd_\Z(x_n,x_\infty)\to 0,$ $\mm_n\weakto \mm_\infty$ in duality with $C_{bs}(\Z)$. Any such space $(\Z,\sfd_\Z)$ (together with an the identification of $\X_n\subset \Z$) is called \emph{realization of the convergence} and (in the case of geodesic uniformly locally doubling spaces) can be taken so that $\sfd^{\Z}_H(B^{\X_n}_R(x_n),B^{\X_\infty}_R(x_\infty))\to 0$ for every $R>0,$ where $\sfd_H^\Z$ is the Hausdorff distance in $\Z$. To avoid confusion when dealing with this identification, we shall sometimes write $B^{\X_n}_r(x)$ with $x \in \X_n$, $r>0$, to denote the set $B^\Z_r(x)\cap \X_n$.

After the works in \cite{Sturm06-1,Sturm06-2,Lott-Villani09,AGS14,GMS15} and thanks to the Gromov's precompactness theorem \cite{Gromov07} we have the following precompactness result.
\begin{theorem}\label{thm:mGHcompact}
    Let $(\X_n,\sfd_n,\mm_n,x_n)$ be a sequence of  pointed $\cd(K_n,N_n)$ (resp. $\RCD(K_n,N_n)$) spaces, $n \in \bar \N$, with $\mm(B_1(x_n))\in [v^{-1},v]$, for $v >1$ and $K_n\to K\in \R, N_n\to N\in [1,\infty)$. Then, there exists a subsequence $(n_k)$ and a pointed $\cd(K,N)$ (resp. $\RCD(K,N)$) space $(\X_\infty,\sfd_\infty,\mm_\infty,x_\infty)$ satisfying
    \[\lim_{k\to \infty}\sfd_{pmGH}\big( (\X_{n_k},\sfd_{n_k},\mm_{n_k},x_{n_k}),(\X_\infty,\sfd_\infty,\mm_\infty,x_\infty)\big) =0. \]
\end{theorem}
We will be frequently consider  the case of compact (with uniformly bounded diameter) metric measure spaces which is the natural setting for the Sobolev embedding of this note, for which we can reduce the above convergence to the so-called \emph{measure Gromov Hausdorff convergence}, mGH-convergence for short, where we simply ignore the convergence of the base points. Also in this case, on every class of uniformly doubling metric measure spaces with uniformly bounded diameter, the mGH-convergence can be metrized by a distance that we denote by $\sfd_{mGH}.$ The extrinsic approach applies verbatim as well, with the exception that the common ambient space $\Z$ can be also taken to be compact.

We now recall some stability and convergence results of functions along pmGH-convergence. For additional details and analogous results we refer to  \cite{Honda15,GMS15, AH17}. For brevity reasons in what follows we fix a sequence of pointed ${\rm CD}(K,N)$ spaces $(\X_n,\sfd_n,\mm_n,x_n)$, for $n \in \bar {\mathbb{N}}$, so that $\X_n \overset{pmGH}{\to} \X_\infty$.

\begin{definition}
 Let $p \in (1,\infty)$, we say that
\begin{itemize}
\item[$(i)$] $f_n\in L^p(\mm_n)$ \emph{converges $L^p$-weak} to $f_\infty\in L^p(\mm_\infty)$, provided $\sup_{n \in \N}\|f_n\|_{L^p(\mm_n)}<\infty$ and $f_n\mm_n \weakto f_\infty\mm_\infty$ in $C_{bs}(\Z)$,
\item[$(ii)$] $f_n\in L^p(\mm_n)$ \emph{converges $L^p$-strong} to $f_\infty\in L^p(\mm_\infty)$, provided it converges $L^p$-weak and $\limsup_n \|f_n\|_{L^p(\mm_n)} \le  \|f_\infty\|_{L^p(\mm_\infty)}$,
\item[$(iii)$] $f_n \in W^{1,2}(\X_n)$ \emph{converges $W^{1,2}$-weak} to $f_\infty \in W^{1,2}(\X)$ provided it converges $L^2$-weak and $\sup_{n \in \N} \| |D f_n|\|_{L^2(\mm_n)}<\infty$,
    \item[$(iv)$] $f_n \in W^{1,2}(\X_n)$ \emph{converges $W^{1,2}$-strong} to $f_\infty \in W^{1,2}(\X)$ provided it converges $L^2$-strong and $\| |D f_n|\|_{L^2(\mm_n)} \to \||Df_\infty|\|_{L^2(\mm_\infty)}$.
%\item[$(v)$] We say that $f_n \in L^1(\mm_n)$ \emph{converges $L^1$-strong} to $f_\infty \in L^1(\mm_\infty)$, provided $\sigma \circ f_n \mm_n\weakto \sigma \circ f_\infty \mm_\infty$ in $C_{bs}(\Z)$ and $\lim_n \|f_n\|_{L^1(\mm_n)} = \|f_\infty \|_{L^1(\mm_\infty)}$, where $\sigma(z)=\emph{sgn}(z)\sqrt{z}$.
\end{itemize}
\end{definition}
Moreover, we say that $f_n$ is uniformly bounded in $L^p$  if $\sup_n\|f_n\|_{L^p(\mm_n)}<\infty$. In the following statement we collect a list of useful properties of $L^p$-convergence.
\begin{proposition}[Properties of $L^p$-convergence]\label{prop:lp prop}  For all $p \in (1,\infty)$, it holds
    \begin{enumerate}
        \item[$(i)$] If $f_n$ converges $L^p$-strong to $f_\infty$, then $\phi(f_n)$ converges $L^p$-strong to $\phi(f_\infty)$ for every $\phi\in\LIP(\R)$ with $\phi(0)=0,$
        \item[$(ii)$] If $f_n$ (resp. $g_n)$ converges $L^p$-strong to $f_\infty$ (resp. $g_\infty $), then $f_n+g_n$ converges $L^p$-strong to $f_\infty+g_\infty$,
        \item[$(iii)$] if $f_n$ converges $L^p$-weak to $f$, then $\|f_\infty\|_{L^p(\mm_\infty)}\le \liminf_n \|f_n\|_{L^p(\mm_n)}$,
        \item[$(iv)$] suppose that  $\sup_n \|f_n\|_{L^p(\mm_n)}<+\infty$, then up to  a subsequence $f_{n}$ converges $L^p$-weak to some $f_\infty \in L^p(\mm_\infty),$
        \item[$(v)$] If $f_n$ converges $L^p$-strong (resp. $L^p$-weak) to $f_\infty$, then $\phi f_n$ converges $L^p$-strong (resp. $L^p$-weak) to $\phi f_\infty$, for all $\phi \in C_b(\Z)$,
        \item [$(vi)$] for every $f \in L^p(\mm_\infty)$ there exists a sequence $f_n \in L^p(\mm_n)$ converging $L^p$-strong to $f$,
        \item [$(vii)$] if $f_n$ are non-negative and converge in $L^p$-strong to $f$, then for every $q\in(1,\infty),$ $f_n^{p/q}$ converge $L^q$-strong to $f^{p/q}$, 
        \item[$(viii)$] Fix $p,q \in (1,\infty]$ so that $p< q.$ If the sequence $(f_n)$ is uniformly bounded in $L^q$ and converges $L^p$-strong to $f_\infty$, then it converges also $L^r$-strong to $f_\infty$ for every $r \in[p,q),$
    \end{enumerate}
\end{proposition}
\begin{proof}
    For the proof of the items $(i)-(v)$ we refer to \cite[Prop. 3.3]{AH17}. $(vi)$ can instead be found in \cite{GMS15} (see also \cite{Honda15}). 
    $(vii)$ follows immediately from the characterization of $L^p$-strong convergence via convergence of graph (see e.g. \cite[Remark 3.2]{AH17}).
    For $(viii)$, the case $q=\infty$ follows immediately from item $(i)$ (see also \cite[e) of Prop. 3.3 ]{AH17}), hence we can assume $q<+\infty.$
Fix $r \in [p,q).$ Clearly from the H\"older inequality $f_n$ is uniformly bounded in $L^r$, hence by definition $f_n$ converges $L^r$-weakly to $f_\infty$.   Moreover from item $(iii)$ we known that $f_\infty \in L^r(\mm_\infty),$ therefore  by truncation and diagonalization we can suppose that $f \in L^\infty(\mm_\infty).$ From $(vi)$ then there exists a sequence $g_n \in L^r(\mm_n)$ converging to $f_\infty$ in $L^r$-strong and by item $i)$ we can also assume that $g_n$ are uniformly bounded in $L^\infty.$ Then, from $(viii)$ in the case $q=\infty$ we have that $g_n$ converge also in $L^p$-strong to $f_\infty.$  Then by $(ii)$ we have that $g_n-f_n$ converges to 0 in $L^p$-strong and in particular $\|f_n-g_n\|_{L^p(\mm_n)}\to 0.$ Finally by the H\"older inequality (since $f_n,g_n$ are both uniformly bounded in $L^q$) we have that $\|f_n-g_n\|_{L^r(\mm_n)}\to 0$. In particular $\lim_n \|f_n\|_{L^r(\mm_n)}=\lim_n \|g_n\|_{L^r(\mm_n)}=\|f_\infty\|_{L^r(\mm_\infty)}$, which concludes the proof.
\end{proof}
We now pass to some convergence and stability results related to Sobolev spaces. We start with the following generalized version of the compact embedding of $W^{1,2} \hookrightarrow L^2$ (reported here specifically for compact metric measure spaces):
\begin{proposition}[\cite{GMS15}]\label{prop:W12compact}
Suppose that $\X_n$, $n \in \bar \N$ are compact and assume that $(f_n) \in W^{1,2}(\X_n)$ are uniformly bounded in $W^{1,2}$, i.e. $\sup_n \|f_n\|_{W^{1,2}(\X_n)}<+\infty$. Then $(f_n)$ has a $L^2$-strongly convergent subsequence. 
\end{proposition}
We recall the $\Gamma$-convergences of the $2$-Cheeger energies  proven in \cite{GMS15}:
\begin{itemize}
    \item[$\circ$] $\Gamma$-$\liminf$: for every $f_n \in L^2(\mm_n)$ $L^2$-strong converging to $f_\infty \in L^2(\mm_\infty)$, it holds
    \begin{equation}
         \int |D f_\infty|^2\, \d \mm_\infty \le \liminf_{n\to \infty}\int|D f_n|^2\, \d \mm_n; \label{eq:Gammaliminf}
    \end{equation}
    \item[$\circ$]  $\Gamma$-$\limsup$: for every $f_\infty \in L^2(\mm_\infty)$, there exists a sequence $f_n \in L^2(\mm_n)$ converging $L^2$-strong to $f_\infty$ so that
    \begin{equation}\limsup_{n\to \infty} \int |D f_n|^2\, \d \mm_n \le \int |D f_\infty |^2\, \d \mm_\infty.\label{eq:Gammalimsup}
    \end{equation}
\end{itemize}
We will also need the $\Gamma$-$\limsup$ inequality also for the $p$-Cheeger energies as proved in \cite[Theorem 8.1]{AH17}:  for every $p\in(1,\infty)$ and every $f_\infty \in L^p(\mm_\infty)$, there exists $f_n \in L^p(\mm_n)$ converging $L^p$-strong to $f_\infty$ so that
    \[ \limsup_{n\to \infty} \int |D f_n|^p\, \d \mm_n \le \int |D f_\infty |^p\, \d \mm_\infty.\]
The above is stated in \cite{AH17} only for a sequence of $\RCD(K,\infty)$ spaces, but it easily seen that the proof works without modification also in the case of $\cd(K,\infty)$ spaces.

We end this part recalling a well known continuity result of the spectral gap (see \cite{GMS15} and \cite{AHP18}): if $\X_n$,  $n \in \bar{\N}$, are all compact it holds 
\begin{equation}
   \lambda^{1,2}(\X_\infty) =\lim_{n\to\infty}  \lambda^{1,2}(\X_n).  \label{eq:lambdalsc}
\end{equation}
%If  $\liminf_{n\to\infty}  \lambda^{1,2}(\X_n) = \infty$, the claim is trivially satisfied, so we are left to handle the case when the right hand most side is finite and suppose without loss of generality that $ \lambda^{1,2}(\X_n)$ is uniformly bounded (possibly along a not relabelled subsequence). To prove this, consider for every $n \in \N$, functions $u_n \in W^{1,2}(\X_n)$ with $\int u_n\, \d \mm_n =0$ so that $\int |D u_n|^2\, \d \mm_n \le   \lambda^{1,2}(\X_n) +\tfrac1n $.  By Proposition \ref{prop:W12compact}, there exists $u_\infty \in W^{1,2}(\X_\infty)$, $L^2$-strong limit up to a not relabelled subsequence. Thus, $\|u_\infty\|_{L^2(\mm_\infty)}=1$ and, by lower mGH-semicontinuity of the Cheeger energy, we have
%\[ \int |D u_\infty|^2\, \d \mm_\infty \le \liminf_{n\to \infty} \int |D u_n|^2\, \d \mm_n \le  \liminf_{n\to \infty} \lambda^{1,2}(\X_n).\]
%But now, from the weak convergence $u_n\mm_n\weakto u_\infty\mm_\infty$, we clearly deduce that $u_\infty$ has zero mean so that taking the infimum on the left hand most side gives the desired conclusion.
We mention that the continuity of the spectral gap was previously obtained in the setting of Ricci-limit spaces by Cheeger and Colding \cite{Cheeger-Colding97I}.

\subsection{P\'olya-Szeg\H{o} inequality}\label{sec:polya}
The P\'olya-Szeg\H{o} inequality, namely the fact that the Dirichlet energy decreases under \emph{decreasing rearrangements}, dates back to Faber and Krahn and was successively formalized in \cite{PolSzeg51}. Later, in \cite{BerMey82}, this collection of ideas was brought to the context of manifolds with Ricci lower bounds to achieve applications concerning the rigidity of the $2$-spectral gap. Concerning the topic of this manuscript, the said inequality has revealed effective in \cite{Said83} in the proof of Theorem \ref{thm:Aopt upper}.

In this part we recall the P\'olya-Szeg\H{o} inequality  for essentially nonbranching ${\rm CD}(K,N)$ spaces proven in \cite{MS20}. We will also collect some additional technical results and definitions from \cite{MS20} that will be  used in Section \ref{sec:euclidean polya} to prove a  Euclidean-variant of this inequality.

\begin{definition}[Distribution function]\label{def:distributionfunction}
	Let $\Xdm$ be a compact metric measure space, $\Omega\subseteq \X$  an open set with $\mm(\Omega)<+\infty$ and $u:\Omega\to[0,+\infty)$ a non-negative Borel function. We define  $\mu:[0,+\infty)\to[0,\mm(\Omega)]$, the distribution function of $u$, as 
	\begin{equation}\label{eq:defdistr}
		\mu(t):=\mm(\{u> t\}).
	\end{equation}
\end{definition}
For $u$ and $\mu$ as above, we let $u^{\#}$ be the generalized inverse of $\mu$, defined by
\begin{equation*}
	u^{\#}(s):=
	\begin{cases*}
		{\rm ess}\sup u & \text{if $s=0$},\\
		\inf\left\lbrace t:\mu(t)<s \right\rbrace &\text{if $s>0$}. 
	\end{cases*}
\end{equation*}
It can be checked that $u^\#$ is non-increasing and left-continuous.

Then, given $\Omega\subseteq X$ an open set and $u:\Omega\to[0,+\infty)$ a non-negative Borel function, we define the \emph{monotone rearrangement} into $I_N=([0,\pi],|.|,\mm_{N})$ (see Definition \ref{def:model space}) as follows: first, we consider $r >0$ so that $\mm(\Omega) = \mm_{N}([0,r])$ and define $\Omega^* := [0,r]$, then we define the monotone rearrangement  function $u^*_{N} : \Omega^* \to \R^+$ as
\[ u_{N}^*(x) := u^\#(\mm_{N}([0,x])),\qquad \forall x \in [0,r].\]
In the sequel, whenever $u$ and $\Omega$ are fixed, $\Omega^*$ and $u^*_{N}$ will be implicitly defined as above.

\begin{theorem}[P\'olya-Szeg\H{o} inequality, \cite{MS20}]\label{thm:KNpolya}
	Let $\Xdm$ be an essentially non braching $\cd(N-1,N)$ space for some $N \in (1,\infty)$ and $\Omega\subseteq \X$ be  open. Then, for every $p \in (1,\infty)$, the monotone rearrangement in $I_N$ maps $L^p(\Omega)$ (resp. $W^{1,p}_0(\Omega)$) into $L^p(\Omega^*)$ (resp. $W^{1,p}(\Omega^*)$) and satisfies:
	\begin{equation}\label{eq:equalLpnorm1}
		\|{u}\|_{L^p(\Omega)}=\|{u_{N}^*}\|_{L^p(\Omega^*)},\qquad \forall u \in L^p(\Omega)
	\end{equation}
	\begin{equation} \int_\Omega |D u|^p\, \d \mm \ge \int_{\Omega^*} |D u_{N}^*|^p\, \d \mm_{N}, \qquad \forall u \in W^{1,p}_0(\Omega). \label{eq:PolyaSzego}
	\end{equation}
\end{theorem}

We will also need the  following rigidity  of the P\'olya-Szeg\H{o} inequality proven in \cite[Theorem 5.4]{MS20}. 
\begin{theorem}\label{thm:Polyarigidity}
	Let $\Xdm$ be an $\RCD(N-1,N)$ space for some $N\in[2,\infty)$ with $\mm(\X)=1$ and $p \in (1,\infty)$. Let $\Omega\subset \X$ be an open set and assume that there exists a non-negative and non-constant function $u \in W^{1,p}_0(\Omega)$ achieving equality in \eqref{eq:PolyaSzego}.
	
	Then $\Xdm$ is  isomorphic to a spherical suspension, i.e. there exists an $\RCD(N-2,N-1)$ space $(\Z,\sfd_\Z,\mm_\Z)$ with $\mm_\Z(\Z)=1$ so that $\X \simeq [0,\pi]\times^{N}_{\sin} \Z$.
\end{theorem}
\begin{remark}\label{rmk:omegax}\rm 
	Observe that in Theorem \ref{thm:Polyarigidity} we did not assume that $\mm(\Omega)<1$, assumption that is actually present in Theorem 5.4 of \cite{MS20}. This is intentional, since we will  need to apply Theorem \ref{thm:Polyarigidity} precisely in the case $\Omega=\X$. This is possible since the arguments in \cite{MS20} work also in the case $\Omega=\X$  without modification. The only part where the argument does not cover explicitly the case $\Omega=\X$ is the proof of the approximation Lemma 3.6 in \cite{MS20}, which however can be easily adapted (see Lemma \ref{lem:nonvanishing gradients} below).\fr 
\end{remark}

The following technical result will be needed in Section \ref{sec:euclidean polya}. We include a sketch of the argument in the case $\Omega=\X$, to further justify the validity of Theorem \ref{thm:Polyarigidity} also in this case (see the above Remark). 
\begin{lemma}[Approximation with non-vanishing gradients]\label{lem:nonvanishing gradients}
  	Let $\Xdm$ be a $\cd(K,N)$ metric measure space with $N<+\infty$, and let $\Omega\subset \X$ be open with $\mm(\Omega)<+\infty.$ Then for any non-negative $u \in \LIP_c(\Omega)$ there exists a sequence of non-negative $u_n \in \LIP_c(\Omega)$ satisfying $|D u_n|_1\neq0 $ $\mm$-a.e.\  in $\{u_n>0\}$ and such that $u_n \to u$ in $W^{1,p}(\X).$
\end{lemma}
\begin{proof}
  The case $\Omega \neq \X$ has been proven in \cite[Lemma 3.6 and Corollary 3.7]{MS20}. The proof presented there, as it is written, does not cover the case  $\Omega=\X$ with $\X$ compact and $\supp(u)=\X$. However, the argument can be easily adapted by considering a sequence $\eps_n\to 0$ such that $\mm(\{\lip(u_n)=\eps_n\})=0$ and taking 
	\[
	u_n\coloneqq u+\eps_nv,
	\]
	with $v(x)\coloneqq\sfd(x_0,x)$, for an arbitrary fixed point $x_0 \in \X$. Since $v\in \LIP(\X)$ and $\lip(v)=1$ $\mm$-a.e.\  in $\X$, arguing exactly as in \cite[Lemma 3.6]{MS20} we get that $u_n \to u$ in $W^{1,p}(\X)$ and $\lip (u_n)\neq0 $ $\mm$-a.e.\ in $\{u_n>0\}$. To get the claimed non-vanishing of $|D u_n|_1$, as in  \cite[Corollary 3.7]{MS20} we use the existence of a constant $c>0$ such that
	\[
	|D u|_1\ge c\,\lip (u), \quad \mm\text{-a.e.},
		\]
	for every $u \in \LIP_{loc}(\X)$, which holds from the results in \cite{Ambrosio-Pinamonti-Speight15} and the fact that $\cd(K,N)$ spaces are locally doubling and supports a local-Poincar\'e inequality.
\end{proof}
\begin{lemma}[Derivative of the distribution function, (\cite{MS20})]\label{lemma:derivativedistribution}
	Let $\Xdm$ be a metric measure space and let $\Omega\subseteq X$ be an open subset with $\mm(\Omega)<+\infty$.
	Assume that $u\in \LIP_{c}(\Omega)$ is non-negative and $|{D u}|_1(x)\neq0$ for $\mm$-a.e.\  $x\in\{u>0\}$. Then its distribution function $\mu:[0, +\infty)\to [0, \mm(\Omega)]$, defined in \eqref{eq:defdistr}, is absolutely continuous. Moreover it holds
	\begin{equation}\label{eq:derivativedistributionfunction}
		\mu'(t)=-\int\frac{1}{|{D u}|_1}\,\d \Per{\{u>t\},\cdot} \qquad \text{a.e.},
	\end{equation}
	where the quantity $1/|{D u}|_1$ is defined to be $0$ whenever $|{D u}|_1=0$.
\end{lemma}

%################################################################

\section{Upper bound for $\alpha_p$}\label{sec:upper bound alpha}
To prove an upper bound of $\alpha_p$  we will need to derive a Sobolev inequality of the type \eqref{eq:alfa sobolev} for some explicit $A$. This will be achieved by  proving first a class of local Sobolev-inequalities (see Theorem \ref{thm:local sobolev intro}) and then ``patch'' them together (see Theorem \ref{thm:Aopt upper cd}) to obtain the desired global inequality. The local-Sobolev inequalities  will be achieved through a Euclidean P\'olya-Szeg\H{o}  symmetrization  inequality (Theorem \ref{thm:eu_polyaszego}).

\subsection{P\'olya-Szeg\H{o} inequality of Euclidean-type}\label{sec:euclidean polya}

The goal of this section is to prove a Euclidean-variant  of the P\'olya-Szeg\H{o} inequality for $\cd(K,N)$ spaces derived in \cite{MS20} (under essentially nonbranching assumption, see also Section \ref{sec:polya}). The main difference is that our inequality holds for arbitrary $K \in \R$ and assumes the a priori validity  of a Euclidean-type isoperimetric inequality, while the one in \cite{MS20} requires $K>0$ and it is based on the L\'evy-Gromov isoperimetric inequality for the ${\rm{CD}}(K,N)$ condition. As opposed to Section \ref{sec:polya}, where the symmetrization has as target the model space for the $\cd(K,N)$ condition with $K>0$, we will use a notion of symmetrization that lives in the weighted half line $([0,\infty),|.|, t^{N-1} \Leb 1)$. It should be remarked that, in general, there is not a  natural  curvature model space to symmetrize functions defined on an arbitrary $\cd(K,N)$-space with $K\le 0$. This  is because there is not a unique model-space for the L\'evy-Gromov isoperimetric inequality in the case $K\le 0$ (see \cite{Milman15}). Therefore, it is unclear in this high-generality where the rearrangements should live. For this reason we will equip the metric measure spaces under consideration with a (possibly local) isoperimetric inequality of  Euclidean-type:
\[ \Per{E} \ge C\mm(E)^{\frac{N-1}N},\]
for $N>1$ and $C$ a non-negative constant.

We start with the definition of \emph{Euclidean model space} $(I_{0,N},|.|,\mm_{0,N})$, $N \in(1,\infty)$:
\[ I_{0,N}:=[0,\infty), \qquad \mm_{0,N}:= \sigma_{N-1} t^{N-1}\Leb 1,\]
where $|.|$ is the Euclidean distance. Next, we define the \emph{Euclidean monotone rearrangement}.
\begin{definition}[Euclidean monotone rearrangement] \label{def:eucl rearrangement}
    Let $\Xdm$ be a metric measure space and $\Omega \subset \X$ be open with $\mm(\Omega)<+\infty$.
    For any Borel function $u\colon \Omega \rightarrow \R^+$, we define $\Omega^*:= [0,r]$ with $\mm_{0,N}([0,r])= \mm(\Omega)$ (i.e. $r^N=\omega_N^{-1}\mm(\Omega)$) and  the monotone rearrangement $u^*_{0,N} : \Omega^* \rightarrow \R^+$ by
    \[ u^*_{0,N}(x) := u^\#(\mm_{0,N}([0,x]))=u^\#(\omega_Nx^N ), \qquad \forall x \in \Omega^*,\]
    where $u^\#$ is the generalized inverse of the distribution function of $u$, as defined in Section \ref{sec:polya}.
\end{definition}
In the sequel, whenever we fix $\Omega$ and $u \colon \Omega \to [0,\infty)$,  the set $\Omega^*$ and the rearrangement $u^*_{0,N}$ are automatically defined as  above. 
\begin{proposition}\label{prop:equimeasurability1}
	 Let $\Xdm$ be a metric measure space and $\Omega \subset \X$ be open and bounded with $\mm(\Omega)<+\infty$. Let $u:\Omega\to[0,+\infty)$ be Borel and let $u_{0,N}^*:\Omega^*\to[0,+\infty)$ be its monotone rearrangement.
	\\ Then, $u$ and $u_{0,N}^*$ have the same distribution function. Moreover 
	\begin{equation}\label{eq:equalLpnorm2}
		\|{u}\|_{L^p(\Omega)}=\|{u_{0,N}^*}\|_{L^p(\Omega^*)}, \qquad \forall \ 1\le p<+\infty,
	\end{equation}
	and the radial decreasing rearrangement operator $L^p(\Omega)\ni u \mapsto u_{0,N}^{*} \in L^p(\Omega^*)$ is continuous.
\end{proposition}
The proof of the above proposition is classical, following e.g. \cite{Kesavan06}, with straightforward modification for the metric measure setting (see also \cite{MS20}). Observe also that, given $u \in L^p(\Omega)$, its monotone rearrangement must be defined by fixing a Borel representative of $u$. However, this choice does not affect the outcome object $u^*_{0,N}$, as clearly the distribution function $\mu(t)$ of $u$ is independent of the representative. 

We now introduce the additional assumption that will make this section meaningful. For some open set $\Omega\subset \X$ and a number $N \in (1,\infty)$, we require the validity of the following local Euclidean-isoperimetric inequality
\begin{equation} \label{eq:isop_eu}
\Per{E} \ge {\sf C_{Isop}} \mm(E)^{\frac{N-1}N}, \qquad \forall\,   E \subset \Omega \text{ Borel.}
\end{equation}
where  ${\sf C_{Isop}}$ is a positive constant independent of $E.$

\begin{remark}\label{rmk:eucl isop} \rm
There is a rich literature about Euclidean-type isoperimetric inequalities in metric measure spaces. Inequalities as in \eqref{eq:isop_eu} have been proven to hold, at least on balls, in the general setting of locally doubling metric measure spaces satisfying a weak local $(1,1)$-Poincar\'e inequality (see, e.g., \cite{Ambrosio02,Miranda03}). In this setting we also mention the recent \cite{AntPasPoz21}, where a global Euclidean-type isoperimetric inequality for small volumes is proved.  In the context of ${\rm CD}(K,N)$ spaces, local almost-Euclidean isoperimetric inequalities have been derived in \cite{CM20}, while in the recent \cite{BK21}, a global version of \eqref{eq:isop_eu} is proven to hold  in ${\rm CD}(0,N)$ spaces with Euclidean-volume growth. For us, the validity of \eqref{eq:isop_eu} will come from Theorem \ref{theorem:almost euclidean isop}.\fr
\end{remark}

\begin{proposition}[Lipschitz to Lipschitz property of the rearrangement]\label{prop:liptolip1}
	Let $(\X,\sfd,\mm)$ be a metric measure space and let $\Omega\subset \X$ be open with $\mm(\Omega)<+\infty$. Assume furthermore that, for some $N \in (1,\infty)$ and ${\sf C_{Isop}}>0$, the isoperimetric inequality in \eqref{eq:isop_eu} holds in $\Omega$. Finally, let  $u\in \LIP_c(\Omega)$ be non-negative with Lipschitz constant $L\ge0$ and such that $|{D u}|_1(x)\neq 0$ for $\mm$-a.e.\  $x\in\{u>0\}$. Then $u_{0,N}^*\in \LIP(\Omega^*)$ with $\Lip(u_{0,N}^*)\le N\omega_N^\frac1N L /{\sf C_{Isop}}$.
\end{proposition}
\begin{proof}
    	We closely follow \cite{MS20}.	Let $\mu$ be the distribution function associated to $u$ and denote by $M:=\sup{u}<+\infty$. The assumptions grant that $\mu$ is continuous and  strictly decreasing. Therefore for any $s,k\ge 0$ such that $s+k\le\mm(\supp(u))$ we can find $0\le t-h\le t\le M$ in such a way that $\mu(t-h)=s+k$ and $\mu(t)=s$. Then from the coarea formula \eqref{eq:coarea} and the $L$-Lipschitzianity of $u$ we get
	\begin{equation}\label{eq:litolip eq2}
		\int_{t-h}^t \Per{\{u>r\},\cdot}\, \d r = \int_{\{t-h< u\le t\}}|D u|_1\, \d \mm \le L (\mu(t-h)-\mu(t))=kL.
	\end{equation}
Observe that $\{u>r\}\subset \Omega$ for every $r>0$, therefore we can apply the isoperimetric inequality \eqref{eq:isop_eu} and obtain that
\[
\Per{\{u>r\}}\ge {\sf C_{Isop}}\mu(r)^{\frac{N-1}{N}},\qquad \forall r>0.
\]
Therefore from \eqref{eq:litolip eq2} and the monotonicity of $\mu$ we obtain
\[
kL\ge {\sf C_{Isop}}\int_{t-h}^t \mu(r)^{\frac{N-1}{N}}\, \d r\ge {\sf C}_{\sf Isop}h \mu(t)^{\frac{N-1}{N}},
\]
from which, observing that in this case $u^\#$ is the inverse of $\mu$, we reach 
\[
u^\#(s)-u^\#(s+k)\le s^{-1+1/N}{\sf C}^{-1}_{{\sf Isop}}kL.
\]
In particular $u^\#$ is Lipschitz in $(\eps,\supp(u)]$ (and thus in $(\eps,\mm(\Omega)]$) for every $\eps>0$ and at every one of its differentiability points $s\in (0,\mm(\Omega))$ it holds that  
$$
-\frac{\d}{\d s}u^\#(s)\le s^{1-1/N}{\sf C}^{-1}_{{\sf Isop}}L.
$$
Fix now two arbitrary and distinct points $x,y \in \Omega^*$ and assume without loss of generality that $y>x$. Recalling the definition of $u_{0,N}^*$ we have that $u_{0,N}^*(x)\ge u_{0,N}^*(y)$ and
\begin{align*}
u_{0,N}^*(x)-u_{0,N}^*(y)&=u^\#(\omega_N x^N)-u^\#(\omega_N y^N)=\int_{\omega_N x^N}^{\omega_N y^N}-\frac{\d}{\d s}u^\#(s)\, \d s\\
			&\le  \int_{\omega_N x^N}^{\omega_N y^N}\frac{s^{-1+1/N}}{{\sf C_{Isop}}}L \, \d s= \omega_N^\frac1N\frac{NL}{{\sf C_{Isop}}}|x-y|,
\end{align*}
which proves that $u_{0,N}^*: \Omega^*\to [0,\infty)$ is $N\omega_N^\frac1N L /{\sf C_{Isop}}$-Lipschitz. 
%The fact that $\supp (u_{0,N}^*)$ is compact in $\Omega^*$ follows immediately from the fact that $\mm(\supp( u))<\mm(\Omega)$ and that $u^\#(s)=0$ for every $s \in (\mm(\supp (u)),\mm(\Omega)].$
\end{proof}

The proof of the following result is exactly the same as in Lemma 3.11 of \cite{MS20}, since the only relevant fact for the proof is that $\mm_{0,N}=h_N\Leb 1$ with weight $h_N$ which is bounded away from zero out of the origin (recall also \eqref{eq:sobolev ac}).
\begin{lemma}\label{lem:identification}
Let $p\in(1,\infty).$ Let $u \in W^{1,p}([0,r],|.|,\mm_{0,N})$, with $r\in(0,\infty)$, be monotone. Then $u \in W^{1,1}_{loc}(0,r)$ and it holds that
	\[
	|D u|_1(t)=|u'|(t)=|D u|(t), \quad \text{for a.e.\  }t \in[0,r]. 
	\]
\end{lemma}

\begin{theorem}[Euclidean P\'olya-Szeg\H{o} inequality]\label{thm:eu_polyaszego}
		Let $(\X,\sfd,\mm)$ be a ${\rm CD}(K,N')$ space, $K \in \R$ $N' \in(1,\infty)$  and let $\Omega\subset \X$ be open with $\mm(\Omega)<+\infty$. Assume furthermore that, for some $N \in (1,\infty)$ and ${\sf C_{Isop}}>0$, the isoperimetric inequality in \eqref{eq:isop_eu} holds in $\Omega$. Then the Euclidean-rearrangement maps $W^{1,p}_0(\Omega)$ to $W^{1,p}(\Omega^*,|.|,\mm_{0,N})$  for any $1<p<+\infty$. Moreover for any $u \in W^{1,p}_0(\Omega)$ it holds 
	\begin{equation}\label{eq:euclpolya}
		\int_{\Omega}|{D u}|^p\d \mm\ge \Big(\tfrac{{\sf C_{Isop}}}{N\omega_N^{1/N}}\Big)^p\int_{\Omega^*}|{D u_{0,N}^*}|^p\d \mm_{0,N}.
	\end{equation}
\end{theorem}
\begin{proof}
	The proof is  a standard  argument and we  follow  \cite{MS20} for its adaptation to the non-smooth setting. We first prove the result assuming that $u \in \LIP_c(\Omega)$ and $|{D u}|_1(x)\neq 0$ for $\mm$-a.e.\  $x\in\{u>0\}$, then the general case will follow by approximation. Set $M\coloneqq \sup u$.
% 	and define the functions $\phi, \psi: [0,M] \to \R^+$ as follows
% 	\[
% 	\phi(t)\coloneqq \int_{\{u>t\}}|D u|_1^p\, \d \mm, \quad \psi(t)\coloneqq \int_{\{u>t\}}|D u|_1\, \d \mm.
% 	\]
% 	An application of the coarea formula \eqref{eq:coarea} gives at once that $\phi,\psi$ are absolutely continuous with 
% 	\[
% 	\phi'(t)=-\int |D u|_1^{p-1} \, \d \Per{\{u>t\},\cdot}, \quad \psi'(t)=-\Per{\{u<t\}}, \quad \text{for a.e.\  $t\in [0,M]$},
% 	\]
% 	for any  Borel representative of $|D u|_1$, which we assume to be fixed from now until the end of the proof. From the H\"older inequality we have
% 	\[
% 	\psi(t-h)-\psi(t)\le (\phi(t-h)-\phi(t))^{1/p}(\mu(t-h)-\mu(t))^{(p-1)/p}, \quad 0\le t-h\le t< M,
% 	\]
% 	where $\mu$ denotes the distribution function of $u$. From Lemma \ref{lemma:derivativedistribution}, we know that also $\mu$ is absolutely continuous, in particular we have that a.e.\  $t\in[0,M]$ is  at the same time a differentiability point for $\phi,\psi$ and $\mu$.  Choosing one of such $t$'s in the above inequality, dividing by $h>0$ and passing to the limit as $h\to 0^+$ we obtain
% 	\[
% 	-\psi'(t)\le (-\phi'(t))^{1/p}(-\mu'(t))^{(p-1)/p}, \quad \text{for a.e.\  $t\in [0,M]$}.
% 	\]
% 	Moreover, by the validity of \eqref{eq:isop_eu},  we have that $\Per{\{u>t\}}\ge {\sf C_{Isop}}\mu(t)^{(N-1)/N}$. Therefore 
% 	\begin{equation*}
% 	-\phi'(t)\ge \frac{{{\sf C}^p_{\sf Isop} } \mu(t)^{\frac{(N-1)p}{N}}}{(-\mu'(t))^{p-1}}, \quad \text{for a.e.\  $t\in [0,M]$}.
% 	\end{equation*}
% and integrating we reach
From the coarea formula   \eqref{eq:coarea} and the assumed isoperimetric inequality \eqref{eq:isop_eu} we can obtain  (see e.g.\ the proof of Prop. 3.12 in \cite{MS20})
\begin{equation}\label{eq:polya part 2}
	\int_{\Omega} |D u|_1^p\, \d \mm \ge \int_0^M  \frac{{{\sf C}^p_{\sf Isop} } \mu(t)^{\frac{(N-1)p}{N}}}{(-\mu'(t))^{p-1}} \, \d t,
\end{equation}
where $\mu'(t)$ exists a.e.\ since from Lemma \ref{lemma:derivativedistribution} $\mu$ is absolutely continuous.

	Recall now from Proposition \ref{prop:equimeasurability1} that $\mu(t)=\mm(\{u_{0,N}^*>t\})$, where $u_{0,N}^*: \Omega^*\to \R^+$ is the Euclidean monotone rearrangement. Moreover, thanks to the non-vanishing assumptions on $|D u|_1$, we have from Proposition \ref{prop:liptolip1} that $u_{0,N}^*\in \LIP(\Omega^*)$. Additionally $u_{0,N}^*$ is strictly decreasing in $(0,\mm(\supp(u)))$ and in particular $\{u^*_{0,N}>t\}=[0,r_t)$ (and $\{u^*_{0,N}=t\}=\{r_t\}$) for some $r_t\in[0,\mm(\Omega)]$, for every $t \in (0,M)$.  Note that $r_t$ can be computed explicitly as $r_t=(\omega_N^{-1}\mu(t))^{1/N}$, which also shows that $t\mapsto r_t$ is a locally absolutely continuous map.
%	 since (by monotonicity) $\{u^*_{0,N} >t\}$ is an interval of type $(0,r_t)\subseteq \R^+$ for some $r_t$ we have that $\lip \ u^*_{0,N}(x) = \lip \ g(r_t) $ for every $x \in \{u^*_{0,N}=t\}$, where $g(r) := u^\#(r^N)$. 
Combining these observations with Lemma \ref{lemma:derivativedistribution} and recalling Lemma \ref{lem:identification} we have following expression for the derivative of $\mu$:
	\[
	-\mu'(t)= \int_{\{u_{0,N}^*=t\}} |D u_{0,N}^*|_1^{-1} \d \Per{\{u_{0,N}^*>t\},\cdot}=\frac{\Per{\{u_{0,N}^*>t\}}}{|(u_{0,N}^*)'|(r_t)}\, \quad \text{for a.e.\  $t\in (0,M)$},
	\]
	where $r_t$ is as above. It is clear that $\Per{[0,r)}=\sigma_{N-1}r^{N-1}$ for every $r \in(0,\infty)$ (where the perimeter is computed in the space $(I_{0,N},|.|,\mm_{0,N})$, therefore  $\Per{\{u_{0,N}^*>t\}}=N\omega_N^\frac1N\mu(t)^\frac{N-1}{N}$, from which we deduce 
	\[
	-\mu'(t)=N\omega_N^\frac1N\frac{\mu(t)^\frac{N-1}{N}}{|(u_{0,N}^*)'|(r_t)}\, \quad \text{for a.e.\  $t\in [0,M]$}.
	\]
	Plugging this identity in \eqref{eq:polya part 2} and recalling also Lemma \ref{lem:identification} we reach
	\begin{align*}
		\int_{\Omega} |D u|_1^p\, \d \mm \ge {{\sf C}^p_{\sf Isop} }(N\omega_N^{1/N})^{1-p}\int_0^M  (|(u_{0,N}^*)'|(r_t))^{p-1}\mu(t)^\frac{(N-1)}{N} \, \d t = 	\big(\tfrac{{\sf C_{Isop}}}{N\omega_N^{1/N}}\big)^p	\int_{\Omega^*}|D u_{0,N}^*|^p\, \d \mm.
	\end{align*}
Recalling that $|Du|_1\le \lip \ u$ $\mm$-a.e., $u \in \LIP_{bs}(\X)$, we obtain \eqref{eq:euclpolya}. For general $u \in W^{1,p}_0(\Omega)$ the result follows via approximation via Lemma \ref{lem:nonvanishing gradients} exactly as in the proof of \cite[Theorem 1.4]{MS20}.
\end{proof}

\begin{remark}
\rm
	It follows from its proof, that Theorem \ref{thm:eu_polyaszego} holds with the weaker assumption that $\Xdm$ is uniformly locally doubling and supports a weak local $(1,1)$-Poincar\'e inequality. Recall also from Remark \ref{rmk:eucl isop} that under these assumptions an isoperimetric inequality as in \eqref{eq:isop_eu} is available. \fr
\end{remark}

\subsection{Local Sobolev inequality}
The main goal of this section is to prove  the following local Sobolev inequality of Euclidean-type.
\begin{theorem}[Local Euclidean-Sobolev inequality]\label{thm:local sobolev intro}
For every $\eps>0$, $N \in(1,\infty)$ and $D>0$ there exists $\delta=\delta(\eps,D,N)>0$ such that the following holds.
	Let $\Xdm$ be a $\cd(K,N)$ space, $K \in \R$. Let $r,R\in (0,\frac{1}{2}\sqrt{N/K^-})$ and $x \in \X$ be such that $r<\delta R$, $R<\delta \sqrt{N/K^-}$ (with $\sqrt{N/K^-}\coloneqq +\infty$ if $K\ge0$) and $\frac{\mm(B_r(x))}{\mm(B_R(x))}\le D(r/R)^N$. Then
	\begin{equation}\label{eq:local sobolev intro}
		\|u\|_{L^{p^*}(\mm)}\le (1+\eps)\eucl(N,p)\left(\frac{\mm(B_R(x))}{R^N\omega_N}\right)^{-\frac{1}{N}}\||D u|\|_{L^p(\mm)}, \quad \forall u \in \LIP_c(B_{r}(x)).
	\end{equation}
\end{theorem}
We mention that local ``almost-Euclidean" Sobolev inequalities as in the above result are well known on Riemannian manifolds, however they usually depend on double sided bounds on the sectional curvature or on Ricci lower bounds coupled with a lower bound on the injectivity radius (see e.g. \cite[Lemma 2.24]{Aubin82} and \cite[Lemma 7.1, Sec. 7.1]{Hebey99}). Instead  in our case we only need a lower bound on the Ricci curvature and bounds on the measure of small balls, for this reason Theorem \ref{thm:local sobolev intro} appears interesting also in the smooth setting.

We face now  a necessary step for the proof of Theorem \ref{thm:local sobolev intro} starting with the following local isoperimetric inequality of Euclidean type to be used in conjunction with P\'olya-Szeg\H{o} inequality developed in the previous section. The proof relies on the Brunn-Minkowski inequality and it is mainly inspired by \cite{BK21}, where sharp global isoperimetric inequalities for $\cd(0,N)$ spaces have been proved (see also \cite{ABFP21} for a refinement and the  previous  \cite{Brendle20} and \cite{FMhull} for   the smooth case). It is worth mentioning that a class of ``almost-Euclidean'' isoperimetric inequalities in essentially nonbranching $\cd$-spaces, similar to the following ones,  were proved in \cite{CM20} via localization-technique.  However, the results in \cite{CM20} present a set of assumptions that are not suitable for our purposes. Moreover our arguments are  different and do not assume the space to be essentially non-branching.

\begin{theorem}[Almost-Euclidean isoperimetric inequality]\label{theorem:almost euclidean isop}
	Let $\Xdm$ be a $\cd(K,N)$ space for some $N\in(1,\infty), K \in \R$.  Then for every $0<r<R<\frac{1}{2}\sqrt{N/K^-}$ (where $\sqrt{N/K^-}=+\infty$ for $K\ge 0$) and $x \in \X$ we have  
	\begin{equation}\label{eq:local euclidean isop}
	\Per E \ge  \mm(E)^\frac{N-1}{N}  N\omega_N^{\frac1N}\theta_{N,R}^{\frac1N}(x)(1-(2C_{r,R}^{1/N}+1)\delta-\eta), \quad \forall \, E \subset B_r(x),
	\end{equation}
	where $\delta\coloneqq \frac{r}{R}$, $\eta\coloneqq  1-\frac{2R\sqrt{K^-/N}}{\sinh (2R\sqrt{K^-/N})}$ (taken to be zero when $K\ge0$) and $C_{r,R}\coloneqq\theta_{N,r}(x)/\theta_{N,R}(x)$.
\end{theorem}
\begin{proof}
	It is sufficient to prove \eqref{eq:local euclidean isop} with the Minkowski content $\mm(E)^+$ instead of the perimeter. Indeed we could then apply the approximation result in Proposition \ref{prop:relax minkowski} below to deduce that for every $r'\in(r,R)$, \eqref{eq:local euclidean isop} holds with $r=r'$ (this time  with $\Per E$). Noticing that $\theta_{N,r'}(x)\to \theta_{N,r}(x)$ as $r'\downarrow r$, sending $r'\to r$ would give the conclusion.

	Let $r,R\in \R^+$ with $r<R$ and fix $E\subset B_r(x_0)$ with $\mm(E)>0.$ We aim to apply the Brunn-Minkowski inequality to the sets $A_0\coloneqq E$, $A_1\coloneqq B_R(x_0)$. The triangle inequality easily yields that $A_t\subset E^{t(r+R)}$ for every $t \in (0,1)$ (recall that $E^\eps$ is the $\eps$-enlargement of the set $E$, while $A_t$ is the set of $t$-midpoint between $A_0,A_1$). We consider first the case $K\ge 0$. From the Brunn-Minkowski applied with $K=0$ we obtain
	\begin{align*}
		\mm^+(E)&=\liminf_{\eps\to 0^+} \frac{\mm(E^{\eps})-\mm(E)}{\eps}= \liminf_{t\to 0^+} \frac{\mm(E^{t(r+R)})-\mm(E)}{t(r+R)}\\
		&\overset{\eqref{eq:brunn}}{\ge} \lim_{t\to 0^+} \frac{\left(t\mm(B_R(x_0))^{1/N}+(1-t)\mm(E)^{1/N}\right)^{N}-\mm(E)}{t(r+R)}\\
		&=N\mm(E)^\frac{N-1}{N}\frac{\mm(B_R(x_0))^{1/N}-\mm(E)^{1/N}}{r+R}\\
		&\ge N\mm(E)^\frac{N-1}{N}\frac{\mm(B_R(x_0))^{1/N}-\mm(B_r(x_0))^{1/N}}{r+R},
	\end{align*}
	where we have used that $E \subset B_r(x_0).$  If instead $K<0$, arguing analogously we obtain
	\begin{align*}
		\mm^+(E)\ge \frac{N\mm(E)^\frac{N-1}{N}}{r+R}\Big(\frac{\theta\sqrt{-K/N}}{\sinh(\theta\sqrt{-K/N})}\mm(B_R(x_0))^{\frac 1N}-\frac{\theta\sqrt{-K/N}\cosh(\theta\sqrt{-K/N})}{\sinh(\theta\sqrt{-K/N})}\mm(B_r(x_0))^{\frac 1N}\Big),
	\end{align*}
	where  $\theta$ is the maximal length of geodesics from $A_{0}$ to $A_{1}$. It is clear that $\theta\le r+R.$
Note also that $\frac{t}{\sinh(t)}$ is decreasing and less or equal than one for $t \ge 0$, moreover  for $t\le 1$ we have   $\cosh(t)\le 1+t$. In particular if $R\le\frac12 \sqrt{-N/K}$ we obtain that
	\begin{align*}
		\mm^+(E)\ge \frac{N\mm(E)^\frac{N-1}{N}}{r+R}\left(\frac{\sqrt{-K/N}(r+R)\mm(B_R(x_0))^{1/N}}{\sinh (\sqrt{-K/N}(r+R))}-(1+\sqrt{-K/N}(r+R))\mm(B_r(x_0))^{1/N}\right).
	\end{align*}
	Going back to the case of a general $K \in \R$, combining the above estimates and rearranging the terms we reach
	\[
	\mm^+(E)\ge \frac{ \mm(E)^\frac{N-1}{N} N\omega_N^{\frac1N}\theta_{N,R}(x)^{\frac1N} }{1+r/R}  \Big(\frac{\sqrt{K^-/N}(r+R)}{\sinh (\sqrt{K^-/N}(r+R))}-(1+\sqrt{\frac{K^-}{N}}(r+R))\frac rR\Big(\frac{\theta_{N,r}(x)}{\theta_{N,R}(x)}\Big)^\frac1N\Big),
	\]
	provided $R\le\frac12 \sqrt{N/K^-}$ and taking $\frac{t}{\sinh(t)}=1$ for $t=0$.
	Setting $\delta\coloneqq \frac{r}{R}$,$\eta\coloneqq 1-\frac{2R\sqrt{K^-/N}}{\sinh (2R\sqrt{K^-/N})}$ and $C\coloneqq\theta_{N,r}(x)/\theta_{N,R}(x)$, the above gives (recalling that $\frac{t}{\sinh(t)}$ is decreasing for $t \ge 0$)
	\[
	\mm^+(E)\ge  \mm(E)^\frac{N-1}{N} N\omega_N^{\frac1N}\theta_{N,R}(x)^{\frac1N} \frac{1}{1+\delta}  \big(1-\eta-2\delta C^{\frac 1N}\big),
	\]
	that easily implies the conclusion.
\end{proof}
In the above proof was used the following approximation result.
\begin{proposition}[\cite{AGDM17}]\label{prop:relax minkowski}
     Let $\Xdm$ be a  metric measure space and let $E\subset B_r(x)$ be Borel with finite perimeter and $\mm(E)<+\infty$. Then for every $r'>r$ there exists a sequence $E_n \subset B_{r'}(x)$ of closed sets such that $\nchi_{E_n}\to \nchi_E$ in $L^1(\mm)$ and
     \[
     \Per E=\lim_{n\to \infty} \mm^+(E_n).
     \]
\end{proposition}
\begin{proof}
    The result is contained in \cite{AGDM17}, however since it does not appear in this exact form we provide some details. The result follows observing that there exists a sequence $f_n \in \LIP(\X)$ with $\supp (f_n)\subset B_{r'}(x)$ so that $f_n \to \nchi_E$ in $L^1(\mm)$ and $\Per E=\lim_n \int \lip f_n \, \d \mm.$ Indeed from this fact,  the conclusion follows arguing as in the end of the proof of \cite[Theorem 3.6]{AGDM17}.
    
    To construct the sequence $(f_n)$ we known that from the definition of perimeter there exist $g_n \in \LIP_{loc}(\X)$ so that $g_n \to \nchi_E$ in $L^1(\mm)$  and $\Per E=\lim_n \int \lip \ g_n \, \d \mm$. Moreover we can build a cut-off function $\eta \in \LIP(\X)$ such that $\eta=1$ in $B_r(x)$, $0\le \eta\le 1$, $\supp (\eta)\subset B_{r'}(\X)$ and $\Lip(\eta)\le 2(r'-r)^{-1}.$ Then we simply take $f_n\coloneqq g_n\eta.$ Clearly  $f_n \to \nchi_E$. Moreover
    \[
    \Per E\le \liminf_{n\to \infty} \int \lip \ f_n \, \d \mm\le \liminf_{n\to \infty} \int \lip \ g_n \, \d \mm + \frac{2}{r'-r}\int_{E^c} g_n \, \d \mm=\Per E,
    \]
    that is what we wanted.
\end{proof}
Next, we recall the following classical one-dimensional inequality by Bliss \cite{Bliss30} (see also \cite{Aubin82,Talenti76}).
\begin{lemma}[Bliss inequality]\label{lem:bliss}
	Let $u\colon [0,\infty)\to \R$ be locally absolutely continuous. Then for any $1<p<N$ it holds 
	\begin{equation}\label{eq:bliss}
		\Big(\sigma_{N-1}\int_0^\infty |u|^{p^*}\, t^{N-1}\, \d t\Big)^\frac{1}{p^*}\le {\rm Eucl}(N,p)\Big(\sigma_{N-1}\int_0^\infty |u'|^{p}\, t^{N-1}\, \d t\Big)^\frac{1}{p},
	\end{equation}
	 whenever one side is finite and where $p^*\coloneqq pN/(N-p)$. Moreover the functions $v_b(r)\coloneqq (1+b r^\frac{p}{p-1})^\frac{p-N}{p}$, $b>0$, satisfy \eqref{eq:bliss} with equality.
\end{lemma}
With the above local isoperimetric inequality and the Euclidean P\'olya-Szeg\H{o} inequality, the strategy is now to symmetrize functions on the space and exploit the Bliss inequality to deduce the desired local-Sobolev inequalities.

\begin{proof}[Proof of Theorem \ref{thm:local sobolev intro}]
We start observing that it is enough to prove \eqref{eq:local sobolev intro} for non-negative functions. 
	Fix $u \in \LIP_c(B_r(x))$ non-negative and consider $u^*_{0,N}:B_r(x)^*\to [0,\infty)$ be the Euclidean-rearrangement of $u$ as in Definition \ref{def:eucl rearrangement}, where $B_r(x)^*=[0,t]$ for some $t>0.$ The local Euclidean-isoperimetric inequality given by Theorem \ref{theorem:almost euclidean isop} implies that the hypothesis of Proposition \ref{prop:liptolip1} and Theorem \ref{thm:eu_polyaszego} are fulfilled with $\Omega=B_{r}(x)$ and ${\sf C_{Isop}}= (1-(2D^{1/N}+1)\delta'-2\eta)N\omega_N^{\frac 1N}\theta_{N,R}(x)^{\frac 1N},$ with $\delta'\coloneqq \frac{r}{R}$, $\eta\coloneqq  1-\frac{2R\sqrt{K^-/N}}{\sinh (2R\sqrt{K^-/N})}$ ($=0$ if $K\ge0$) and $D\coloneqq\theta_{N,r}(x)/\theta_{N,R}(x)$. In particular it holds that $u^*_{0,N}\in W^{1,p}([0,t],|.|,\mm_{0,N})$, which implies \sloppy (recall \eqref{eq:sobolev ac}) that $u^*_{0,N} \in W^{1,1}_{loc}(0,t)$  with $(u^*_{0,N})' \in L^p(\mm_{0,N})$  and $|Du^*_{0,N}|=|(u^*_{0,N})'|$ a.e.. Moreover, since $\mm_{0,N}$ is bounded away from 0 far from the origin, $u^*_{0,N} \in W^{1,1}(\eps,t]$ for every $\eps>0$ and  by definition $u^*_{0,N}(t)=0$. Therefore $u^*_{0,N}$ (extended by 0 in $(t,\infty)$) satisfies the assumptions for the Bliss inequality.  Recall also from Proposition \ref{prop:equimeasurability1} that $\|u^*_{0,N}\|_{L^p(\mm_{0,N})}=\|u\|_{L^p(\mm)}$ for every $p \in [1,\infty).$ Therefore we are in position to apply the Euclidean P\'olya-Szeg\H{o} inequality given by \eqref{eq:euclpolya}, that combined with the Bliss-inequality \eqref{eq:bliss} gives
	\begin{align*}
		\|u\|_{L^{p^*}(\mm)}&=\|u^*_{0,N}\|_{L^{p^*}(\mm_{0,N})}\overset{\eqref{eq:bliss}}{\le } {\rm Eucl}(N,p) \||D u^*_{0,N} |\|_{L^p(\mm_{0,N})} \\
		&\overset{\eqref{eq:euclpolya}}{\le } \frac{\eucl(N,p) \theta_{N,R}(x)^{-\frac1N} }{(1-(2D^{1/N}+1)\delta'-2\eta)}\||D u |\|_{L^p(\mm)}.
	\end{align*}
	Finally from  the above and observing that  $\frac{\mm(B_r(x))}{\mm(B_R(x))}= D(r/R)^N,$ we immediately see that there exists  $\delta:=\delta(\eps,D,N)$ so that, provided $\delta',\eta<\delta$, \eqref{eq:local sobolev intro} holds.
\end{proof}

We end this section with another simpler variant of local Sobolev inequality. It will be needed to deal with cases where $\theta_N(x)=+\infty,$ where  Theorem \ref{thm:local sobolev intro} does not give the right information.
\begin{proposition}[Local Sobolev embedding]\label{prop:local sobolev embedding}
	Let $\Xdm$ be a $\cd(K,N)$ space for some $N\in(1,\infty)$, $K \in \R$. Then, for every $p \in(1,N)$ and every $B_r(x)\subset \X$ with $r\le 1$, it holds
	\begin{equation}\label{eq:local sobolev}
		\Big(\int_{B_r(x)}|u|^{p^*} \, \d \mm \Big)^\frac{p}{p^*} \le \Big(\frac{C  \mm(B_r(x)) }{r^N}\Big)^{-\frac pN} \int_{B_{2r}(x)} |D u|^p\, \d \mm+2^p\mm(B_r(x))^{-\frac pN}\int_{B_r(x)}|u|^{p} \, \d \mm ,
	\end{equation}
	for every $u \in \LIP(\X)$, where $p^*=pN/(N-p)$ and $C=C(K,N,p).$
\end{proposition}
\begin{proof}
	Applying \eqref{eq:improved poincaret} and the Bishop-Gromov inequality
	\begin{align*}
		\Big(\int_{B_r(x)}|u|^{p^*} \, \d \mm \Big)^\frac{1}{p^*} &\le C_1r\frac{\mm(B_r(x))^{1/p^*}}{\mm(B_{2r}(x))^{1/p}} \Big(\int_{B_{2r}(x)} |D u|^p\Big)^\frac{1}{p}+\mm(B_r(x))^{1/p^*} |u_{B_r(x)}|\\
		&\le C_2r\frac{\mm(B_r(x))^{1/p^*}}{\mm(B_{r}(x))^{1/p}} \Big(\int_{B_{2r}(x)} |D u|^p\Big)^\frac{1}{p}+\mm(B_r(x))^{\frac 1{p^*} - \frac 1p}\Big(\int_{B_r(x)}|u|^{p} \, \d \mm \Big)^\frac{1}{p},
        \end{align*}
	for suitable positive constants $C_1,C_2$ depending only on $K,N,p$. The desired conclusion follows raising to the $p$ in the above inequality.
\end{proof}

\subsection{Proof of the upper bound}
The strategy of the proof of the following result is by-now classical and combines local-Sobolev inequalities with a partition of unity argument (see \cite{Aubin76-2},\cite[Chp. 2 Sec. 7]{Aubin82}, \cite[Theorem 4.5]{Hebey99} and also \cite[Prop. 3.3]{ACM13}).
\begin{theorem}[Upper bound on $\alpha_p$]\label{thm:Amin}
	Let $\Xdm$ be a compact  $\cd(K,N)$ space, for some $N\in(1,\infty)$, $K \in \R$. Then, for every $\eps>0$ and every $p \in(1,N)$, there exists a constant $B=B(\eps,p,\X)>0$ such that 
	\begin{equation}\label{eq:thm4.5}
		\|u\|_{L^{p^*}(\mm)}^p\le \Big(\frac{\eucl(N,p)^p}{\min_{\X} \theta_N(x)^{p/N}}+\eps\Big) \||D u|\|_{L^p(\mm)}^p+B\|u\|_{L^p(\mm)}^p, \qquad \forall u \in \LIP(\X).
	\end{equation}
\end{theorem}
\begin{proof}
We start claiming that the following local version of \eqref{eq:thm4.5} holds: for any $x \in\X$ and every $\eps>0$ there exists $r=r(\eps,x)>0$ and $C=C(\eps,p,x)<+\infty$ such that 
	\begin{equation}\label{eq:uniform local}
		\|u\|^p_{L^{p^*}(\mm)}\le \Big(\frac{\eucl(N,p)^p}{\min_{y\in\X} \theta_N(y)^{p/N}}+\eps\Big) \||D u|\|^p_{L^p(\mm)}+C\|u\|_{L^p(\mm)}^p, \qquad \forall u \in \LIP_c(B_{r}(x)).
	\end{equation}
	To show the above we observe first that in the case that $\theta_{N}(x)=+\infty,$  \eqref{eq:uniform local}  follows immediately from \eqref{eq:local sobolev} for $r$ small enough. We are left with the case $0< \theta_N(x)<+\infty$. We start by fixing $\eps\in(0,1/2)$. From the definition of $\theta_N(x)$, there exists $r'=r'(x,\eps)$ so that for every $r\in (0,r')$ it holds $\theta_{N,r}(x)\in ((1-\eps)\theta_N(x),(1+\eps)\theta_N(x)).$  In particular we have that $\frac{\theta_{N,r}(x)}{\theta_{N,R}(x)}\le 4$ for every $r,R \in (0,r').$ We are therefore in position to apply Theorem \ref{thm:local sobolev intro} and deduce that there exists $\delta=\delta(\eps,N)$ so that for every $r,R \in (0,r'\wedge \delta \sqrt{N/K^-})$, with $r<\delta R$,  the following inequality holds for every $ u \in \LIP_c(B_{r}(x))$
	\[
		\|u\|_{L^{p^*}(\mm)}^p\overset{\eqref{eq:local sobolev intro}}{\le} (1+\eps)^p\frac{\eucl(N,p)^p}{\theta_{N,R}(x)^{p/N}}\||D u|\|^p_{L^p(\mm)}\le  \frac{(1+\eps)^p}{(1-\eps)^{p/N}}\frac{\eucl(N,p)^p}{\min_{\X} \theta_N(x)^{p/N}}\||D u|\|^p_{L^p(\mm)},
	\]
 where in the second inequality we have used $\theta_{N,R}(x)\ge (1-\eps)\theta_{N}(x)$. Therefore \eqref{eq:uniform local} (with $C=0$) follows from the above provided we choose $\eps$ small enough.

	Since $\X$ is compact we can extract a finite covering of balls $\{B_i\}_{i=1}^M$ from  the covering $\cup_{x \in\X}B_{r(\eps,x)/2}(x)$. We also set $C\coloneqq\max_i C_i$ and
	$$A\coloneqq \frac{\eucl(N,p)^p}{\min_{\X} \theta_N(x)^{p/N}}+\eps.$$
	We claim that there exists a partition of unity made of functions $\{\phi_i\}_{i=1}^M$ such that $\phi_i \in \LIP_c(2B_i)$, $0\le \phi_i\le 1$ and  $\phi_i^{1/p}\in \LIP_c(2B_i)$ for all $i$, having denoted $2B_i$, the ball of twice the radius. To build such partition of unity we can argue as follows: start considering functions $\psi_i\in \LIP_c(2B_i)$, such that $0\le \psi_i\le 1$ and $\psi_i\ge1$ in $B_i$. Then we fix $\beta>p$ and take
	\[
	\phi_i\coloneqq\frac{\psi_i^\beta}{\sum_{j=1}^M \psi_j^\beta}.
	\]
	Since by construction $\sum_{j=1}^M\psi_j^\beta\ge 1$ everywhere on $\X$, we have that $\phi^{1/p}_i \in \LIP_c(2B_i)$. Finally  it is clear that $\sum_{i=1}^M  \phi_i=1.$ 
	
	We are now ready to prove \eqref{eq:thm4.5}. Fix $u \in \LIP(\X)$ and observe that
	\begin{equation}\label{eq:first estimate}
		\|u\|_{L^{p^*}(\mm)}^{p}=\Big\|\sum_i \phi_i|u|^p \Big \|_{L^{p^*/p}(\mm)}\le \sum_i\left \|\phi_i|u|^p \right \|_{L^{p^*/p}(\mm)}=
		\sum_i\big \|\phi_i^{1/p}|u| \big \|_{L^{p^*}(\mm)}^p.
	\end{equation}
	Since $\phi_i^{1/p}|u| \in \LIP_c(2B_i)$ we can apply \eqref{eq:uniform local} to obtain
	\begin{align*}
		\|u\|_{L^{p^*}(\mm)}^{p}&\le \sum_{i=1}^M A \int \left(| D \phi_i^{1/p}||u|+|D u|\phi_i^{1/p} \right)^p \, \d \mm + C \int \phi_i |u|^p\, \d \mm\\
		& \le\sum_{i=1}^M A \int \phi_i|D u|^p+c_1|D u|^{p-1}\phi_i^{\frac{p-1}{p}}|D \phi_i^{1/p}||u|+c_2|D \phi_i^{1/p}|^{p}|u|^{p} \, \d \mm + C \int \phi_i |u|^p\, \d \mm,
	\end{align*}
	where $c_1,c_2\ge 0$ are such that $(1+t)^p\le 1+c_1t+c_2t^p$ for all $t\ge 0.$ Recalling that the functions $0\le \phi_i^{1/p}\le1$ are Lipschitz we obtain
	\begin{align*}
		\|u\|_{L^{p^*}(\mm)}^{p}&\le A \int |D u|^p\, \d \mm+\tilde C\int |D u|^{p-1}|u|\, \d \mm+\tilde C \int |u|^p\, \d \mm,
	\end{align*}
	where $\tilde C=\tilde C(p,M,L)$, $L$ begin the maximum of the Lipschitz constants of the functions $\phi_i^{1/p}.$ Finally from the Young inequality we have for every $\delta>0$
	\[
	\int |D u|^{p-1}|u|\, \d \mm\le  \frac{p\delta^{\frac{p}{p-1}}}{p-1} \int |Du|^p \, \d \mm+  
	\frac{1}{p\delta^p} \int |u|^p \, \d \mm, \quad \forall \delta>0
	\]
	and plugging this estimate above, choosing $\delta$ small enough (but independent of $u$), we obtain that
	\[
	\|u\|_{L^{p^*}(\mm)}^{p}\le (A+\eps) \int |D u|^p\, \d \mm+C' \int |u|^p\, \d \mm,
	\]
	for some $C'=C'(\eps,L,M,p)$. Since $\eps>0$ and $u \in\LIP(\X)$ were arbitrary,  this concludes the proof.
\end{proof}

\section{Lower bound on $\alpha_p$}\label{sec:lower bound alfa}
The rough idea of the lower bound on $\alpha_p$ is that, when $\theta_N(x)<+\infty$ the space near $x$  has a conical structure, hence the constant in the Sobolev inequality cannot be better than the one of the tangent structures of the underlying  space. This will be formalized with a blow-up argument combined with a stability result for the Sobolev constants.

\subsection{Blow-up analysis of Sobolev constants}
For convenience, we  introduce the following notation: whenever in a metric measure space $\Xdm$ it holds that
\[
	\|u\|_{L^{q}(\mm)}^p\le A  \||D u|_p\|_{L^p(\mm)}^p + B  \|u\|_{L^p(\mm)}^p , \qquad \forall u \in W^{1,p}(\X).
\]
for some constants $A,B>0$ and exponents $1<p<q$, we will say that \textbf{$\X$ supports a $(q,p)$-Sobolev inequality with constants $A,B$}. This convention will be used often here, and sometimes in the subsequent sections, without further notice.

We make precise the scaling enjoyed by the Sobolev inequalities under consideration. It is immediate to check that if a space $\Xdm$ supports a $(p^*,p)$-Sobolev for $p \in (1,N)$ and $p^*\coloneqq \frac{pN}{N-p}$  with constants $A,B$, then for every $r>0$ we have
\begin{equation}
    (\X,\sfd/r,\mm/r^N)\text{ supports a $(p^*,p)$-Sobolev with constants } A,Br^p.\label{eq:scaling}
\end{equation}

We pass to the stability of Sobolev embeddings under pmGH-convergence (see also \cite[Thm. 3.1]{Honda18} for a similar result for Ricci-limits).
\begin{lemma}[pmGH-Stability of Sobolev constants]\label{lem:sobolev stability}
	Let $(\X_n,\sfd_n,\mm_n,x_n)$, $n \in \bar \N$, be a sequence of $\cd(K,N)$ spaces for some $K \in \mathbb{R}$, $N \in (1,\infty)$ with $\X_n \overset{pmGH}{\to} X_\infty$. Suppose $\X_n$ support a $(q,p)$-Sobolev inequality for $1<p<q$ with constants $A,B$. Then also $\X_\infty$ supports a $(q,p)$-Sobolev inequality with the same constants $A,B$.
\end{lemma}
\begin{proof}
		Fix $u \in \LIP_c(\X_\infty)$, from the $\Gamma$-$\limsup$ inequality of the ${\rm Ch}_p$ energy, there exists a sequence $u_n \in W^{1,p}(\X_\infty)$ such that $u_n$ converges in $L^p$-strong to $u$ and $\limsup_{n}\int|D u|^p\, \d \mm_n\le \int |Du|^p\, \d \mm_\infty$. In particular
	\begin{align*}
		\limsup_n \|u_n\|_{L^q(\mm_n)}^p
		&\le \limsup_{n\to \infty} A \||D u_n|\|_{L^p(\mm_n)}^p+B  \|u_n\|_{L^p(\mm_n)}^p\\
		&\le A \||D u|\|_{L^p(\mm_\infty)}^p+ B  \|u\|_{L^p(\mm_\infty)}^p<+\infty.
	\end{align*}
	Therefore $u_n$ converge also $L^q$-weak to $u$. From the  lower semicontinuity of the $L^q$-norm with respect to $L^q$-weak convergence  and the arbitrariness of $u \in \LIP_c(\X_\infty)$ the conclusion follows.
\end{proof}

The following result is a consequence of the existence of the disintegration  and can be found for example in \cite[Corollary 3.8]{DPG15}.
\begin{lemma}\label{lem:volume cone}
     Let $\Xdm$ be a $\cd(0,N)$ space with  $N \in [1,\infty)$. Suppose that for some $x_0 \in \X$ it holds that $\frac{\mm(B_r(x_0))}{\omega_Nr^N}=1$ for every $r\in(0,\infty)$, then 
     \[
     \int \phi(\sfd(x_0,x)) \, \d \mm=\sigma_{N-1}\int_0^\infty \phi(r)r^{N-1}\, \d r, \qquad \forall \phi \in C_c([0,\infty]).
     \]
\end{lemma}

\begin{lemma}\label{lem:extremal sequence cone}
     Let $\Xdm$ be a $\cd(0,N)$ space, $N\in(1,\infty)$,  $p \in (1,N)$ and set $p^*\coloneqq\frac{pN}{N-p}$. Suppose that for some $x_0 \in \X$ it holds that $\frac{\mm(B_r(x_0))}{\omega_N r^N}=1$ for every $r\in(0,\infty)$. Then 
     there exists a sequence of non-constant functions $u_n \in \LIP_c(\X)$ satisfying
     \[
     \lim_n \frac{\|u_n\|_{L^{p^*}(\mm)}}{\||D u_n|\|_{L^p(\mm)}}\ge  \eucl(N,p).
     \]
\end{lemma}
\begin{proof}
Let $v:[0,\infty)\to [0,\infty)$, $v \in C^\infty(0,\infty)$, be an extremal function for the Bliss inequality \eqref{eq:bliss} as given by Lemma \ref{lem:bliss}. It can be easily shown that we can approximate $v$ with functions $v_n \in \LIP_c([0,\infty))$ so that $\|v_n\|_{L^{p^*}(h_N\Leb 1)}\to \|v\|_{L^{p^*}(h_Nn\Leb 1)}$ and $\|v_n'\|_{L^{p}(h_N\Leb 1)}\to \|v'\|_{L^{p}(h_N\Leb 1)}$, where $h_N\, \Leb 1=\sigma_{N-1}t^{N-1}\, \Leb 1.$ For example we can take $v_n\coloneqq \phi_n(u_b)$ with $\phi_n \in\LIP[0,\infty)$, $ \phi_n\ge 0$, $\phi_n(t)\le |t|$, $\Lip (\phi_n)\le 2$,  $\phi_n(t)=t$ in $[2/n,\infty)$ and $\supp(\phi_n)\subset[1/n,\infty).$ The claimed approximation of the norms then follows immediately from the fact that $v$ is decreasing and vanishing at infinity. Therefore we have
\begin{equation}\label{eq:extremal sequence line}
    \lim_n \frac{\|v_n\|_{L^{p^*}(h_n\Leb 1)}}{\|v_n'\|_{L^{p}(h_n\Leb 1)}}= \eucl(N,p).
\end{equation}
We can now define $u_n\coloneqq v_n\circ \sfd_{x_0}$, where $\sfd_{x_0}(\cdot)\coloneqq\sfd(x_0,\cdot)$. We clearly have that $u_n \in \LIP_c(\X)$ and from the chain rule also that $|D u_n|=|v_n'|\circ \sfd_{x_0} |D\sfd_{x_0}|\le |v_n'|\circ \sfd_{x_0}$ $\mm$-a.e., since $\sfd_{x_0}$ is 1-Lipschitz. Hence applying Lemma \ref{lem:volume cone} we obtain $\|u_n\|_{L^{p^*}(\mm)}=\|v_n\|_{L^{p^*}(h_N\Leb 1)}$ and $\||D u_n|\|_{L^p(\mm)}\le \|v_n'\|_{L^{p}(h_N\Leb 1)}$. This combined with \eqref{eq:extremal sequence line} (up to passing to a subsequence) gives the conclusion.
\end{proof}

\begin{theorem}[Lower bound on the Sobolev constant]\label{thm:A lower bound}
	Let $\Xdm$ be a $\cd(K,N)$ space, $K \in \mathbb{R},\, N\in (1,\infty)$ that supports a $(p^*,p)$-Sobolev inequality for $p\in(1,N)$  with constants $A,B,$ where $p^*=pN/(N-p).$  Then 
	\begin{equation}\label{eq:A lower bound}
		A\ge\frac{{\rm Eucl}(N,p)^p}{\theta_N(x)^{\frac pN}}, \qquad \forall \, x \in \X.
	\end{equation}
\end{theorem}
\begin{proof} If $\theta_N(x) =\infty$, there is nothing to prove. Hence we can assume that $\theta_N(x)<+\infty$. From the compactness and stability of the $\cd(K,N)$ condition, there exists a sequence $r_i \to 0$ such that $\X_i\coloneqq(\X,\sfd/r_i,\mm/{r_i}^N,x)$ pmGH-converge to a $\cd(0,N)$ space $(\Y,\sfd_\Y,\mm_\Y,\oo_\Y)$. Moreover, from \eqref{eq:scaling} we have that $\X_i$ supports a $(p^*,p)$-Sobolev inequality with constants $A, r_i^pB$. This combined with Lemma \ref{lem:sobolev stability} shows that $(\Y,\sfd_\Y,\mm_\Y)$ supports a $(p^*,p)$-Sobolev inequality with constants $A,0$. However we clearly have that $\mm_\Y$ satisfies $\frac{\mm_\Y(B_r(\soo_\Y))}{\omega_Nr^N}=\theta_N(x)$ for every $r>0.$ Therefore Lemma \ref{lem:extremal sequence cone}, after a rescaling, ensures that 
$A\ge \frac{{\rm Eucl}(N,p)^p}{\theta_N(x)^{\frac pN}},$ which is what we wanted.
\end{proof}

The above,  together with Theorem \ref{thm:Amin}, proves our main result Theorem \ref{thm:alfa} concerning $\alpha_p(\X)$.

Using Theorem \ref{thm:A lower bound} we can also prove the topological rigidity of the Sobolev inequality on non-collapsed $\RCD$ spaces. More precisely combining the volume rigidity for non-collapsed $\RCD$ spaces  (\cite[Theorem 1.6]{GDP17}) and the Cheeger-Colding's metric Reifenberg's theorem (\cite[Theorem A.1.2]{Cheeger-Colding97I})  (see also \cite{KapMon21}) we can obtain the following result.
\begin{corollary}[Manifold-regularity from almost Euclidean-Sobolev inequality]\label{cor:top compact}
    For every $K \in \R$, $N \in \N$, $p \in(1,N)$, $\alpha \in(0,1)$, $\eps>0$ there exists $\delta=\delta(K,N,\eps,\alpha)$ such that the following holds. Suppose that $(\X,\sfd,\Haus{N})$ is a compact $\RCD(K,N)$ space satisfying the following Sobolev inequality
    \begin{equation}\label{eq:top rigidity compact}
            \|u\|_{L^{p^*}(\Haus{N})}^p\le (\eucl(N,p)^p+\delta)\||Du|\|_{L^p(\Haus{N})}^p+B \|u\|_{L^{p}(\Haus{N})}^p, \quad \forall \, u \in W^{1,p}(\X),
    \end{equation}
    for some constant $B>0$, where $p^*\coloneqq pN/(N-p)$.

    Then, there exists a smooth $N$-dimensional Riemannian manifold $M$ and an $\alpha$-biH\"older homeomorphism $F: M \to \X.$
\end{corollary}
%Note that to deduce the above result we do not need the full Theorem \ref{thm:alfa}, but is actually sufficient to have a lower bound on $\alpha_p(\X)$ (see Section \ref{sec:lower bound alfa}).

\begin{proof}
The argument is analogous to \cite[Theorem 3.1]{KapMon21}, however for completeness we include the details.

We start fixing $\eps>0$, $N\in \N$, $K \in \R$, $p \in(1,N)$ and two numbers $\bar \delta=\bar \delta(K,N,p,\eps)>0$ $\bar r=\bar r(K,N,p,\eps)$ small enough to be chosen later.

    Suppose that $(\X,\sfd,\Haus{N})$ is a compact $\RCD(K,N)$ space that supports a $(p^*,p)$-Sobolev inequality with constant $\eucl(N,p)^p+\delta,B$, for some $\delta\le \bar \delta$ and $B>0$ (i.e. such that \eqref{eq:top rigidity compact} holds). Then from \eqref{eq:A lower bound}, if $\bar \delta\le \eucl(N,p)^p/4,$ we have that 
    \[
    \theta_N(x)\ge 1-2\delta, \quad \forall x \in \X.
    \]
    Therefore for every $x \in \X$ there exists $r_x\in(0,\bar r)$ such that $\Haus{N}(B_{r_x}(x))\ge (1-3\delta)r_x^N\omega_N.$ Moreover from the Bishop-Gromov inequality,  for every $y \in B_{\delta r_x}(x)$ and every $s \in (0,r_x)$  it holds that
    \begin{equation}\label{eq:bghaus}
        \frac{\Haus{N}(B_{s}(y))}{v_{K,N}(s)} \ge    \frac{\Haus{N}(B_{(1+\delta)r_x}(y))}{v_{K,N}((1+\delta)r_x)} \ge \frac{\Haus{N}(B_{r_x}(x))}{v_{K,N}((1+\delta)r_x)} \ge \frac{(1-3\delta)r_x^N\omega_N}{v_{K,N}((1+\delta)r_x)}.
    \end{equation}
    Recalling that $\lim_{r \to 0^+} \frac{\omega_N r^N}{v_{K,N}(r)}=1$, from  \eqref{eq:bghaus} we deduce that if both $\bar r$ and $\bar \delta$ are small  enough, with respect to $K,N,p,\eps$, then 
    \[
     \Haus{N}(B_{s}(y)) \ge (1-\eps)s^N\omega_N, \quad \forall y \in B_{r_x}(x), \, s \in(0,r_x).
    \]
    Finally from the compactness of $\X$ there exists a finite number of points $x_i$, $i=1,...,m$ such that $\X \subset \cup_i B_{r_{x_i}}(x_i)$. Taking $R\coloneqq \min_i r_{x_i}<\bar r$ we then have 
     \[
     \Haus{N}(B_{s}(y)) \ge (1-\eps)s^N\omega_N, \quad \forall y \in \X, \, s \in(0,R).
    \]
    From this the conclusion  follows combining the volume rigidity theorem for non-collapsed $\RCD$ spaces (\cite[Theorem 1.6]{GDP17}) and the intrinsic metric-Reifenberg's theorem (\cite[Theorem A.1.2]{Cheeger-Colding97I}).
\end{proof}

\subsection{Sharp  and rigid Sobolev inequalities under Euclidean volume growth}\label{sec:sharp sobolev}
Here we prove the sharp Sobolev inequalities on $\cd(0,N)$ spaces contained Theorem \ref{thm:sharpSobgrowth}. The validity of the inequality \eqref{eq:sharp sobolev growth} will be derived as a  consequence of the local-Sobolev inequalities in Theorem \ref{thm:local sobolev intro}. The sharpness instead follows from a well known principle for which the validity of a Euclidean-Sobolev inequality implies certain growth on the measure of balls. In particular we have the following result:
\begin{theorem}\label{thm:sobolev implies growth}
	Let $\Xdm$ be an $\cd(0,N)$, $N \in(1,\infty)$ such that for some $p\in(1,N)$ and $A>0$
	\begin{equation}\label{eq:euclidean sobolevA}
	   	\|u\|_{L^{p^*}(\mm)}\le A \||Du|\|_{L^p(\mm)}, \qquad \forall u \in \LIP_c(\X),
	\end{equation}
	where $p^*\coloneqq\frac{pN}{N-p}$. Then  $\X$ has Euclidean volume-growth and 
	\begin{equation}\label{eq:A lower bound avr}
		{\sf AVR}(\X)\ge \Big(\frac{\eucl(N,p)}{A}\Big)^N.
	\end{equation}
\end{theorem}
On the general setting of $\cd$ spaces Theorem \ref{thm:sobolev implies growth} is proved in \cite{Kristaly16} (see also \cite{KO13} for the case $p=2$), extending to non-smooth setting the same results for Riemannian manifolds due to  Ledoux \cite{Ledoux99} and improved by Xia \cite{Xia01}.
We mention also \cite{DCX04} and \cite{Xia05} for analogous statements related to different class of inequalities.  In all the cited works the arguments depend on rather intricate ODE-comparison (originated in \cite{Ledoux99} and inspired by the previous \cite{BL96})  and heavily rely on the explicit knowledge of the extremal functions for the inequalities. However, using the results in Section \ref{sec:lower bound alfa} we are able to give a short proof of Theorem \ref{thm:sobolev implies growth}, which uses  a more direct blow-down procedure, that we believe being interesting on its own. The main advantage of this approach is that we will never need, as opposed to the ODE-comparison approach, the explicit expression of extremals functions in the Euclidean Sobolev inequality \eqref{eq:EuclSob}.
\begin{proof}[Proof of Theorem \ref{thm:sobolev implies growth}]
	The fact that $\mm(\X)=+\infty$ can be immediately seen by plugging in the Sobolev inequality functions  $u_R \in\LIP_c(\X)$ so that $u_R=1$ in $B_R(x_0)$ $\supp(u_R)\subset B_{2R}(x_0)$ and $\Lip(u_R)\le 1/R$ and sending $R\to +\infty.$ The fact that $\X$ has Euclidean volume growth follows by considering instead functions $u_R(\cdot)\coloneqq (R-\sfd_{x_0}(\cdot))^+$ as $R\to +\infty$ with fixed $x_0 \in \X$ and using the Bishop-Gromov inequality.
	
	\sloppy It remains to prove \eqref{eq:A lower bound avr}. We argue via blow-down. Let $R_i \to +\infty$. From the Euclidean volume-growth property, up to passing to a non relabeled subsequence, the rescaled spaces $(\X,\sfd/R_i,\mm/R_i^N,x_0)$, $x_0 \in \X$, pmGH-converge to an $\cd(0,N)$ space $(\Y,\sfd_\Y,\mm_\Y,\oo_\Y)$ satisfying $\frac{\mm_\Y(B_R(\soo_\Y))}{\omega_N r^N}={\rm AVR}(\X).$  Moreover combining \eqref{eq:euclidean sobolevA} with Lemma \ref{lem:sobolev stability} proves that $Y$ satisfy a $(p^*,p)$-Sobolev inequality with constants $A,0.$ Then \eqref{eq:A lower bound avr} follows from Lemma \ref{lem:extremal sequence cone}.
\end{proof}

We can now move to the proof of the sharp Sobolev inequalities under the Euclidean volume growth assumption. 
\begin{proof}[Proof of Theorem \ref{thm:sharpSobgrowth}]
Fix $x \in \X$. From the definition of ${\sf AVR}(\X)$, for every $r$ big enough $\theta_{N,r}(x)\le 2{\sf AVR}(\X).$ Fix one of such $r>0.$ From the Bishop-Gromov inequality we also have that $\theta_{N,R}(x)\ge {\sf AVR}(\X)$ for every $R>0$. In particular $\theta_{N,r}(x)/\theta_{N,R}(x)\le 2$ for every $R>0.$ Hence by  Theorem \ref{thm:local sobolev intro} (for $K=0$) we have that for every $\eps>0$, there exists $\delta=\delta(\eps)>0$ so that  for every $ R>r/\delta$ the following local Euclidean Sobolev inequality holds:
\[
		\|u\|_{L^{p^*}(\mm)}\le (1+\eps)\eucl(N,p) \theta_{N,R}(x)^{-\frac{1}{N}}\\||D u|\|_{L^p(\mm)}, \qquad \forall u \in \LIP_c(B_{r}(x)).
\]
Taking $R\to \infty$ we achieve 
\[
		\|u\|_{L^{p^*}(\mm)}\le (1+\eps)\eucl(N,p) {\sf AVR}(\X)^{-\frac{1}{N}}\\||D u|\|_{L^p(\mm)}, \qquad \forall u \in \LIP_c(B_{r}(x)).
\]
Since $\eps$ was chosen arbitrarily and independent of $r>0$, we can first send $\eps \to 0^+$ and then $r\to +\infty$ to achieve the first part of the statement. 

The sharpness of \eqref{eq:sharp sobolev growth} instead follows immediately from Theorem \ref{thm:sobolev implies growth}.
\end{proof}

\section{The constant $\aopt_q$ in metric measure spaces}\label{sec:Aopt}
In this section we will prove some upper and lower bounds on $\aopt_q$ in the case of metric measure spaces. Some of the results contained here (more precisely, Section \ref{Secdiameterbounds}) are actually not used  in other parts of the note, however we chose to include them here for completeness and to give a more clear picture around the value of $\aopt_q$. Let us also remark that the results of this part are valid for a general lower bound $K\in \R$.

We start recalling the definition of $\aopt_q$. In this section we assume that $\Xdm$ is a metric measure space with $\mm(\X)=1$. For every $q \in(2,+\infty)$ we define $\aopt_q(\X) \in[0,+\infty]$ as the minimal constant satisfying
\begin{equation}\label{eq:critical sobolev body}
	\|u\|_{L^{q}(\mm)}^2\le \aopt_q(\X)\,   \||D u|_2\|_{L^2(\mm)}^2 +  \|u\|_{L^2(\mm)}^2 , \quad \forall u \in W^{1,2}(\X),
\end{equation}
with the convention that $A\coloneqq +\infty$ if no such $A$ exists. Note that, since $\mm(\X)=1$, this is the same definition given right after \eqref{eq:critical sobolev}.  In the following sections we will prove three type of bounds on $\aopt_q(\X)$: an upper bound in the case of synthetic Ricci curvature and dimension bounds; a lower bound in terms of the first non-trivial eigenvalue; a lower bound related to the diameter. 

\subsection{Upper bound on $\aopt_q$ in terms of Ricci  bounds}\label{sec:upper bound A}
Here we prove a generalization to the non-smooth setting of a well known estimate on $\aopt_q$ valid on manifolds (recall \eqref{eq:upper bound curvature manifold}). The two key ingredients for the proof are the Sobolev-Poincar\'e inequality and an inequality due to Bakry:

\begin{proposition}\label{prop:sobolev embedding general}
For every $K \in \R$, $N\in(2,\infty)$ and $D>0$ there exists a constant $A=A(K,N,D)>0$ such that the following holds.
    Let $(\X,\sfd,\mm)$ be a compact  $\cd(K,N)$ space with $N \in(1,\infty)$, $K\in \R$, $\mm(\X)=1$ and $\diam(\X)\le D$. Then for every $q\in(2,2^*]$ we have
	\begin{equation}\label{eq:tight sobolev}
	\|u\|_{L^q(\mm)}^2\le A\||D u|\|_{L^2(\mm)}^2+\|u\|_{L^2(\mm)}^2, \qquad \forall u \in W^{1,2}(\X)
	\end{equation}
	and in particular $\aopt_q(\X)\le A(K,N,D).$
\end{proposition}
\begin{proof}
   The proof is based on the following inequality: for every $q\in (2,\infty)$ 
   \begin{equation}\label{eq:bakry}
       \Big(\int |u|^q\, \d \mm\Big)^{2/q}\le (u_\X)^2+(q-1)\Big(\int |u-u_{\X}|^q\, \d \mm\Big)^{2/q} \qquad \forall u \in L^{q}(\mm),
   \end{equation}
   where $u_{\X}=\int u\, \d \mm.$ See (\cite{Bakry94} or \cite[Prop. 6.2.2]{BGL14} ) for a proof of this fact. Then \eqref{eq:tight sobolev} follows combining \eqref{eq:bakry} with \eqref{eq:improved poincaret} and the Jensen inequality.
\end{proof}
Recall that for $K>0$ an explicit and sharp upper bound on $\aopt_q$ exists and has been proven in \cite{CM17} (see Theorem \ref{thm:Aopt upper cd}). The argument in \cite{CM17} relies on the powerful \emph{localization  technique}. However, it is worth to point out that Theorem \ref{thm:Aopt upper cd} can also be deduced from the P\'olya-Szeg\H{o} inequality proved in \cite{MS20} (see  Theorem \ref{thm:KNpolya}) and the Sobolev inequality on the model space \eqref{eq:sobolev model}.

\subsection{Lower bound on $\aopt_q$ in terms of the first eigenvalue}\label{sec:lower bound aopt}
It is well known that a ``tight-Sobolev inequality" as in \eqref{eq:critical sobolev body} (i.e. with a constant 1 in front of $\|u\|_{L^2}$ when $\X$ is normalized with unit volume) implies a Poincar\'e-inequality (see e.g. \cite[Prop. 6.2.2]{BGL14}). This can be rephrased as a lower bound on $\aopt_q$ in terms of the first non-trivial  eigenvalue:

\begin{proposition}\label{prop:aopt2gapbound}
Let $\Xdm$ be a metric measure space with $\mm(\X)=1$.  Then for every $q \in(2,+\infty)$ it holds 
\begin{equation}
 \aopt_q(\X)\ge \frac{q-2}{\lambda^{1,2}(\X)},\label{eq:Aopt2gap}\end{equation}
 (meaning that if $\lambda^{1,2}(\X)=0,$ then $ \aopt_q(\X)=+\infty).$
\end{proposition}
We will give a detailed proof of this result, which amounts to a linearization procedure. Indeed   a refinement of the same argument will also play a key role on the rigidity and almost-rigidity results in the sequel (see Section \ref{sec:linearization}). 

We start with an elementary linearization-Lemma.
\begin{lemma}\label{lem:technical linearization}
    Let $\Xdm$ be a metric measure space with $\mm(\X)=1$ and fix $q\in(2,\infty)$. Let $f \in L^2\cap L^q(\mm)$ with $\int f\d \mm=0$. Then
    \begin{equation}\label{eq:technical linearization}
    \begin{split}
         \Big |	\Big( \int |1+f|^q\, \d \mm \Big)^{2/q} -&\int (1+f)^2\, \d \mm -(q-2)  \int |f|^2\, \d \mm \Big|\\
         &\le C_q \Big(\int |f|^{3\wedge q}+|f|^q\, \d \mm  +\Big( \int |f|^q\, \d \mm\Big)^2+\Big(\int |f|^2\, \d \mm\Big)^2\Big),
    \end{split}
    \end{equation}
where $C_q$ is a constant depending only on $q$.
\end{lemma}
\begin{proof}
     We start defining $I\coloneqq\int |1+f|^q\, \d \mm  -1$ and observe that 
\begin{equation}\label{eq:expansion 1}
		\Big |\Big( \int |1+f|^q\, \d \mm \Big)^{2/q}-1-\frac{2}{q}I\Big |\le c_q |I|^2,
\end{equation}
which follows from the inequality $||1+t|^{2/q}-1-2t/q|\le c_qt^2$, $t \ge 0.$
It remains to investigate the behavior of $I.$ Exploiting the inequality $||1+t|^q-1-qt|\le \tilde c_q(|t|^2+|t|^q)$, $t \ge 0$, and the fact that $f$ has zero mean we have the following simple bound
\begin{equation}\label{eq:expansion 3}
	|I|\le \tilde c_q \int |f|^2+|f|^q\, \d \mm.
\end{equation}
We will also need a more precise estimate of $I$, which will follow from the following inequality
\begin{equation}\label{eq:taylor}
	\Big||1+t|^q-1-qt-\frac{q(q-1)}{2}t^2 \Big|\le C_q(|t|^{3\wedge q}+|t|^q), \quad \forall t\in \R,
\end{equation}
that can be seen using Taylor expansion when $|t|\le 1/2$ and elementary estimates in the case $|t|\ge1/2$. Using \eqref{eq:taylor} we obtain that
\begin{equation*}
	\Big |I-\int qf+\frac{q(q-1)}{2}|f|^2\, \d \mm  \Big |\le  C_q\int |f|^{3\wedge q}+|f|^q\, \d \mm
\end{equation*}
and since we are assuming that $f$ has zero mean, we deduce 
\begin{equation}\label{eq:expansion 2}
	\Big |I-\frac{q(q-1)}{2} \int |f|^2\, \d \mm  \Big |\le  C_q\int |f|^{3\wedge q}+|f|^q\, \d \mm.
\end{equation}
Combining \eqref{eq:expansion 1}, \eqref{eq:expansion 3} and \eqref{eq:expansion 2}, noting that $\int(1+f)^2\, \d\mm=1+\int f^2\, \d \mm,$ we deduce   \eqref{eq:technical linearization}.
\end{proof}
Exploiting the above linearization, we can now prove the  lower bound on $\aopt_q$ in terms of the first eigenvalue.
\begin{proof}[Proof of Proposition \ref{prop:aopt2gapbound}]
If $\aopt_q(\X)=+\infty$ there is nothing to  prove, hence we assume that $\aopt_q(\X)<+\infty.$ Let $f \in \LIP(\X)\cap L^2(\mm)$ with $\int f\, \d \mm=0$ and $\|f\|_{L^2(\mm)}=1$. Observe also that, since  $\aopt_q(\X)<+\infty$, $f \in L^q(\X)$. Therefore applying \eqref{eq:technical linearization}  we obtain
\[
	\Big( \int |1+\eps f|^q\, \d \mm \Big)^{2/q} -\int (1+\eps f)^2\,\d \mm -(q-2)  \int |\eps f|^2\, \d \mm =o(\eps^2),
\]
which combined with \eqref{eq:critical sobolev body} gives
\[
\aopt_q(\X) \eps^2 \int |D f|_2^2\, \d \mm -(q-2)  \int |\eps f|^2\, \d \mm \ge o(\eps^2).
\]
Dividing by $\eps^2$ and sending $\eps \to 0$ gives  that $\lambda^{1,2}(\X)\ge \frac{q-2}{\aopt_q(\X)}$, which  concludes the proof.
\end{proof}

\subsection{Lower bound on $\aopt_q$ in terms of the diameter}\label{Secdiameterbounds}
We start recalling the following result, which was proved in \cite{BL96} in the context of Markov-triple and which proof works with straightforward modifications  also in the setting of metric measure spaces (see also \cite{Hebey99} for an exposition of the argument on Riemannian manifolds). For this reason we shall omit its proof. We stress that, since this result and its consequences are used only on this section, the exposition of the rest of the note remains self-contained.
\begin{theorem}\label{thm:5.5}
Let $q\in(2,\infty)$ and define $N(q) \coloneqq \frac{2q}{q-2}$. Let $\Xdm$ be a compact metric measure with $\diam(\X)= \pi$, $\mm(\X)=1$ and suppose that
\begin{equation}\label{eq:sobolev Ledoux}
    \|u\|_{L^q(\mm)}\le \frac{q-2}{N(q)}\||D u|\|_{L^2}^2+\|u\|_{L^2(\mm)}^2, \quad \forall u \in W^{1,2}(\X).
\end{equation}
Then there exists a non-constant function $f \in \LIP(\X)$ realizing equality in \eqref{eq:sobolev Ledoux}.
\end{theorem}
Note that  $q=2N(q)/(N(q)-2)$, so that in a sense $``q=2^*(N(q))"$. With Theorem \ref{thm:5.5} we can now prove the following lower bound on $\aopt_q(\X)$. The proof uses a scaling argument due to Hebey \cite[Proposition 5.11]{Hebey99}.
\begin{proposition}\label{prop:lowbound}
Let $\Xdm$ be a compact metric measure space with $\mm(\X)=1$ and  \emph{diam}$(\X) \le \pi$. Then for every $q \in(2,\infty)$ it holds
\begin{equation} A^{\rm opt}_{q}(\X) \ge \left (\frac{\emph{diam}(\X)}{\pi}\right)^2 \frac{q-2}{N(q)}\label{eq:lowerAoptdiam},\end{equation}
where $N(q)= \frac{2q}{q-2}$.
\end{proposition}
\begin{proof}
Set $D:=\diam(\X)$ and, by contradiction, suppose that $A^{\rm opt}_{q}(\X)< (\tfrac{D}{\pi})^2 \tfrac{q-2}{N(q)}$.  Define the scaled metric measure space 
\[ (\X',\sfd',\mm') := (\X,\tfrac{1}{D/\pi}\sfd, \mm ).\]
It can be directly checked that $\X'$ satisfies the hypotheses of Theorem \ref{thm:5.5}. Hence there exists a non-constant function $u \in \LIP(\X)$ satisfying \eqref{eq:sobolev Ledoux} with equality (in the space $\X'$), which rewritten on the the original space $\X$ reads as
 \[ \|u\|_{L^q(\mm)}= \left (\frac{D}{\pi}\right)^2\frac{q-2}{N(q)}\||D u|\|_{L^2(\mm)}^2+\|u\|_{L^2(\mm)}^2,\]
 which however contradicts the assumption $A^{\rm opt}_{2^*}(\X)< (\tfrac{D}{\pi})^2 \tfrac{q-2}{N(q)}$.
\end{proof}

\begin{remark} \rm
Arguing exactly as in \cite{BL96}, it is possible to prove that under the assumptions of Theorem \ref{thm:5.5} and assuming $\X$ to be also infinitesimal Hilbertian, there exists a function satisfying $\Delta u=N(q)u$. From this, it directly follows that equality in \eqref{eq:lowerAoptdiam} (in the case of an Infinitesimally Hilbertian space) implies the existence of a function satisfying:
\[
\Delta u= \left(\frac{\pi}{\diam(\X)}\right)^2N(q)u.
\]
Since this is not relevant in the present note, we will not  provide the details of such result.
\fr 
\end{remark}

%#####################################################################################################

\section{Rigidity of $\aopt_q$}\label{sec:rigidity}

\subsection{Concentration Compactness}\label{sec:CC}

In this section we assume that $(\X_n,\sfd_n,\mm_n)$ is a sequence of compact ${\rm RCD}(K,N)$ spaces, for some fixed $K\in \R,N \in (2,\infty)$, which converges in mGH-topology to a compact  ${\rm RCD}(K,N)$ space $(\X_\infty,\sfd_\infty,\mm_\infty)$. We will also adopt the extrinsic approach \cite{GMS15} identifying $\X_n,\X_\infty$ as subset of a common compact metric space $(\Z,\sfd_\Z)$, with $\supp(\mm_n)=\X_n,$ $\supp(\mm_\infty)=\X_\infty$, $\mm_n\rightharpoonup \mm_\infty$ in duality with $C_b(\Z)$ and $X_n \to \X_\infty$ in the Hausdorff topology of $\Z$. To lighten the discussion, we shall not recall in the following statements these facts and assume $(\X_n,\sfd_n,\mm_n)$, $n \in \bar \N =\N \cup \{\infty\}$  and $(\Z,\sfd)$ to be fixed as just explained. Also, we will set $2^* \coloneqq 2N/(N-2)$ without recalling its expression in the statements.

Our  main goal then is to prove the following dichotomy for the behavior of extremizing sequence for the Sobolev inequalities, on varying metric measure spaces.
\begin{theorem}[Concentration-compactness for Sobolev-extremals]\label{thm:CC_Sob}
	Suppose that $\mm_n(\X_n)$, $\mm_\infty(\X_\infty)=1$ and that $\X_n$ supports a $(2^*,2)$-Sobolev inequality
	\[ \|u\|^2_{L^{2^*}(\mm_n)} \le A\| |D u|\|^2_{L^2(\mm_n)} + B\|  u \|^2_{L^2(\mm_n)},\qquad \forall u \in W^{1,2}(\X_n), \]
	for some constants $A,B>0.$
	Suppose that $u_n \in W^{1,2}(\X_n)$ is a sequence of non-zero functions satisfying 
	\[ \|u_n\|^2_{L^{2^*}(\mm_n)} \ge A_n\| |D u_n|\|^2_{L^2(\mm_n)} + B_n\|  u_n \|^2_{L^2(\mm_n)},\]
	for some sequences $A_n \to A$, $B_n\to B$.

	Then, setting $\tilde u_n\coloneqq u_n\|u_n\|_{L^{2^*}(\mm_n)}^{-1}$, there exists a non relabeled subsequence such that only one of the following holds:
	\begin{itemize}
		\item[\textsc{I)}]\label{item:comp} $\tilde u_{n}$ converges $L^{2^*}$-strong to a function $u_\infty \in W^{1,2}(\X_\infty)$;
		\item[\textsc{II)}]\label{item:conc} $\|\tilde u_{n}\|_{L^2(\mm_n)}\to 0$ and there exists $x_0 \in \X_\infty$ so that $|u_n|^{2^*}\mm_n\weakto \delta_{x_0}$ in duality with $C_b(\Z)$.
	\end{itemize}
\end{theorem}
The principle behind the concentration compactness technique is very general and was originated in \cite{Lions85,Lions84}. In our case, since we will work in  a compact setting, the lack of compactness is formally due to dilations or rescalings (and not to translations) and the fact that we deal with the critical exponent in the Sobolev embedding. The main idea behind the principle is  first to prove that in general the failure of compactness can only be realized by concentration on a countable number of points. The second step is then to exploit a strict sub-additivity property of the minimization problem to show that either we have full concentration  at a single point or we do not have concentration at all and thus compactness. 

We start by proving necessary results towards the proof of Theorem \ref{thm:CC_Sob}.

A variant of the following appears also in \cite[Prop. 3.27]{Honda15}. For the sake of completeness, we provide here a complete proof.
\begin{proposition}\label{prop:coupling}
	 Let $p,q \in(1,\infty)$ with $\frac1p+\frac1q=1.$ Suppose that $u_n$  converges $L^q$-strong to $u_\infty$ and that $v_n$ converges $L^{p}$-weak to $v_\infty$, then
	\[
	\lim_{n\to \infty }\int u_nv_n\, \d \mm_n = \int u_\infty v_\infty\, \d\mm_\infty.
	\]
\end{proposition}
\begin{proof}
	It is sufficient to consider the case $u_n\ge 0 ,u_\infty \ge 0$, then the conclusion will follow recalling that  $u_n^+\to u_\infty^+$, $u_n^-\to u_\infty^-$ strongly in $L^q$.
	
	The argument is  similar to the one for the case $p=2$ (see, e.g., in \cite{AH17}), except that we need to consider the functions $u_n^{q/p}+tv_n$, $t \in\R.$ Observe first that $u_n^{p/q} \to u_\infty^{q/p}$ strongly in $L^p$ (by $(vii)$ of Prop. \ref{prop:lp prop}). In particular $u_n^{q/p}+tv_n$ converges to $u_\infty^{q/p}+tv_\infty$ weakly in $L^p$ and in particular from $iii)$ of Prop. \ref{prop:lp prop} we have
	\begin{equation}\label{eq:liminf norm}
	\|u_\infty^{q/p}+tv_\infty\|_{L^p(\mm_\infty) }\le \liminf_{n}\|u_n^{q/p}+tv_n\|_{L^p(\mm_n)}.
	\end{equation}
    The second ingredient is the following inequality
    \begin{equation}\label{eq:taylor2}
    \left ||a+|b||^p-|b|^p-pa|b|^{p-1} \right |\le C_p(|a|^{p\wedge 2} |b|^{p-p\wedge 2}+|a|^p), \quad \forall a,b \in \R,
    \end{equation}
    which is easily derived from $ \big||1+t|^p-1-pt\big|\le C_p(|t|^{p\wedge 2} +|t|^p),$ $\forall t \in \R$. Combining \eqref{eq:taylor2} and \eqref{eq:liminf norm} we have
    \begin{align*}
    	&\int |u_\infty|^q\, \d \mm_\infty+pt\int u_\infty v_\infty \, \d \mm_\infty-C_p t^{p\wedge 2} \int |v_\infty|^{p\wedge 2} |u_\infty^{q/p}|^{p-p\wedge 2}\, \d \mm_\infty-C_p t^{p} \int |v_\infty|^{p}\, \d \mm_\infty\\
    	&\le  \|u_\infty^{q/p}+tv_\infty \|_{L^p(\mm_\infty)}^p\le 	\liminf_{n}\|u_n^{q/p}+tv_n\|_{L^p(\mm_n)}\\
    	&\le \liminf_{n} \int |u_n|^q\, \d \mm_n+pt\int u_nv_n\, \d \mm_n+C_p t^{p\wedge 2} \int |v_n|^{p\wedge 2} |u_n^{q/p}|^{p-p\wedge 2}\, \d \mm_n+C_p t^{p} \int |v_n|^{p}\, \d \mm_n
    \end{align*}
Observe that in the case $p<2$ we have
\[
 \limsup_n \int |v_n|^{p\wedge 2} |u_n^{q/p}|^{p-p\wedge 2}\, =\limsup_n \int |v_n|^p \, \d \mm_n <+\infty,
\]
while for $p\ge 2$ using the H\"older inequality 
\[
 \limsup_n \int |v_n|^{p\wedge 2} |u_n^{q/p}|^{p-p\wedge 2}\, \le \limsup_n  \|v_n\|^2_{L^p(\mm_n)}\|u_n\|^{q(p-2)/p}_{L^q(\mm_n)}<+\infty.
\]
In particular, recalling that $\int |u_n|^q \d \mm_n \to \int |u_\infty|^q\d \mm_\infty$ and choosing first $t\downarrow 0$  and then $t\uparrow 0$ above we obtain the desired conclusion.
\end{proof}

The following is a version for varying-measure of the famous Brezis-Lieb Lemma \cite{BrezisLieb83}. The key difference with the classical version of this result, is that  in our setting  it does not makes sense to write ``$|u_\infty-u_n|$'', since $u_\infty$ and $u_n$ will be integrated with respect to different measures. Hence we need to replace this term in \eqref{eq:brezislieb} with $|v_n-u_n|$, where $v_n$ is sequence approximating $u_\infty$ in a strong sense.
\begin{lemma}[Brezis-Lieb type Lemma] \label{lem:BreLieb} 
	Suppose that $\mm_n(\X_n),\mm_\infty(\X_\infty)=1$, let $q\in[2,\infty)$ and $q'\in(1,q).$ Suppose that $u_n \in L^q(\mm_n)$ satisfy $\sup_n \|u_n\|_{L^q(\mm_n)}<+\infty$ and that $u_n$ converges to $u_\infty$ strongly in $L^{q'}$ to some $u_\infty\in L^{q'}\cap L^{q}(\mm_\infty)$. Then for any sequence $v_n \in L^q(\mm_n)$ such that $v_n \to u_\infty$ strongly both in $L^{q'}$ and $L^{q}$, it holds 
	\begin{equation}\label{eq:brezislieb}
	    \lim_{n\to \infty }\int |u_n|^q\, \d \mm_n-\int |u_n-v_n|^q\, \d \mm_n=\int |u_\infty|^q\,\d \mm_\infty.
	\end{equation}
\end{lemma} 
\begin{proof}
	The proof  is based on the following inequality:
	\begin{equation}\label{eq:taylor3}
		\big||a+b|^q-|b|^q-|a|^q \big|\le C_p(|a||b|^{q-1}+|a|^{q-1}|b|), \quad \forall a,b \in \R.
	\end{equation}
Indeed, if $a=v_n-u_n$ and $b=v_n$, we get from the above
\begin{equation}\label{eq:brezis book}
\int \big| |u_n|^q-|v_n-u_n|^q-|v_n|^q \big|\, \d \mm_n \le C_q  \int |v_n-u_n||v_n|^{q-1}+|v_n-u_n|^{q-1}|v_n|\, \d \mm_n.
\end{equation}
Since  $\int |v_n|^q\, \d \mm_n \to \int |u_\infty|^q\, \d \mm_\infty$, to conclude it is sufficient to show that the right hand side of \eqref{eq:brezis book} vanishes as $n \to +\infty.$ We wish to apply Proposition \ref{prop:coupling}. It follows from our assumptions that $|v_n|\to |u_\infty|$ strongly in $L^q$ and $|v_n|^{q-1}\to |u_\infty |^{q-1}$ strongly in $L^p$, with $p\coloneqq q/(q-1)$ (recall Prop. \ref{prop:lp prop}). Hence it remains only to show that $|v_n-u_n| ,|v_n-u_n|^{q-1}$ converges to 0  weakly in $L^q$ and  weakly in $L^p$ respectively. We have that $\sup_n \|u_n-v_n\|_{L^q(\mm_n)}<+\infty,$ hence by $iv)$ in Prop. \ref{prop:lp prop} up to a subsequence $|u_n-v_n|$ converge weakly in $L^q$ to a function $w \in L^q(\mm)$. However by assumption the sequences $(v_n),(u_n)$ both converge  strongly in $L^{q'}$ to $u$, hence  $v_n-u_n\to 0$ strongly in $L^{q'}$ (recall  $ii)$ in Prop. \ref{prop:lp prop}) and in particular by from $i)$ of Prop. \ref{prop:lp prop} we have that $|v_n-u_n|\to 0$ strongly in $L^{q'}$, which  implies that $w=0.$ Analogously we also get that up to a subsequence  $|u_n-v_n|^{q-1}$ converge weakly in $L^p$ to a non-negative function $w'\in L^p(\mm)$. Suppose first that $q'\le q-1$. taking $t \in[0,1]$ such that $q-1=tq'+(1-t)q$ we have 
$$\int w'\, \d \mm_\infty=\lim_n \int |u_n-v_n|^{q-1}\, \d \mm_n \le \|v_n-u_n\|_{L^{q'}(\mm_n)}^{tq'}\|v_n-u_n\|_{L^{q}(\mm_n)}^{(1-t)q}\to0,$$
where we have used again that  $u_n-v_n\to 0$ strongly in $L^{q'}$ and that $u_n-v_n$ is uniformly bounded in $L^q.$ If instead $q'\ge q-1$ by H\"older inequality we have 
$$\int w'\, \d \mm_\infty=\lim_n \int|u_n-v_n|^{q-1}\, \d \mm_n\le \left(\int|u_n-v_n|^{q'}\, \d \mm_n\right)^{(q-1)/q'}\to0.$$
In both cases we deduce that $w'=0$, which concludes the proof.
\end{proof}

\begin{lemma}\label{lem:mixrecov}
    Let $q \in [2,\infty)$ and let $u_\infty \in W^{1,2}(\X_\infty)\cap L^{q}(\mm_\infty)$. Then, there exists a  sequence $u_n \in W^{1,2}(\X_n)\cap L^q(\X_n)$ that converges both $L^{q}$-strong and $W^{1,2}$-strong to $u_\infty$.
\end{lemma}
\begin{proof}
	By truncation and a diagonal argument we can assume that $u_\infty\in L^\infty(\mm_\infty).$ 
%    We define, for every $m \in \N$, the function $\phi_M \in \LIP(\R)$ via $\phi(t):= (t\wedge m) \vee -m $, observe that is $1$-Lipschitz and that $\phi_m(0) =0$. Define the truncated function $f_{\infty,m} := \phi_m\circ f_\infty$ which clearly satisfies $f_{\infty,m} \to f_\infty$ in $L^{p^*}(\mm_\infty)$ and in $W^{1,p}(\X_\infty)$. 
    By the $\Gamma$-$\limsup$ inequality of the $\rmCh_2$ energy there exists a sequence $v_{n} \in W^{1,2}(\X_n)$  converging strongly in $W^{1,2}$ to $u_{\infty}$. Defining $u_{n} :=  (v_{n}\wedge C) \vee -C $, with $C\ge \|u_\infty\|_{L^\infty(\mm_\infty)}$, we have by $(i)$ of Proposition \ref{prop:lp prop} that $u_{n}$ converges in $L^2$-strong to $u_{\infty}$. Moreover $|Du_n|\le |D v _n|$ $\mm_n$-a.e., therefore $\limsup_{n}\int |D u_n|^2\, \d \mm_n\le \limsup_{n}\int |D v_n|^2=\int |D u_\infty|^2\, \d \mm_\infty$, which ensures that $u_n$ converges also $W^{1,2}$-strongly to $u_\infty.$ 
    Finally, the sequence $u_{n}$ is uniformly bounded in $L^\infty$ and converges to $u_\infty$  in $L^2$-strong, hence  by $(viii)$ of Proposition \ref{prop:lp prop}. we have that that $u_{n}$ is also $L^{q}$-strongly convergent to $u_{\infty}$. 
\end{proof}

The  following statement is the analogous in metric measure spaces of  \cite[Lemma I.1]{Lions85}. We shall omit its proof since the arguments presented there in $\R^n$ extend to this setting with obvious modifications (see also Remark I.5 in \cite{Lions85}).
\begin{lemma}\label{lem:concompI}
Let $\Xdm$ be a metric measure space and $\mu,\nu \in \MM^+_b(\X)$. Suppose that
\[ \left(\int |\varphi|^{q}\,\d\nu\right)^{1/q}\le C\left(\int|\varphi|^p\, \d \mu\right)^{1/p}, \qquad \forall \varphi \in \LIP_b(\X), \]
for some $1\le p<q<+\infty$ and $C\ge 0.$ Then there exists a countable set of indices $J$, points $(x_j)_{j\in J}\subset \X$ and positive weights $(\nu_j)_{j\in J}\subset\R^+$ so that
        \begin{equation} \nu =\sum_{j\in J}\nu_j\delta_{x_j}, \quad \mu \ge C^{-p}\sum_{j\in J}\nu_j^{p/q}\delta_{x_j}. \label{eq:revHold}\end{equation}
\end{lemma}
%\begin{proof}
%We point out that, if $(ii)$ holds, then necessarily $\sum_{j\in J} \nu_j^{2/2^*} <\infty$. The proof is classical if $\X = \R^N$ and goes back to . In Remark I.5 it is there remarked that the proof works in an arbitrary measure space, as soon as one relax suitably the notion of \emph{atom}. Since metric measure spaces are complete metric spaces, atoms must be points whence the proof of \cite{Lions} yields the statement.
%\end{proof}
Next, we present a generalized Concentration-Compactness principle, with underlying varying ambient space.  For the sake of generality and for an application to  the Yamabe equation in Section \ref{sec:scalarPDE}, we will be working with a slightly more general Sobolev inequality containing an arbitrary $L^q$-norm (apart from Section \ref{sec:scalarPDE}, we will use this statement only with $q=2$). 
\begin{lemma}[Concentration-Compactness Lemma]\label{lem:conccompII}
  	Suppose that $\mm_n(\X_n),\mm_\infty(\X_\infty)=1$ and that for some fixed $q \in (1,\infty)$ the spaces $\X_n$ satisfy the following Sobolev-type inequalities 
	\begin{equation}\label{eq:disturbed sobolev}
		 \|u\|^2_{L^{2^*}(\mm_n)} \le A_n \| |Du|\|_{L^2(\mm_n)}^2 + B_n \| u \|^2_{L^q(\mm_n)}, \qquad \forall u \in W^{1,2}(\X_n),
	 \end{equation}
	with uniformly bounded positive constants $A_n,B_n$. Let also $u_n \in W^{1,2}(\X_n)$ be $W^{1,2}$-weak and both $L^2$-strong and $L^q$-strong converging to $u_\infty \in W^{1,2}(\X_\infty)$ and suppose that $|Du_n|^2 \mm_n \weakto \mu,$ $|u_n|^{2^*}\mm_n\weakto \nu$ in duality with $C_b(\Z)$ for two given measures $\mu,\nu \in \MM^+_b(\Z)$. 
   
    Then,
    \begin{itemize}
        \item[$(i)$] there exists a countable set of indices $J$, points $(x_j)_{j\in J}\subset \X_\infty$ and positive weights $(\nu_j)_{j\in J}\subset\R^+$ so that
        \[ \nu =|u_\infty|^{2^*}\mm_\infty +\sum_{j\in J}\nu_j\delta_{x_j};\]
        \item[$(ii)$] there exist $(\mu_j)_{j\in J} \subset \R^+$ satisfying $\nu_j^{2/2^*} \le (\limsup_{n}A_n)\mu_j$ and such that
        \[\mu \ge |D u_\infty|^2\mm_\infty +\sum_{j\in J}\mu_j\delta_{x_j}.\]
        In particular, we have $\sum_j \nu_j^{2/2^*}<\infty$.
    \end{itemize}
\end{lemma}
\begin{proof}
We subdivide the proof in two steps.

\noindent \textsc{Step 1}. We assume that $u_\infty=0$. Let $\varphi \in \LIP_b(\Z)$ and consider the sequence $(\varphi u_n)\in W^{1,2}(\X_n)$ which plugged in the Sobolev inequality for each $\X_n$ gives
\[ \left( \int |\varphi|^{2^*}|u_n|^{2^*}\, \d \mm_n \right)^{1/{2^*}} \le \left(A_n\int |D(\varphi u_n)|^2\, \d \mm_n +B_n \Big(\int |\varphi|^qu_n^q\,\d \mm_n \Big)^{2/q}\right)^{1/2} , \qquad \forall n\in \N.\]
It is clear that, by weak convergence, the left hand side of the inequality tends to $(\int |\varphi|^{2^*}\,\d \nu )^{1/2^*}$. While for the right hand side we discuss the two terms separately. First, by $L^q$-strong convergence, we have $\int \varphi^qu_n^q\,\d \mm_n \to 0$, while an an application of the Leibniz rule gives $\int |D(\varphi u_n)|\, \d \mm_n \le \int |D \varphi||u_n| + |\phi||D u_n|\, \d \mm_n$. Moreover again by strong convergence $\int |D \varphi|^2|u_n|^2\d \mm_n\to0$. Combining these observations we reach
\[  \left( \int |\varphi|^{2^*} \d \nu \right)^{1/2^*} \le \big( \limsup_n A_n\big )^{1/2}\left(\int |\varphi|^2\, \d \mu\right)^{1/2}, \qquad \forall \varphi \in \LIP_{b}(\Z).\]
Thus, Lemma \ref{lem:concompI} (applied in the space $(\Z,\sfd_\Z)$) gives $(i)$-$(ii)$, for the case $u_\infty=0$, except for the fact that we currently do no know whether the points $(x_j)_{j\in J}$ are in $\X_\infty$. This last simple fact can be seen as follows. Fix $j \in J$. From the  weak convergence  $|u_n|^{2^*}\mm_n \rightharpoonup\nu$, there must be a sequence $y_n \in \supp(\mm_n)=\X_n$ such that $\sfd_\Z(y_n,x_j)\to0.$ Then the GH-convergence of $\X_n$ to $\X_\infty$ ensures that $x_j \in \X_\infty$, which is what we wanted.
% Since $X_\infty=\supp(\mm_\infty)$ is closed there exists $\eps>0$ such that $B^\Z_{\eps}(x_j)\cap X_\infty=\emptyset.$ Fix $\nchi\in \LIP_c(B^\Z_{\eps}(x_j))$ such that $\nchi(x_j)=1$. By weak convergence, $\liminf_{n}\int \nchi |u_n|^{2^*}\, \d \mm_n \ge \nu_j>0$, but evidently, for $n$ large enough, the GH-convergence of $\X_n$ to $\X_\infty$ ensures that $\int \nchi |u_n|^{2^*}\, \d \mm_n =0$ (indeed, the support of $\nchi$ will be definitely disjoint to every $\X_n$). This is a contradiction and gives finally that $(x_j)_{j\in J} \subset \X_\infty$.

\noindent \textsc{Step 2}. We now consider the case of a general $u_\infty$. Observe that from Lemma \ref{lem:sobolev stability} $\X_\infty$ supports a $(2^*,2)$-Sobolev inequality hence, $u_\infty \in L^{2^*}(\mm_\infty)$. From Lemma \ref{lem:mixrecov} there exists a sequence $\tilde u_n \in W^{1,2}(\X_n)$ such that $\tilde u_n$ converges to $u_\infty$ both strongly in $W^{1,2}$ and strongly in $L^{2^*}$. Consider now the sequence $v_n := u_n-\tilde u_n$. Clearly $v_n$ converges to zero both in $L^{2}$-strong and in $W^{1,2}$-weak. Moreover the measures $|v_n|^{2^*}\mm_n$ and $|Dv_n|^2\mm_n$ have uniformly bounded mass. Since $(Z,\sfd)$ is compact, passing to a non-relabeled subsequence we have $|v_n|^{2^*}\mm_n\rightharpoonup \bar \nu$ and $|Dv_n|^2\mm_n\rightharpoonup \bar \mu$ in duality with $C_b(\Z)$ for some $\bar \nu, \bar \mu \in \MM^+_b(\Z).$ Therefore we can apply Step 1 to the sequence $v_n$ to get $\bar \nu = \sum_{j \in J}\nu_j \delta_{x_j}$, $\bar \mu \ge \sum_{j \in J}\mu_j\delta_{x_j}$ for a suitable countable family $J$, $(x_j)\subset \X_\infty$ and weights $(\nu_j),(\mu_j)$ satisfying $\nu_j^{2/2^*} \le (\limsup_{n} A_n)\mu_j$. To carry the properties of $v_n$ to the sequence $u_n$ we invoke Lemma \ref{lem:BreLieb} (with $q'=2$ and $q=2^*$) to deduce that
\begin{equation} \lim_{n\rightarrow \infty}\int |\varphi|^{2^*}|u_n|^{2^*}\, \d \mm_n -\int|\varphi|^{2^*}|v_n|^{2^*}\, \d \mm_n =\int |\varphi|^{2^*}|u_\infty|^{2^*}\, \d \mm_\infty ,\label{eq:BrLieb}\end{equation}
and, taking into account the weak convergence, this implies that
\[ \int \varphi^{2^*} \, \d \nu -\int \varphi^{2^*}\, \d \bar \nu = \int |u_\infty|^{2^*}\varphi^{2^*}\, \d \mm_\infty, \]
for every non-negative $\varphi \in C_b(\Z)$. In particular, this is equivalent to say that $\nu = |u_\infty|^{2^*} \mm_\infty + \bar \nu =  |u_\infty|^{2^*} \mm_\infty + \sum_{j\in J}\nu_j \delta_{x_j} $, which proves $i)$. Next, we claim that $\mu \ge \sum_{j \in J}\mu_j\delta_{x_j}$ and, to do so, we consider for each $j \in J$ and $\eps >0$, $\nchi_\eps \in \LIP_b(\Z)$, $0\le \nchi_\eps \le 1, \nchi_\eps(x_j)=1$ and supported in $B_\eps(x_j)$.  The key ingredient is the following estimate
\begin{equation*}
    \begin{split}
       \Big| \int \nchi_\eps|Du_n|^2\, \d \mm_n -\int \nchi_\eps|Dv_n|^2\,\d\mm_n\Big| &\le \int\nchi_\eps  \big||Du_n|-|Dv_n|\big|\big(|Du_n|+|Dv_n|\big)\, \d \mm_n \\
        &\le \int \nchi_\eps|D\tilde u_n|\big(|Du_n|+|Dv_n|\big)\, \d \mm_n \\
       &\le\Big(\int\nchi_\eps^2|D\tilde u_n|^2\, \d \mm_n\Big)^{1/2}\Big(\| |Du_n|\|_{L^2(\mm_n)} +\| |Dv_n|\|_{L^2(\mm_n)}\Big).
    \end{split}
\end{equation*}
Observe now that from \cite[Theorem 5.7]{AH17} $|D \tilde u_n|\to |D u_\infty|$ strongly in $L^2$ and in particular $\int\nchi_\eps^2|D\tilde u_n|^2\, \d \mm_n\to \int\nchi_\eps^2|D u_\infty|^2\, \d \mm_\infty$. Moreover $\int\nchi_\eps^2|D u_\infty|^2\, \d \mm_\infty\to0$ as $\eps\to 0^+$ and $u_n,v_n$ are uniformly bounded in $W^{1,2}(\X_n)$. Therefore taking in the above inequality first $n \to +\infty$ and afterwards $\eps\to 0^+$ we  ultimately deduce that 
$$\mu(\{x_i\})=\bar \mu(\{x_i\})\ge \mu_j, \quad \forall \, j \in J.$$
In particular, since $\mu$ is non-negative, $\mu \ge \sum_{j \in J}\mu_j\delta_{x_j}$, as claimed. Finally, by the weak lower semicontinuity result in \cite[Lemma 5.8]{AH17}, we have 
 $$\int \phi |Du_\infty|^2\, \d \mm_\infty \le \liminf_n\int \phi |Du_n|^2\, \d \mm_n = \int \phi \, \d \mu$$
  for every $\phi \in C_b(\Z)$ positive. Therefore, we get $\mu\ge |D u_\infty|^2\mm_\infty$ and, by mutual singularity of the two lower bounds, we have $(ii)$ and the proof is now concluded.
\end{proof}

We are finally ready to prove the main result of this section.

\begin{proof}[Proof of Theorem \ref{thm:CC_Sob}]
Set $\tilde u_n\coloneqq u_n\|u_n\|_{L^q(\mm_n)}^{-1}$. By assumption 
    \begin{equation} 
    1 \ge A_n\| |D \tilde u_n|\|^2_{L^2(\mm_n)} + B_n\|  \tilde u_n \|^2_{L^2(\mm_n)},\qquad \forall n \in \N.\label{eq:concompge}
    \end{equation}
    Moreover again by hypothesis $A_n\to A>0$, $B_n\to B>0$, therefore the sequences $A_n,B_n$ are bounded away from zero and thus $\sup_n \| \tilde u_n\|_{W^{1,2}(\X_n)} <\infty$. Hence, up to passing to a non relabeled subsequence, Proposition \ref{prop:W12compact} grants that $\tilde u_n$ converges $L^2$-strongly to a function $u_\infty \in W^{1,2}(\X_\infty)$. Moreover, the  measures $|D \tilde u_n|^2\mm_n$, $|\tilde u_n|^{2^*}\mm_n$ have uniformly bounded mass. In particular  up to a further not relabeled subsequence, there exists $\mu,\nu \in \MM_b^+(\Z)$ so that $|D \tilde u_n|^2\mm_n\weakto \mu$ and $|\tilde u_n|^{2^*}\mm_n \weakto \nu$ in duality with $C_b(\Z)$. We are in position to apply  Lemma \ref{lem:conccompII} to get the existence of at most countably many points $(x_j)_{j \in J}$ and weights $(\nu_j)_{j \in J}$, so that $\nu =|u_\infty|^{2^*}\mm_\infty +\sum_{j\in J}\nu_j\delta_{x_j}$ and $\mu \ge |Du_\infty|^2\mm_\infty +\sum_{j\in J}\mu_j\delta_{x_j}$, with $A\mu_j \ge \nu_j^{2/2^*}$  and in particular $\sum_j \nu_j^{2/2^*}<\infty$. Finally from Lemma \ref{lem:sobolev stability} we have that $\X_\infty$ supports a $(2^*,2)$-Sobolev inequality with constants $A,B.$ Therefore we can perform the following estimates
   \[
    \begin{split}
        1 = \lim_n\|  \tilde u_n \|^2_{L^{2^*}(\mm_n)}  &\ge \lim_n   A_n\| |D \tilde u_n|\|^2_{L^2(\mm_n)} + B\|  \tilde u_n \|^2_{L^2(\mm_n)}  \\
           &= A \mu(\X_\infty) + B\int |u_\infty|^2\, \d \mm_\infty \\
         &\ge  A\int |D u_\infty|^2\,\d \mm_\infty +B \int |u_\infty|^2\, \d \mm_\infty + \sum_{j\in J} \nu_j^{2/2^*} \\
          &\ge \Big(\int |u_\infty|^{2^*}\, \d \mm_\infty\Big)^{2/2^*} + \sum_{j\in J} \nu_j^{2/2^*} \\
           &\ge \Big( \int |u_\infty|^{2^*}\, \d \mm_\infty  + \sum_{j\in J} \nu_j \Big)^{2/2^*} = \nu(\X_\infty)^{2/2^*} = 1,
    \end{split}
    \]
    where in the last inequality we have used the concavity of the function $t^{2/2^*}.$ In particular all the inequalities must be equalities and, since $t^{2/2^*}$ is strictly concave, we infer that every  term in the sum $\int |u_\infty|^{2^*}\, \d \mm_\infty  + \sum_{j\in J} \nu_j^{2/2^*}$ must vanish except for one that must be equal to 1. If $\int |u_\infty|^{2^*}\, \d \mm_\infty=1$ then {\sc I)} must hold. If instead $\nu_j=1$ for some $j \in J$, then $u_\infty=0$ and by definition of $\nu$, $|\tilde u_n|^{2^*}\mm_n\rightharpoonup \delta_{x_j}$, which is exactly {\sc II)}.
\end{proof}

\subsection{Quantitative linearization}\label{sec:linearization}
A key point in our argument for the rigidity, and especially for the almost-rigidity, of $\aopt_q$ will be a more ``quantitative'' version of the elementary linearization of the Sobolev inequality contained in Lemma \ref{lem:technical linearization}. To state our result, given $q\in (2,\infty)$ and $u \in W^{1,2}(\X)$ with $\int |Du|_2, \d \mm>0$, it is convenient to define the Sobolev ratio associated to $u$ as the quantity
\begin{equation} \cQ_q^\X (u) :=  \frac{\|u\|_{L^{q}(\mm)}^2-\|u\|_{L^2(\mm)}^2}{\||D u|_2\|_{L^2(\mm)}^2}.\label{eq:quotient}\end{equation}
Observe that, if $\lambda^{1,2}(\X)>0$, $\int |Du|_2 \d \mm>0$ as soon as $u$ is not ($\mm$-a.e.\  equal to a) constant.
\begin{lemma}[Quantitative linearization]\label{lem:expansion}
	For all numbers  $A,B\ge 0$, $q >2$  and $\lambda>0$ there exists a constant $C=C(q  ,A,B,\lambda)$ such that the following holds.
	Let $\Xdm$ be a metric measure space with $\mm(\X)=1$, $ \lambda^{1,2}(\X)\ge \lambda$ and supporting a $(q,2)$-Sobolev inequality with constants $A,B$. 
	Then, for every non-constant $f \in W^{1,2}(\X)$ satisfying $\|f\|_{L^2(\X)}\le 1/2$, it holds 
	\begin{equation}\label{eq:ratio vs quotient}
 \Big| \cQ_q^\X (1+f) - \frac{(q-2)  \int \big(f-\int f \d \mm\big)^2\, \d \mm }{\int |D f|_2^2\, \d \mm}\Big|  \le C\big(\|f\|^{3\wedge q-2}_{W^{1,2}(\X)}+\|f\|^{q-2}_{W^{1,2}(\X)}+\|f\|^{2q-2}_{W^{1,2}(\X)}\big).
	\end{equation}
\end{lemma}
\begin{proof}
We claim that it is enough to prove the statement for  functions $f \in W^{1,2}(\X)$ with zero mean  (and arbitrary $L^2$-norm). Indeed for a generic $f \in W^{1,2}(\X)$ satisfying $\|f\|_{L^2(\X)}\le 1/2,$ we can take $\tilde f\coloneqq \frac{f-\int f \, \d \mm}{1+\int f \, \d \mm}$, which clearly has zero mean.
Then the conclusion would follow observing that the left hand side of \eqref{eq:ratio vs quotient} computed at $\tilde f$ coincides with  the left hand side of \eqref{eq:ratio vs quotient} computed at $f$  and from the fact that
 $$\|\tilde f\|_{W^{1,2}(\X)}\le \|f\|_{W^{1,2}(\X)}\Big (1+\int f\, \d \mm \Big )^{-1}\le \|f\|_{W^{1,2}(\X)}(1-\|f\|_{L^2(\X)})^{-1}\le 2 \|f\|_{W^{1,2}(\X)}.$$
	Therefore we can now fix $f \in W^{1,2}(\X)$ with $\int f\, \d \mm=0.$	We start with a basic estimate of the $L^r$ norm of $f$ for $r \in[1,q]$. Combining the H\"older and the $(q,2)$-Sobolev inequalities we have
\begin{equation}\label{eq:vanish lp}
\int |f|^r \, \d \mm \le \Big (\int |f|^q \, \d \mm \Big)^\frac{r}{q}\le (A^{r/2}+B^{r/2})\|f\|_{W^{1,2}(\X)}^{r}
\end{equation}
In the case $r \in(2,q]$ the following refined estimate  holds:
\begin{align}\label{eq:strong vanish lp}
	\frac{\int |f|^r \, \d \mm }{\int |D f|_2^2 \, \d \mm} &\le C_qA^{r/2}\Big(\int |D f|_2^2 \, \d \mm \Big)^{\frac{r}{2}-1}+ C_q B^{r/2}\Big(\int |f|^2 \, \d \mm \Big)^{\frac{r}{2}-1 }\frac{\int |f|^2\, \d\mm}{\int |D f|_2^2\, \d\mm}\nonumber \\
	&\le C_q(A^{r/2}+B^{r/2}\lambda^{-1})\|f\|_{W^{1,2}(\X)}^{r-2}.
\end{align}
We now apply \eqref{lem:technical linearization} to $f$, which we rewrite here for the convenience of the reader:
\[
\begin{split}
         \Big |	\Big( \int |1+f|^q\, \d \mm \Big)^{2/q} -&\int (1+f)^2\, \d \mm -(q-2)  \int |f|^2\, \d \mm \Big|\\
         &\le \tilde C_q \Big(\int |f|^{3\wedge q}+|f|^q\, \d \mm  +\Big( \int |f|^q\, \d \mm\Big)^2+\Big(\int |f|^2\, \d \mm\Big)^2\Big),
         \end{split}
\]
where $\tilde C_q$ is a constant depending only on $q$. Dividing by $\int |D f|_2^2\, \d \mm$ the above inequality and rearranging terms, using the definition of $ \lambda^{1,2}(\X)$ and the estimates \eqref{eq:vanish lp}, \eqref{eq:strong vanish lp} we obtain \eqref{eq:ratio vs quotient}.
\end{proof}

\subsection{Proof of the rigidity}
Here we prove Theorem \ref{thm:rigidity}. This result  will follow from the following theorem, which characterizes the behavior of extremal sequences for the Sobolev inequality and which combines the tools of concentration compactness and linearization, developed in the previous sections. This result can be summarized as: either there exist non-constant extremals, or we have information on the first eigenvalue $\lambda^{1,2}(\X)$, or we have information on the density $\theta_N.$
\begin{theorem}[The Sobolev-alternative]\label{thm:alternative}
	Let $\Xdm$ be a compact $\RCD(K,N)$ space for some $K \in \R$, $N\in (2,\infty)$ and with $\mm(\X)=1$. Let $q \in (2,2^*]$, with $2^*\coloneqq2N/(N-2)$. Then at least one of the following holds:
	\begin{enumerate}[i)]
	    \item there exists a non-constant function $u\in W^{1,2}(\X)$ satisfying
	    \begin{equation}\label{eq:optimal extremal}
	         \|u\|_{L^{q}(\mm)}^2=\aopt_q(\X)\||D u|\|_{L^2(\mm)}^2+\| u\|_{L^2(\mm)}^2,
	    \end{equation}
	    \item $\aopt_{q}(\X)=\tfrac{q-2}{\lambda^{1,2}(\X)}$,
	    \item $q=2^*$ and $\aopt_{2^*}(\X)=\alpha_2(\X)=\frac{\eucl(N,2)^2}{\min \theta_N^{2/N}}$ (see the introduction and \eqref{eq:euclnp} for the definition of $\alpha_2(\X)$ and $\eucl(N,2)$).
	\end{enumerate}
\end{theorem}
\begin{proof}
By definition of $\aopt_q(\X)$ there exists a sequence of non-constant functions $u_n \in \LIP(\X)$ such that  $\cQ_q^\X(u_n)\to \aopt_q(\X)$ (recall \eqref{eq:quotient}). By scaling we can suppose that $\|u_n\|_{L^{2^*}(\mm)}\equiv 1.$ In particular $(u_n)$ is bounded in $W^{1,2}(\X)$. We distinguish two cases.

\noindent {\sc Subcritical:} $q<2^*.$ By compactness (see Proposition \ref{prop:W12compact}), up to passing to a subsequence, $u_n\to u$ strongly in $L^q$ to some function $u \in W^{1,2}(\X)$ such that, from the lower semicontinuity of the Cheeger energy,  $\cQ_q^\X(u)=\aopt_q(\X).$ If $u$ is non-constant $(i)$ holds and we are done, so suppose that $u$ is constant. Then from the renormalization we must have $u\equiv 1.$ Moreover, since $\|u_n\|_{L^q(\mm)},\|u_n\|_{L^2(\mm)}\to 1$ and $\cQ_q^\X(u_n)\to \aopt_q(\X)$, we deduce that $\||D u|\|_{L^2(\mm)}^2\to 0.$ Consider now the functions $f_n\coloneqq u_n-1 \in \LIP(\X)$, which  are non-constant and such that $f_n \to 0$ in $W^{1,2}(\X).$  We are therefore in position to apply Lemma \ref{lem:expansion} and deduce that
\[
\aopt_q(\X)= \lim_{n\to \infty} \cQ_q^\X (u_n)=\lim_n \frac{(q-2)  \int \left(f_n-\int f_n \d \mm\right)^2\, \d \mm }{\int |D f_n|^2\, \d \mm}\le \frac{q-2}{ \lambda^{1,2}(\X)}.
\]
Combining this with \eqref{eq:Aopt2gap}, we get that $\aopt_q(\X)=\frac{q-2}{ \lambda^{1,2}(\X)}$, i.e. $(ii)$ is true  and we conclude the proof in this case.

\noindent {\sc Critical:} $q=2^*.$ We apply the concentration-compactness result in Theorem \ref{thm:CC_Sob} and deduce that up to a subsequence: either $u_n \to u$ in $L^{2^*}(\mm)$ to some $u \in W^{1,2}(\X)$ or $\|u_n\|_{L^2(\mm)}\to 0.$ In the first case we argue exactly as above using  Lemma \ref{lem:expansion} and deduce that either $(i)$ or $(ii)$ holds. Hence we are left to deal with the case  $\|u_n\|_{L^2(\mm)}\to 0.$  From the definition of $\alpha_2(\X)$, for every $\eps$ there exits $B_{\eps}$ so that a $(2^*,2)$-Sobolev inequality with constants $\alpha_2(\X)+\eps$ and $B_{\eps}$ is valid. Hence we have
\[
\cQ_{2^*}^\X(u_n)\||D u_n|\|_{L^2(\mm)}^2+\| u_n\|_{L^2(\mm)}^2= \|u_n\|_{L^{2^*}(\mm)}\le (\alpha_{2}(\X)+\eps)\||D u_n|\|_{L^2(\mm)}^2+B_\eps\| u_n\|_{L^2(\mm)}^2,
\]
which gives
\[
\cQ_{2^*}^\X(u_n)\le (\alpha_{2}(\X)+\eps)+ B_{\eps}\| u_n\|_{L^2(\mm)}^2(\||D u_n|\|_{L^2(\mm)}^2)^{-1}.
\]
Observing that $\liminf_{n}\||D u_n|\|_{L^2(\mm)}^2>0$ (which follows from the Sobolev inequality, $\| u_n\|_{L^2(\mm)}^2\to0$ and $\|u_n\|_{L^{2^*}(\mm)}=1$) and letting $n \to +\infty$ we arrive at $\aopt_{2^*}(\X)\le (\alpha_{2}(\X)+\eps)$. From the arbitrariness $\eps$  we deduce that $\aopt_{2^*}(\X)\le \alpha_{2}(\X)$ and the proof is concluded (indeed by definition $\alpha_2(\X)\ge \aopt_{2^*}(\X)$ is always true).
\end{proof}

We can finally come to the proof of the principal result of this note.
\begin{proof}[Proof of  Theorem \ref{thm:rigidity}]

The ``if" implication is direct as any $N$-spherical suspension, $\X$ is so that $\aopt_q(\X)=\frac{q-2}{N}$. This can be seen from the lower bound in Proposition \ref{prop:aopt2gapbound} (recall also Theorem \ref{thm:obata}) and the upper bound given in Theorem \ref{thm:Aopt upper cd}.

For the  ``only if' implication, the result will follow from three different rigidity results, one for each of the alternatives in Theorem \ref{thm:alternative}. Up to scaling the reference measures, we can suppose $\mm(\X)=1$.

\noindent{\sc Case 1:} $i)$ in Theorem \ref{thm:alternative} holds. Let $u$ be the non-constant function satisfying \eqref{eq:optimal extremal}. Observe that we can assume that $u$ is non-negative. We aim to apply the P\'olya-Szeg\H{o} inequality with the model space $I_{N}$ as in Section \ref{sec:polya}. Let $u_N^*: I_{N}\to [0,\infty] $ be the monotone-rearrangement of $u$. From the P\'olya-Szeg\H{o} inequality in Theorem \ref{thm:KNpolya} we have  that $u_N^* \in W^{1,2}(I_{N},|.|,\mm_{N})$, $\|u\|_{L^p(\mm)}=\|u^*_N\|_{L^p(\mm_N)}$ for both $p\in \{q,2\}$ and that $\||D u^*_N|\|_{L^2(\mm_{N})}\le \||D u|\|_{L^2(\mm)}$. Combining this with  \eqref{eq:sobolev model} we have
\begin{align*}
 \|u\|_{L^q(\mm)}^2&=\|u_N^*\|_{L^q(\mm_{N})}^2 \le \tfrac{q-2}{N} \||D u_N^*|\|_{L^2(\mm_{N})}^2+\|u_N^*\|_{L^2(\mm_{N})}^2\\
 &\le \tfrac{q-2}{N} \||D u|\|_{L^2(\mm)}^2+\|u\|_{L^2(\mm)}^2 = \|u\|_{L^q(\mm)}^2.
\end{align*}
Therefore $\||D u_N^*|\|_{L^2(\mm_{N})}=\||D u|\|_{L^2(\mm)}$ and, since $u$ is non-constant, we are in position to apply the rigidity of the P\'olya-Szeg\H{o} inequality of Theorem \ref{thm:Polyarigidity} and conclude the proof in this case.

\noindent{\sc Case 2:} $ii)$ in Theorem \ref{thm:alternative} holds. We immediately deduce that $ \lambda^{1,2}(\X)=N$ and the conclusion follows from the Obata's rigidity (Theorem \ref{thm:obata}).

\noindent{\sc Case 3:} $iii)$ in Theorem \ref{thm:alternative} holds.  From Theorem \ref{thm:Amin} and the explicit expression for $\eucl(N,2)$ (see \eqref{eq:eucln2}) we have that 
$$\frac{2^*-2}{N}=\aopt_{2^*}(\X)=\alpha_2(\X)=\frac{\eucl(N,2)^2}{\min_{x\in \X} \theta_N(x)^{2/N}}=\frac{2^*-2}{N \sigma_{N}^{2/N}\min_{x\in \X} \theta_N(x)^{2/N}},$$
therefore $\min_{x \in \X} \theta_N= \sigma_{N}^{-1}.$ On the other hand by the Bishop-Gromov inequality  and identity \eqref{eq:theta equivalent} 
$$\frac{1}{\sigma_{N}}=\inf_{\X} \theta_N(x)\ge \frac{\mm(\X)}{v_{N-1,N}(\diam(\X))}=\frac{1}{v_{N-1,N}(\diam(\X))},$$
which, from the definition of $v_{N-1,N}$  and \eqref{eq:integral sin} forces $\diam(\X)=\pi$. The conclusion then follows by the rigidity of the maximal diameter (Theorem \ref{thm:diameter rigidity}).
\end{proof}

\begin{remark} \label{rmk:sharper bound}\rm
The rigidity result for $\aopt_q(M)$ in the subcritical range $q<2^*$ was already observed in \cite{Ledoux00} as a consequence of the following sharper estimate due to \cite{Fontenas97}: for any $n$-dimensional Riemannian manifolds $M$, $n \ge 3$, with ${\rm Ric}\ge n-1$ it holds
	\begin{equation}\label{eq:sharper bound}
		 \aopt_q(M) \le  \frac{(q-2)}{\kappa(\theta)}, \qquad \forall q \in (2,2^*),
	\end{equation}
where $\kappa(\theta) := \theta n + (1-\theta)\lambda^{1,2}(M)$,  $\lambda^{1,2}(M)$ being the first non trivial eigenvalue and $\theta=\theta(q)\in[0,1]$ is a suitable interpolation parameter. The spectral gap inequality $\lambda^{1,2}(M) \ge n$ grants that the bound \eqref{eq:sharper bound} improves the one of \eqref{eq:Aopt upper}. For every $q \in (2,2^*)$, the condition $\aopt_q(M)=\aopt_q(\S^n)(=(q-2)/n)$ forces $\kappa(\theta)=n$ which in turn implies $\lambda^{1,2}(M)=n$. By appealing to the classical Obata's Theorem, this argument covers the rigidity of Theorem \ref{thm:rigidity manifold} for $q<2^*$. Nevertheless, this does not extend to the critical exponent: more precisely $\theta(q)\to 1$ as $q \to 2^*$, hence the quantity $\kappa(\theta)$ carries no information on the spectral gap in this case. 
\fr
\end{remark}

\section{Almost rigidity of $\aopt$}\label{sec:almrigidity} 

\subsection{Behavior at concentration points}
The following technical result will be needed for the almost-rigidity result and has the role of replacing in the varying-space case, the Sobolev inequality with constants $\alpha_2(\X)+\eps,B_{\eps}$ which we used in the fixed-space case of the rigidity (see the proof of Theorem \ref{thm:alternative}). Indeed it is not clear how to control the constant $B_{\eps}$ in a sequence of mGH-converging spaces. Therefore we need a more precise local analysis that fully exploits the local Sobolev inequalities in  Theorem \ref{thm:local sobolev intro}  and Proposition \ref{prop:local sobolev embedding}.

\begin{lemma}[Behavior at concentration points]\label{lem:concpoint}
	Let $(\X_n,\sfd_n,\mm_n,x_n)$, $n \in \bar \N$, be a sequence of $\rcd {K}{N}$ spaces $K\in \R$, $N\in(1,\infty)$, so that $\X_n \overset{pmGH}\to \X_\infty$. Fix $p \in(1,N)$, set $p^*\coloneqq pN/(N-p)$ and assume that $u_n \in \LIP_c(\X_n)$ is a sequence satisfying 
	\begin{equation}\label{eq:reverse sobolev}
	\|u_n\|^p_{L^{p^*}(\mm_n)}\ge A_n \||D u_n|\|^p_{L^p(\mm_n)}-B_n\|u_n\|^p_{L^{s}(\mm_n)}, 
	\end{equation}
	for some constants $A_n,B_n\ge 0$  uniformly bounded and $s>0$ so that $s \in [p,p^*)$. Assume furthermore that $u_n\to 0$ strongly in $L^p$, $\|u_n\|_{L^{p^*}(\mm_n)}=1$ and that $|u_n|^{p^*}\mm_n\rightharpoonup \delta_{y_0}$ for some $y_0 \in \X_\infty$ in duality with $C_{bs}(\Z)$ (where  $(\Z,\sfd_\Z)$ is a proper space
	 realizing the convergence in the extrinsic approach). Then
	 \begin{equation}\label{eq:link theta A}
	 	\theta_{N}(y_0)\le \eucl(N,p)^N(\limsup_n A_n)^{-N/p},
	 \end{equation}
 meaning that if $\theta_{N}(y_0)=+\infty$, then $\limsup_n A_n =0.$ 
\end{lemma}
\begin{proof}
We subdivide the proof in two cases.

\noindent{ \sc Case 1:} $\theta_N(y_0)<+\infty$. 

Fix $\eps<\theta_N(y_0)/4$ arbitrary. Since $\theta_{N,r}(y_0)\to \theta_N(y_0)$ as $r \to 0^+$  there exists $\bar r=\bar r(\eps)$ such that
	\begin{equation}\label{eq:theta buona}
		|\theta_{N,r}(y_0)-\theta_N(y_0)|\le \eps, \quad \forall r<\bar r.
	\end{equation}
Let $\delta\coloneqq\delta(2\eps,D,N)$, with $D=4,$  be the constant given by  Theorem \ref{thm:local sobolev intro} and fix two radii $r,R\in (0,\bar r)$ such that $R<\delta \sqrt{N/K^-}$ and $r<\delta R.$  Consider now a sequence $y_n \in \X_n$ such that $y_n \to y_0$. From the convergence of the measures $\mm_n$ to $\mm_\infty$ we have that $\theta_{N,r}(y_n)\to \theta_{N,r}(y_0)$ and $\theta_{N,R}(y_n)\to \theta_{N,R}(y_0)$. In particular by \eqref{eq:theta buona}  there exists $\bar n=\bar n(r,R,\eps)$ such that
\begin{equation}\label{eq:theta buona n}
		|\theta_{N,R}(y_n)-\theta_N(y_0)|,|\theta_{N,r}(y_n)-\theta_N(y_0)|\le 2\eps, \qquad \forall n\ge \bar n.
\end{equation}
From the initial  choice of $\eps$  this also implies that $\theta_{N,r}(y_n)/\theta_{N,R}(y_n)\le 4$ for every $n \ge \bar n.$ We are in position to apply  Theorem \ref{thm:local sobolev intro} and get that for every $n \ge \bar n$
	\begin{equation}\label{eq:local sobolev sequence}
	\|f\|_{L^{p^*}(\mm_n)}\le \frac{(1+2\eps)\eucl(N,p)}{(\theta_{N}(y_0)-2\eps)^{\frac{1}{N}}} \||D f|\|_{L^p(\mm_n)}, \qquad \forall f \in \LIP_c(B_{r}(y_n)).
\end{equation}
Choose $\phi \in \LIP(\Z)$ such that $\phi=1$ in $B^\Z_{r/8}(y_0)$, $\supp (\phi)\subset B^\Z_{r/4}(y_0)$ and $0\le \phi\le 1$.
From the assumptions, we have that $\int \phi |u_n|^{p^*}\d \mm_n\to 1$, in particular up to increasing $\bar n$ it holds that $\int \phi |u_n|^{p^*}\d \mm_n\ge 1-\eps$ for all $n \ge \bar n$. Moreover, again up to increasing $\bar n$, we have that $\sfd_\Z(y_n,y_0)\le r/4$  for all $n \ge \bar n$, therefore  
\begin{equation}\label{eq:concentration}
	1-\eps\le \int_{B_{r/2}(y_n)}|u_n|^{p^*}\d \mm_n, \qquad \forall n \ge \bar n.
\end{equation}
For every $n$ we choose a cut-off function $\phi_n\in \LIP(\X_n)$ such that $\phi_n=1$ in $B_{r/2}(y_n)$, $0\le \phi_n\le 1$,  $ \supp(\phi_n)\subset  \LIP_c(B_{r}(y_n))$  and $\Lip (\phi_n)\le 2/r.$ Plugging the function $u_n\phi_n\in \LIP_c(B_r(y_n))$ in \eqref{eq:local sobolev sequence} and using \eqref{eq:concentration} we obtain 
\begin{equation}\label{eq:partial sobolev}
	(1-\eps)^{\frac 1{p^*}} \le \|u_n\phi_n\|_{L^{p^*}(\mm_n)}\le   \frac{(1+2\eps)\eucl(N,p)}{(\theta_{N}(y_0)-2\eps)^{\frac{1}{N}}} 
	\big(\||D u_n|\|_{L^p(\mm_n)}+\tfrac{2}{r}\|u_n\|_{L^p(\mm_n)} \big).
\end{equation}
Moreover recalling that $\|u_n\|_{L^{p^*}(\mm_n)}=1$ and the assumption \eqref{eq:reverse sobolev}, from \eqref{eq:partial sobolev}  we reach
\[
 (1-\eps)^{\frac 1{p^*}} \big( A_n^{1/p}\||D u_n|\|_{L^p(\mm_n)} -B_n\|u_n\|^p_{L^{s}(\mm_n)} \big) \le \frac{(1+2\eps)\eucl(N,p)}{(\theta_{N}(y_0)-2\eps)^{\frac{1}{N}}} 
\big(\||D u_n|\|_{L^p(\mm_n)}+\tfrac{2}{r}\|u_n\|_{L^p(\mm_n)} \big).
\]
We also observe that from the assumption $\|u_n\|_{L^p(\mm_n)}\to 0$ and the fact that $\|u_n\|_{L^{p^*}(\mm_n)}=1$, we have by $(viii)$ in Proposition \ref{prop:lp prop} that $\|u_n\|_{L^s(\mm_n)}\to 0.$ Finally by \eqref{eq:partial sobolev} and the assumption $\|u_n\|_{L^p(\mm_n)}\to 0$ it holds that  $\liminf_n\||D u_n|\|_{L^p(\mm_n)}>0.$ In particular for $n$ big enough we can divide by $\||D u_n|\|_{L^p(\mm_n)}$ the above inequality and letting $n \to +\infty$ we get
\[
\limsup_nA_n^{1/p}  \le \frac{(1+2\eps)\eucl(N,p)}{(1-\eps)^{1/{p^*}} (\theta_{N}(y_0)-2\eps)^{\frac{1}{N}}}.
\]
From the arbitrariness of $\eps$, the conclusion follows. 

\noindent{ \sc Case 2:} $\theta_N(y_0)=\infty$.

The argument is similar to Case 1, but we will use Proposition \ref{prop:local sobolev embedding} instead of  Theorem \ref{thm:local sobolev intro}. Let $M>0$ be arbitrary.  There exists $r\le 1$ such that
	$\theta_{N,r}(y_0)\ge 2M$. As above we choose a sequence $y_n \to y_0$. For $n$ big enough we have that 
\begin{equation}\label{eq:theta grande}
	\theta_{N,r}(y_n)\ge M.
\end{equation}
Applying Proposition \ref{prop:local sobolev embedding}, from \eqref{eq:theta grande} we get that for every $n$ big enough
\begin{equation}\label{eq:local sobolev sequence2}
	\|f\|^p_{L^{p^*}(B_r(y_n))}\le  \frac{C_{K,N,p}}{M^{\frac pN}}\||D f|\|^p_{L^p(B_r(y_n))}+\frac{C_{p,N}\|f\|^p_{L^p(B_r(y_n))}}{r^{p/N}M^{\frac{p}{N}}}, \qquad \forall f \in \LIP(\X_n).
\end{equation}
Observing that \eqref{eq:concentration} is still satisfied with  $\eps=1/M$ and $n$ big enough, we can repeat the above argument, using \eqref{eq:reverse sobolev} and plugging $\phi_n u_n$ in \eqref{eq:local sobolev sequence2}, where $\phi_n$ is as above. This leads us to 
\[
\limsup_n A_n^{1/p}  \le  \frac{C_{K,N,p}}{(1-1/M)^{1/p^*}M^{\frac1N}},
\]
which from the arbitrariness $M$ implies the conclusion.
\end{proof}

%#######################################################################

\subsection{Continuity of $\aopt$ under mGH-convergence}\label{sec:continuity aopt}

In Lemma \ref{lem:sobolev stability}, we proved that Sobolev embeddings are stable with respect to pmGH-convergence. A much more involved task it to prove that \emph{optimal} constants are also continuous: indeed, if $\X_n \overset{mGH}{\to} \X_\infty $, in general Lemma \ref{lem:sobolev stability} ensures only that $\aopt_q(\X_\infty) \le \liminf_n \aopt_q(\X_n)$. With the concentration compactness tools developed in Section \ref{sec:CC}, the ``quantitative-linearization'' result in Lemma \ref{lem:expansion} and the technical tool developed in the previous section we can now prove the mGH-continuity of $\aopt_q(\X_n)$ as stated in Theorem \ref{thm:mGHAopt}, that we restate here for convenience of the reader.
\begin{theorem}[Continuity of $\aopt_q$ under mGH-convergence]\label{thm:mGHAopt2}
Let $(\X_n,\sfd_n, \mm_n)$ be a sequence, $n \in \N\cup \{\infty\}$, of compact $\RCD(K,N)$-spaces with $\mm_n(\X_n)=1$ and for some $K\in \R$, $N \in(2,\infty)$ so that $\X_n\overset{mGH}{\to} \X_\infty$. Then, $A_q^{\rm opt}(\X_\infty)=\lim_n \aopt_q(\X_n)$, for every $q \in (2,2^*]$.
\end{theorem}
\begin{proof}
By definition of $\aopt_q(\X_n)$, there exists  sequence of non-negative and non-constant functions $u_n \in \LIP(\X_n)$  satisfying 
    \begin{equation}\label{eq:sobolev stability fine}
    \|u_n\|_{L^{q}(\mm_n)}^2\ge A_n\||D u_n|\|_{L^2(\mm_n)}^2 +\|u_n\|_{L^2(\mm_n)}^2,\end{equation}
    having set $A_n:=\aopt_q(\X_n) -\tfrac{1}{n}$. By scaling invariance, it is not restrictive to suppose $\|u_n\|_{L^{q}(\mm_n)}=1$ for every $n \in \N$. Observe that thanks to Lemma \ref{lem:sobolev stability} we already have that $0 < \aopt_{q}(\X_{\infty}) \le \liminf_n \aopt_q(\X_n)$, hence we only need to show that $\aopt_{q}(\X) \ge \limsup_n \aopt_q(\X_n)$. To this aim, we distinguish two cases.

\noindent\textsc{Subcritical:} $q<2^*$. It is clear that $A_n$ is uniformly bounded from below whence the sequence $u_n$ has uniformly bounded $W^{1,2}$ norms. Then,  by Proposition \ref{prop:W12compact} and the $\Gamma$-$\liminf$ inequality of the $\rmCh_2$ energy, there exists a (not relabeled) subsequence $L^2$-strongly converging to some $u_\infty \in W^{1,2}(\X_\infty)$. Moreover, since $u_n$ are bounded in $L^{2^*}$, they also converge to $u_\infty $ in $L^q$-strong and in particular $\|u_\infty\|^2_{L^q(\mm_\infty)}=1$. Suppose first that the function $u_\infty$ is not constant, then we get
\[ \begin{split}
 1=\|u_\infty\|^2_{L^q(\mm_\infty)}   &\ge \limsup_{n \to \infty} A_n \||D u_n|\|_{L^2(\mm_n)}^2 +\|u_n\|_{L^2(\mm_n)}^2 \\
\eqref{eq:Gammaliminf} + \ L^2\text{-strong}\qquad & \ge \limsup_{n \to \infty} \aopt_q(\X_n) \||D u_\infty|\|_{L^2(\mm_\infty)}^2 +\|u_\infty\|_{L^2(\mm_\infty)}^2.
\end{split} \]
Since $u_\infty$ is not constant this in turn yields $ \limsup_n\aopt_q(\X_n)  \le \aopt_q(\X_\infty)$ which is what we wanted.

Suppose now that $u_\infty$ is constant. Then, necessarily $u_\infty =1$. Define now $f_n := 1-u_n$ and observe that  $\|f_n\|_{W^{1,2}(\X_n)}\to 0$, which follows from \eqref{eq:sobolev stability fine} and the fact that $\|u_n\|_{L^2(\mm_n)}\to 1.$ Moreover from \eqref{eq:lambdalsc} we have that $\lambda^{1,2}(\X_n)$ are uniformly bounded below away from zero. Therefore we can apply Lemma \ref{lem:expansion} to deduce (recall \eqref{eq:quotient} for the def. of $\cQ^{\X}_q$)
\begin{equation} \limsup_{n\to \infty} \aopt_q(\X_n) = \limsup_{n\to \infty} \cQ^{\X_n}_q(u_n) = \limsup_{n \to \infty} \frac{(q-2) \int \big|f_n -\int f_n\, \d \mm_n\big|^2\, \d \mm_n}{\int |D f_n|^2\, \d \mm_n} \le \limsup_{n\to \infty} \frac{(q-2)}{\lambda^{1,2}(\X_n)} = \frac{(q-2)}{\lambda^{1,2}(\X_\infty)}, \label{eq:inter1}
\end{equation}
having used, in the last inequality, the continuity  of the $2$-spectral gap \eqref{eq:lambdalsc}. This combined with \eqref{eq:Aopt2gap} gives that  $ \limsup_n\aopt_q(\X_n)  \le \aopt_q(\X_\infty)$.

\noindent\textsc{Critical exponent:} $q=2^*$. Observe that we are now in position to invoke Theorem \ref{thm:CC_Sob} and, up to a further not relabeled subsequence, we just need to handle  one of the two different situations \textsc{I),II)} occurring in Theorem \ref{thm:CC_Sob}. If the case \textsc{I)} occurs, we argue exactly as in the {\sc Subcritical:} $q<2^*$ case, to conclude that $ \limsup_n\aopt_q(\X_n)  \le \aopt_q(\X_\infty)$. Hence we are left with situation \textsc{II)}, where the sequence $u_n$ develops a concentration point $y_0 \in \X_\infty$. Recalling Lemma \ref{lem:concpoint}, either $\theta_N(y_0) = \infty$ and $\limsup_n \aopt_{2^*}(\X_n) =0$ or $\theta_N(y_0) < \infty$. The first situation cannot happen, since $\aopt_{2^*}(\X_\infty)>0$. In the second one rearranging in \eqref{eq:link theta A} we have
\[ \limsup_{n\to \infty} \aopt_{2^*}(\X_n)  \overset{\eqref{eq:link theta A}}{\le}\frac{\eucl(N,2)^2 }{\theta_N(y_0)^{2/N}}\overset{\eqref{eq:alfa}}{\le} \alpha_{2}(\X_\infty)\le \aopt_{2^*}(\X_\infty).\]
\end{proof}

\subsection{Proof of the almost-rigidity}
Combining the rigidity result for $\aopt_q$ with the continuity result proved in the previous part we can now prove the almost-rigidity result for $\aopt_q.$
\begin{proof}[Proof of Theorem \ref{thm:almrigidity}]
	We argue by contradiction, and suppose that there exists $\eps>0$, $q \in(2,2^*]$ and a sequence $(\X_n,\sfd_n,\mm_n)$ of $\RCD(N-1,N)$-spaces with $\mm_n(\X_n)=1$ so that 
	\begin{equation}\label{eq:distant}
		\sfd_{mGH}( (\X_n,\sfd_n,\mm_n), (\Y,\sfd_\Y,\mm_\Y))>\eps,\end{equation}
	for every spherical suspension $(\Y,\sfd_\Y,\mm_\Y)$ and
$
	  \lim_n\aopt_q(\X_n)=\frac{q-2}{N}.
$
	 Theorem \ref{thm:mGHcompact} (recall that $\mm_n(\X_n)=1$) ensures that up to passing to a non-relabeled subsequence we have $	\X_n \overset{mGH}{\to}\X_\infty,$ for some $\RCD(N-1,N)$-space $(\X_\infty,\sfd_\infty,\mm_\infty)$ with $\mm_\infty(\X_\infty)=1$. Hence \eqref{eq:distant} implies
	\begin{equation}
	\sfd_{mGH}((\X_\infty,\sfd_\infty,\mm_\infty),(\Y,\sfd_\Y,\mm_\Y))\ge\eps, \label{eq:Xinfty}
	\end{equation}
	for every spherical suspension $(\Y,\sfd_\Y,\mm_\Y)$. Finally, by Theorem \ref{thm:mGHAopt} we deduce $$\aopt_q(\X_\infty) = \lim_n \aopt_q(\X_n)=\frac{q-2}{N}.$$ Therefore, by invoking the rigidity Theorem \ref{thm:rigidity}, we get that $(\X_\infty,\sfd_\infty,\mm_\infty)$ is isomorphic to a spherical suspension. This contradicts \eqref{eq:Xinfty} and concludes the proof.
\end{proof}
\begin{remark}
\rm
The results of Theorem \ref{thm:almrigidity} (and therefore of Theorem \ref{thm:rigidity}) extend directly to the class of $\RCD(K,N)$ spaces for some $K>0$ and $N\ge 2$ with normalized volume. Consider an $\RCD(K,N)$ space $\Xdm$ and define $(\X',\sfd',\mm'):=(\X,\sqrt{\tfrac{K}{N-1}}\sfd,\mm)$ which is $\RCD(N-1,N)$. Then, since $\aopt_q(\X') = \tfrac{K}{N-1} \aopt_q(\X)$, it is straightforward to set $\delta= \delta(K,N,\eps,q):= \tfrac{N-1}{K}\delta(N,\eps,q)$ and extend the aforementioned results also for arbitrary $K>0$.\fr 
\end{remark}

\section{Application: The Yamabe equation on $\RCD(K,N)$ spaces}\label{sec:scalarPDE}
In this section we apply Theorem \ref{thm:alfa} and the concentration compactness results of Section \ref{sec:CC} to study the Yamabe equation to the $\RCD(K,N)$ setting. In particular, we prove an existence result for the Yamabe equation and continuity of the generalized Yamabe constants under mGH-convergence, extending and improving some of the results proved in \cite{Honda18} in the case of Ricci limits. For  results concerning the Yamabe problem and the Yamabe constant in non-smooth spaces see also \cite{ACM13,ACM14,Mondello17,AkuMon19,Mondello17}.

We recall that the Yamabe problem \cite{Yamabe60} asks if a compact Riemannian manifold admits a conformal metric with constant scalar curvature. This has been completely solved and shown to be true after the works of Trudinger, Aubin and Schoen \cite{Trudinger68, Aubin76-3, Schoen84}. We also refer to \cite{LP87} for an introduction to this problem and for a complete and self-contained proof of this result. 

The Yamabe problem turns out  to be linked to  the so-called Yamabe equation:
\begin{equation}\label{eq:yamabe eq}
    -\Delta u + \sca\,u =\lambda u^{2^* -1}, \qquad \lambda \in \R, \, \sca\in L^\infty(M),
\end{equation}
where $2^*=\frac{2n}{n-2}$. Indeed solving the Yamabe problem is equivalent to find a non-negative and non-zero solution to \eqref{eq:yamabe eq} for some $\lambda \in \R$ and with $\sca={\rm Scal}$, the scalar curvature of $M$. In this direction, it is relevant to see that the Yamabe equation is the Euler-Lagrange equation of the following functional:
\[Q(u):= \frac{\int |D u|^2+ \sca |u|^2 \,\d \vol}{\|u\|_{L^{2^*}}^2}, \qquad u \in W^{1,2}(M)\setminus\{0\},\]
where $\vol$ is the volume measure of $M$. One then defines the Yamabe constant as the infimum of the above functional:
\[
\lambda_\sca(M)\coloneqq \inf_{u\in W^{1,2}(M)\setminus\{0\}} Q(u).
\]
A crucial step in the solution of the Yamabe problem is:
\begin{theorem}[\cite{Trudinger68, Aubin76-3,Yamabe60}]\label{thm: yamabe 1}
    Let $M$ be a compact $n$-dimensional Riemannian manifold satisfying $\lambda_\sca(M)<\eucl(n,2)^{-2}$. Then there is a non-zero solution to \eqref{eq:yamabe eq} with $\lambda=\lambda_\sca(M).$
\end{theorem}
Recall that $\eucl(n,2)$ denotes the optimal constant in the sharp Euclidean Sobolev inequality \eqref{eq:EuclSob}. It has  also  been proven by Aubin \cite{Aubin76-2} (see also \cite{LP87}) that
\begin{equation}\label{eq:yamabe upper bound}
    \lambda_\sca(M)\le \eucl(n,2)^{-2}
\end{equation}
always holds. 

The relevant point for our discussion is that Theorem \ref{thm: yamabe 1} turns out to be linked to the notion of optimal Sobolev constant $\alpha_2(M)$, in particular it is actually a corollary of the fact that $\alpha_2(M)=\eucl(n,2)^2$ (recall \eqref{eq:alfa manifold}). Since we generalized this last result to setting of compact $\RCD(K,N)$-spaces (see Theorem \ref{thm:alfa}), it is natural to ask if an analogue of Theorem \ref{thm: yamabe 1} holds also in this singular framework. We will positively address this in this part of the note.

\subsection*{Capacity and quasi continuous functions}
In the next section we will use the notions of capacity and quasi continuous functions. We briefly recall here the needed definitions and properties.

Given a metric measure space $\Xdm$, the \emph{capacity} of a set $E\subset \X$ is defined as 
\begin{equation}\label{eq:def cap}
	\Cap(E)\coloneqq\inf \{\|f\|_{W^{1,2}(\X)}^2\, : f \in W^{1,2}(\X),\, \ f\ge 1\, \mm\text{-a.e.\ in a neighborhood of $E$}\}.
\end{equation}
It turns out (see, e.g., \cite[Proposition 1.7]{DGP21}) that $\Cap$ is a submodular outer measure on $\X$ and satisfies $\mm(E)\le \Cap(E)$ for every Borel set $E\subset \X$.

A function $f:\X \to \R$ is said to be \emph{quasi-continuous} if for every $\eps>0$ there exists a set $E\subset \X$ such that $\Cap(E)<\eps$ and $f\restr{\X\setminus E}$ is continuous. We denote by $\mathcal{QC}(\X)$ the set of all equivalence classes-up to $\Cap$-a.e.\ equality-of quasi-continuous functions.

In \cite{DGP21} it has been proven that, in situations where continuous functions are dense in $W^{1,2}(\X)$,  there exists a unique  map 
$${\sf QCR} : W^{1,2}(\X)\to L^0(\Cap)$$
that is linear and such that ${\sf QCR}(f)$ is  (the $\Cap$-a.e.\ equivalence class of) a function which is \emph{quasi continuous} and coincides $\mm$-a.e.\  with $f$. Recall that when $\X$ is reflexive, then  Lipschitz functions are dense in $W^{1,2}(\X)$ (see, e.g., \cite[Proposition 7.6]{ACDM15}), hence the map ${\sf QCR}$  is available.

We conclude with the following convergence result contained in \cite{DGP21}:
\begin{equation}\label{eq:w21 conv implies cap}
    f_n \to f \text{ strongly in $W^{1,2}(\X)$ } \implies \text{ up to subsequence } {\sf QCR}(f_n)\to {\sf QCR}(f) \,\, \Cap\text{-a.e..}
\end{equation}

\subsection{Existence of solutions to the Yamabe equation on compact $\RCD$ spaces}\label{sec:existence yamabe}
We start by clarifying in which sense \eqref{eq:yamabe eq} is intended and, to this aim, we fix $\Xdm$  a compact $\RCD(K,N)$ space for some $K\in \R, N\in (2,\infty)$ with $\mm(\X)=1$. We will also denote by $2^*$ the Sobolev-exponent defined as $2^*\coloneqq2N/(N-2).$ We fix a radon measure $\sca$ in $\X$ so that, for some $p>N/2$, it satisfies
\begin{equation}\label{eq:scalar}
	\sca \ge g \mm,\,\, g \in L^p(\mm)\quad \text{and}\quad \sca \ll \Cap,
\end{equation}
where $\Cap$ denotes the capacity of $\X$ as defined above.  We also denote by $|\sca|$ the total variation of $\sca$ which for instance can be characterized by the formula $\sca = \sca^++\sca^-$, being $\sca^\pm$ the Hahn's decomposition of a general signed $\sigma$-additive measure. The reason for this more general choice of $\sca$ is the fact that on $\RCD(K,N)$ spaces a ``scalar curvature'' that is bounded is not natural (recall that to solve the Yamabe problem one would like to take $\sca= {\rm Scal}$). Indeed,  requiring only a synthetic lower bound on the Ricci curvature, it is more desirable to impose only lower bounds on $\sca$. 

Recall  that every function $u \in W^{1,2}(\X)$ has a well defined and unique quasi continuous representative ${\sf QCR}(u)$ defined $\Cap$-a.e.. In particular, thanks to \eqref{eq:scalar}, the object ${\sf QCR}(u)$ is also defined $\sca$ or $|\sca|$-a.e.. To avoid heavy notation, for any $u \in W^{1,2}(\X)$, we shall denote in the sequel by $u$ its quasi-continuous representative without further notice.

The goal is then to discuss positive solutions $u \in D(\DDelta)\cap L^2(|\sca |)$ of
	\begin{equation} 
	-\DDelta u  =\lambda u^{2^* -1}\mm -u  \sca, \qquad \lambda \in \R.\label{eq:PDE}
	\end{equation} 
	Observe that if $u \in D(\DDelta)\subset W^{1,2}(\X)$, by the Sobolev embedding we have that $u \in L^{2^*}(\mm)$ and thus, the right hand side of \eqref{eq:PDE} is a well defined Radon measure on $\X.$
A solution for this equation will be deduced with a variational approach as described above. More precisely we define the functional $Q_\sca \colon W^{1,2}(\X)\setminus\{0\} \to \R$ defined as
\[ u\mapsto Q_\sca(u):= \frac{\int |D u|^2\, \d \mm + \int |u|^2 \,\d \sca}{\|u\|_{L^{2^*}(\mm)}^2}.\]
Observe that since $\sca \ge g\mm$, with $g \in L^p(\mm)$, $p>N/2$, the integral $\int |u|^2\, \d \sca$ exists, i.e. its value is well defined. We then define 
\begin{equation}\label{eq:yamabe min}
\begin{split}
\lambda_\sca(\X)&\coloneqq \inf \{ Q_\sca(u) \ : \ u\in  W^{1,2}(\X)\setminus\{0\}\}\\
            &=\inf \{ Q_\sca(u) \ : \ u\in  W^{1,2}(\X), \, \|u\|_{L^{2^*}(\mm)}=1\},
\end{split}
\end{equation}
and  claim that
\begin{equation}\label{eq:Yamabe lower bound}
\lambda_\sca(\X)\in(-\infty,+\infty).
\end{equation}
Indeed, $\lambda_\sca(\X)<+\infty$ as can be seen considering constant functions. On the other hand  for every $u\in  W^{1,2}(\X)$ with $\|u\|_{L^{2^*}(\mm)}=1$, H\"older inequality yields
$$Q_\sca(u)\ge -\|g\|_{L^p(\mm)}\|u\|_{L^{2^*}(\mm)}=-\|g\|_{L^p(\mm)}.$$
The ultimate goal of this section is to prove the following:
\begin{theorem}\label{thm:4.3}
	Let $\Xdm$ be a compact $\RCD(K,N)$ space for some $K \in \R$, $N\in (2,\infty)$ with $\mm(\X)=1$ and let $\sca$ as in \eqref{eq:scalar}. If
	\begin{equation}\label{eq:yamabe assumption}
	    \lambda_\sca(\X) < \frac{\min_\X \theta_N^{2/N} }{\eucl(N,2)^2 },
	\end{equation}
	then there exists a non-negative and non-zero $u \in D(\DDelta)\cap L^2(|\sca|)$ which is a minimum for \eqref{eq:yamabe min} and satisfies \eqref{eq:PDE}.
\end{theorem}

We start by showing that \eqref{eq:PDE} is the Euler-Lagrange equation for the minimization problem \eqref{eq:yamabe min}.
\begin{proposition}\label{prop:existence min}
	Let $\Xdm$ be a compact ${\rm RCD}(K,N)$-space for some $K \in \R,N \in (2,\infty)$ with $\mm(\X)=1$ and let $\sca$ be as in \eqref{eq:scalar}. Suppose  $u\in W^{1,2}(\X)\cap L^2(|\sca|)$ is a minimizer for \eqref{eq:yamabe min} satisfying $\|u\|_{L^{2^*}(\mm)}=1$. Then 
	\begin{equation}  \int \la \nabla u,\nabla v \ra \, \d \mm =-\int uv \, \d \sca + \lambda_\sca(\X) \int u^{{2^*} -1}v\, \d \mm, \qquad \forall v \in \LIP(\X).\label{eq:weakPDEs}
	\end{equation}
\end{proposition}
\begin{proof}
	We consider  for every  $\eps\in(-1,1)$ and $v \in \LIP(\X)$, the function $u^\eps\coloneqq\|u+ \eps v\|_{L^{2^*}(\mm)}^{-1} (u + \eps v )$, whenever $\|u + \eps v\|_{L^{2^*}(\mm)}$ is not zero. It can be seen that  for a fixed $v$ then $u^\eps$ is well defined at least for $\eps$ close to zero. Indeed, the fact that $\int |u|^{2^*},\d \mm=1$ grants that $\|u + \eps v\|_{L^{2^*}(\mm)}\to 1$ as $\eps \to 0$ (see below) and in particular $\|u + \eps v\|_{L^{2^*}(\mm)}$ does not vanish for $|\eps|$ small enough. By minimality  we have (recall also \eqref{eq:scalare})
	\[ 0 \le \lim_{\eps \downarrow 0}\frac{Q_\sca(u^\eps) -Q_\sca(u)}{\eps} = \lim_{\eps \downarrow0} \frac{1}{\eps}\left(\frac{1}{I_\eps^2} -1\right)\lambda_\sca(\X)+ \frac{2}{I_\eps^2} \int \la \nabla u,\nabla v\ra \, \d \mm + \int u v\, \d \sca,\]
	where $I_\eps\coloneqq\|u+ \eps v\|_{L^{2^*}(\mm)}$. Furthermore, from the elementary estimate $| |a+\eps b|^q-|a|^q |\le q|\eps b|\big||a+\eps b|^{q-1}+|a|^{q-1}\big|, $ with $q=2^*,$ and the fact that $u,v\in L^{2^*}(\mm)$, we have that $\int |u+\eps v|^q\, \mm \to 1$ as $\eps \to 0.$ Thanks to  the same estimates,  the dominated convergence theorem grants that
	\[ \lim_{\eps \downarrow 0} \frac{1- I_\eps^2}{\eps} = \frac{2}{2^*}\lim_{\eps \downarrow 0} \int \frac{ |u|^{2^*} - |u+\eps v|^{2^*}}{\eps} \, \d \mm = -2\int  u^{2^*-1}v\, \d \mm. \]
	Arguing analogously considering $\eps\uparrow 0$ gives \eqref{eq:weakPDEs}.
\end{proof}

We can now prove Theorem \ref{thm:4.3} which, thanks to the previous proposition, amounts to the existence of a minimizer for \eqref{eq:yamabe min}. We will do so using the concentration-compactness tools developed in Section \ref{sec:CC}, here employed with a fixed space $\X$.
\begin{proof}[Proof of Theorem \ref{thm:4.3}]
    Let $u_n \in W^{1,2}(\X)$ be such that $Q_\sca(u_n)\to \lambda_\sca(\X)$ and $\|u_n\|_{L^{2^*}(\mm)}=1$. We claim that $u_n$ are uniformly bounded in $W^{1,2}(\X)$. Indeed, this can be seen from the estimate
    	\[
	\int |D u_n|^2 +|u_n|^2 \,\d \mm\le \int |D u_n|^2 \, \d \mm + \int |u_n|^2\, \d \sca + (1+\|g\|_{L^p(\mm)})\|u_n\|_{L^{2^*}(\mm)}= 1+Q_\sca(u_n)+\|g\|_{L^p(\mm)},
	\]
	obtained combining the H\"older inequality with \eqref{eq:scalar}. Hence, by compactness (see Proposition \ref{prop:W12compact}), up to a not relabeled subsequence, we have $u_n \to u$ in $L^2(\mm)$ for some $u \in W^{1,2}(\X).$ Observe that, since $u\in W^{1,2}(\X)$, $u$ admits a quasi-continuous representative (still denoted by $u$) and thus thanks to \eqref{eq:scalar} it makes sense to integrate $u^2$ against $|\sca|$. We  claim that  $u \in L^2(|\sca|)$ and
    \begin{equation}\label{eq:lsc scalar}
       \int u^2 \, \d \sca\le \liminf_n \int u_n^2 \, \d \sca.
    \end{equation}
    Observe first that, by \eqref{eq:scalar}, we have $\sca^-\le |g| \mm$. In particular by the H\"older inequality, denoted by $p'$ the conjugate exponent to $p$, $\int u^2 \d \sca^-\le \|g\|_{L^p(\mm)}\|u\|_{L^{2p'}}^{2}<+\infty $, since $u \in L^{2^*}(\mm)$ by the Sobolev embedding, hence $u \in L^2(\sca^-)$. Moreover, again by the H\"older inequality, since $u_n \to u$ in $L^2(\mm)$, we get that and $u_n \to u$ also in $L^2(\sca^-).$ To prove \eqref{eq:lsc scalar} it remains to prove that $\int u^2 \, \d \sca^+\le \liminf_n \int u_n^2 \, \d \sca^+.$
    Observe first that up to passing to a further non-relabeled subsequence we can assume that the right hand side is actually a limit. From Mazur's lemma there exists a sequence $(N_n)\subset \N$ and numbers $(\alpha_{n,i})_{i=n}^{N_n}\subset [0,1]$ such that $\sum_{i=n}^{N_n}\alpha_{n_i}=1$ for every $n \in N$ and $v_n \coloneqq \sum_{i=n}^{N_n}\alpha_{n_i}u_i$ converges to $u$ strongly in $W^{1,2}(\X).$ In particular from \eqref{eq:w21 conv implies cap} up to a subsequence $v_n \to u$ also $\Cap$-a.e.\ and thus, since $\sca^+ \ll \Cap$ (recall \eqref{eq:scalar}), also $\sca^+$-a.e.. Therefore, from Fatou's Lemma and the convexity of the $L^2$-norm we have
    \[
    \|u\|_{L^2(\sca^+)}\le \liminf_n \|v_n\|_{L^2(\sca^+)}\le  \sum_{i=n}^{N_n}\alpha_{n_i} \|u_i\|_{L^2(\sca)}\le \lim_n \|u_n\|_{L^2(\sca^+)},
    \]
    since we are assuming that the last limit exists. This proves the claim.
    
    We now distinguish two cases:
    
    \noindent {\sc Case} 1. $\lambda_\sca(\X)<0$. By lower semicontinuity of the Cheeger-energy and \eqref{eq:lsc scalar} we have
    \[
    0>\lambda_\sca(\X)=\lim_n Q_\sca(u_n)\ge \int |D u|^2 \, \d \mm +\int u^2\d \sca.
    \]
    In particular $u$ is not identically zero and by the lower semicontinuity of the $L^{2^*}(\mm)$-norm we have $0<\|u\|_{L^{2^*}(\mm)}\le 1$. Moreover, from the above we have that $\int |D u|^2 \, \d \mm +\int u^2\d \sca$ is negative, hence
    \[
    \lambda_\sca(\X) \ge \|u\|_{L^{2^*}(\mm)}^{-2} \left(\int |D u|^2 \, \d \mm +\int u^2\d \sca\right)=Q_\sca(\|u\|_{L^{2^*}(\mm)}^{-1}u).
    \]
    Therefore $\|u\|_{L^{2^*}(\mm)}^{-1}u$ is a minimizer for $Q_\sca(u)$.

     \noindent {\sc Case} 2. $\lambda_\sca(\X)\ge 0$.
    Recall that the sequence $(u_n)$ is uniformly bounded both in $L^{2^*}(\mm)$ and in $W^{1,2}(\X)$. Therefore since $\X$ is compact, again up to a subsequence, $|D u_n|^2 \mm \rightharpoonup \mu$ and $|u_n|^{2^*}\rightharpoonup \nu$ for some $\mu \in \mathscr{M}_b^+(\X)$ and $\nu \in \PP(\X)$ in duality with $C(\X).$ By assumption there exists $\eps>0$ such that $\lambda_\sca(\X)< \frac{\min_\X \theta_N^{2/N}}{\eucl(N,2)^2 +\eps}\eqqcolon \lambda_\epsilon$.  We fix one of such $\eps>0$ and define $A_\eps = \lambda_\eps^{-1}$. From Theorem \ref{thm:alfa} there exists a constant $B_\eps>0$ so that
    \[
    \|u\|_{L^{2^*}(\mm)}^2\le A_\eps\||Du|\|_{L^2(\mm)}^2+B_{\eps}\|u\|_{L^2(\mm)}^2,\quad \forall u \in W^{1,2}(\X).
    \]
    Hence we are in position to apply Lemma \ref{lem:conccompII} (with fixed space $\X$) to deduce that there exists a countable set of indices $J$, points $(x_j)_{j\in J}\subset \X$ and weights $(\mu_j)\subset \R^+ $, $(\nu_j)\subset \R^+ $ such that $\mu_j \ge \lambda_\eps \nu_j^{2/2^*}$ for every $j \in J$ and
    \begin{align*}
        &\nu=|u|^{2^*}\mm+\sum_{j\in J} \nu_j \delta_{x_j}, \quad \mu \ge |D u|^2\mm+\sum_{j \in J}\mu_j\delta_{x_j}.
    \end{align*}
We now observe that
\begin{equation}\label{eq:imporant observation}
    \int |D u|^2\, \d \mm+ \int u^2\, \d \sca \ge \|u\|_{L^{2^*}(\mm)}^2\lambda_\sca(\X).
\end{equation}
Indeed, this is obvious if $u=0$ $\mm$-a.e., hence we assume that $u \neq 0$  $\mm$-a.e.. In this case, \eqref{eq:imporant observation} follows noticing that $\lambda_\sca(\X)\le Q_\sca(u\|u\|_{L^{2^*}(\mm)}^{-1})=\|u\|_{L^{2^*}(\mm)}^{-2} \left(\int |D u|^2\d \mm+ \int u^2\, \d \sca \right)$. 
Therefore using again \eqref{eq:lsc scalar} we have
\begin{align*}
    \lambda_\sca (\X) = \lim_n Q_\sca (u_n) &\ge \mu(\X)+\int u^2\, \d \sca \ge \int |D u|^2\d \mm+\lambda_\eps\sum_{j \in J}\nu_j^{2/2^*} + \int  u^2\, \d \sca\\
        &\overset{\eqref{eq:imporant observation}}{\ge}\|u\|_{L^{2^*}(\mm)}^2\lambda_\sca(\X)+\lambda_\eps\sum_{j \in J}\nu_j^{2/2^*} \ge \lambda_\sca(\X)(\|u\|_{L^{2^*}(\mm)}^2+\sum_{j \in J}\nu_j^{2/2^*})\\
    &\ge \lambda_\sca(\X) \Big( \int |u|^{2^*}\,\d \mm+\sum_{j \in J}\nu_j \Big)^{2/2^*}=\lambda_\sca(\X)\nu(\X) = \lambda_\sca(\X),
\end{align*}
where in the last line,  we used the concavity of the function $t^{2/2^*}$, the fact that $\nu \in \PP(\X)$ and finally that $\lambda_\sca(\X)\ge0.$ Hence all the inequalities are equalities and in particular from the strict concavity of $t^{2/2^*}$ we deduce that either $\int |u|^{2^*}\d \mm=1$ or $u=0$ (and the numbers $\nu_j$ are all zero except one that is equal to one). In the second case, plugging $u=0$ in the above chain of inequalities, we infer that $\lambda_\eps=\lambda_\sca(\X)$  which is a contradiction. Hence, we must have $\|u\|_{L^{2^*}(\mm)}=1$ and $u_n \to u$ strongly in $L^{2^*}(\mm)$ and in particular $u$ is a minimizer for \eqref{eq:yamabe min}. This together with Proposition \ref{prop:existence min} concludes the proof.
\end{proof}

We conclude by extending the classical upper bound \eqref{eq:yamabe upper bound} to the setting of $\RCD(K,N)$ spaces. This in particular shows that \eqref{eq:yamabe assumption} is a reasonable assumption.  Unfortunately, at present, we are able to prove this comparison only by adding  integrability conditions on $\sca$. 
\begin{proposition}\label{prop:lambdaalpha}
     	Let $\Xdm$ be a compact $\RCD(K,N)$ space for some $K \in \R$, $N\in (2,\infty)$ and let $\sca \in L^p (\mm)$, with $p > \frac N2$. Then
	\[ \lambda_\sca(\X) \le \frac{\min_\X \theta_N^{2/N}}{\eucl(N,2)^2}.\]
\end{proposition}
\begin{proof}
	The argument is almost the same as for Theorem \ref{thm:A lower bound}. We start noticing that in the case $\min_\X \theta_N=+\infty$, evidently there is nothing to prove.  We are left then to deal with the case $0< \min_\X \theta_N<+\infty$. Let  $x \in \X$ such that $\theta_N(x)=\min_\X \theta_N.$ Then there exists a sequence $r_i\to 0$ such that the sequence of metric measure spaces $(\X_i,\sfd_i,\mm_i,x_i)\coloneqq(\X,\sfd/r_i,\mm/r_i^N,x)$ pmGH-converges to an $\RCD(0,N)$  space $(\Y,\sfd_\Y,\mm_\Y,\oo_\Y)$ satisfying $\mm_\Y(B_r(\soo_\Y))=\omega_N\theta_N(x) r^N$ for every $r>0$ (this space is actually a cone by \cite{GDP17}). In particular from Lemma \ref{lem:extremal sequence cone} for every $\eps>0$ there exists a non-zero $u \in \LIP_c(\Y)$ such that $\frac{\|u\|_{L^{2^*}(\mm_\Y)}^2}{\||Du|\|_{L^2(\mm_\Y)}^2}\ge  \frac{\eucl(N,2)^2-\eps}{\theta_N(x)^{2/N}}$. Then by  the $\Gamma$-convergences of the $2$-Cheeger energies  there exists a sequence $u_i \in W^{1,2}(\X_i)$ such that $u_i \to u$ strongly in $W^{1,2}$. Moreover, since $u_i$ are uniformly bounded in $W^{1,2}$ (meaning in $W^{1,2}(\X_i)$), by the Sobolev embedding (recall also the scaling property in \eqref{eq:scaling}) we have $\sup_i \|u_i\|_{L^{2^*}(\mm_i)}<+\infty$. In particular from the lower semicontinuity of the $L^{2^*}$-norm we get
	\begin{equation}\label{eq:lambda limit}    \liminf_i\frac{\|u_i\|_{L^{2^*}(\mm)}^2}{\||D u_i|\|_{L^2(\mm)}^2}=\liminf_i\frac{\|u_i\|_{L^{2^*}(\mm_i)}^2}{\||D u_i|_i\|_{L^2(\mm_i)}^2}\ge\frac{\|u\|_{L^{2^*}(\mm_\Y)}^2}{\||Du|\|_{L^2(\mm_\Y)}^2}\ge   \frac{\eucl(N,2)^2-\eps}{\min_\X \theta_N^{2/N}},
	\end{equation}
where $|D u_i|_i$ denotes the weak upper gradient computed in the space $\X_i.$

Denote by $p'\coloneqq p/(p-1)$ the conjugate exponent of $p$ and observe that by hypothesis $2p'<2^*.$ This and  the fact that $u_i$ are bounded in $L^{2^*}$, by Proposition \ref{prop:lp prop} $(viii)$ imply  that $u_i$ converges in $L^{2p'}$-strong to $u$. 
Finally using the H\"older inequality  we can write
		\begin{align*}
			\limsup_i Q_\sca(u_i)&\le \limsup_i  \frac{\int |D u_i|^2\, \d \mm}{\|u_i\|_{L^{2^*}(\mm)}^2}+ \limsup_i \frac{\int \sca|u_i|^2 \,\d \mm}{\|u_i\|_{L^{2^*}(\mm)}^2}\\
			&\overset{\eqref{eq:lambda limit}}{\le}  \frac{\min_\X \theta_N^{2/N}}{\eucl(N,2)^2-\eps} + \limsup_i \|\sca\|_{L^p(\mm)} \frac{\big(\int |u_i|^{2p'} \,\d \mm\big)^{1/p'}}{\|u_i\|_{L^{2^*}(\mm)}^2},\\
			&=\frac{\min_\X \theta_N^{2/N}}{\eucl(N,2)^2-\eps} + \limsup_i \|\sca\|_{L^p(\mm)} r_i^{N\left(\frac{1}{p'}-\frac{2}{2^*}\right)}\frac{\|u_i\|_{L^{2p'}(\mm_i)}^2}{\|u_i\|_{L^{2^*}(\mm_i)}^2}= \frac{\min_\X \theta_N^{2/N}}{\eucl(N,2)^2-\eps}.
		\end{align*}
	where we have used that $1/p'<2/2^*$, that  $\liminf_i \|u_i\|_{L^{2^*}(\mm_i)}\ge  \|u\|_{L^{2^*}(\mm_{\Y})}>0$ and as observed above $\|u_i\|_{L^{2p'}(\mm_i)}\to \|u\|_{L^{2p'}(\mm_{\Y})}$. From the arbitrariness of $\eps>0$ the proof is now concluded.
\end{proof}

\subsection{Continuity of $\lambda_\sca$ under mGH-convergence}\label{sec:continuity yamabe}
In \cite{Honda18} it has been proven in the setting of Ricci-limits a result about mGH-continuity of the generalized Yamabe constant, under some additional boundedness assumption on the sequence. In the following result we extend this fact in the setting of $\RCD$-spaces and we  remove such extra assumption.

We start proving that $\lambda_\sca$ is upper semicontinuous under mGH-convergence.
\begin{lemma}\label{lem:upper lambda}
    Let $(\X_n,\sfd_n,\mm_n)$ be a sequence of compact $\rcd KN$-spaces with $\mm(\X_n)=1$, $n \in \bar \N$, for some $K \in \R, N \in (2,\infty)$ and satisfying $\X_n\overset{mGH}{\to}\X_\infty$. Let also $\sca_n \in L^p(\mm_n)$ be $L^p$-weak convergent to $\sca$, for some $p>N/2$. Then, 
    \begin{equation}
       \limsup_{n\to \infty} \lambda_{\sca_n}(\X_n)\le \lambda_\sca(\X_\infty).\label{eq:lambdalimsup}
    \end{equation} 
\end{lemma}
\begin{proof}
 Fix a non-zero $u \in W^{1,2}(\X_\infty)$. By the Sobolev embedding on $\X_\infty$ we know that $u \in L^{2^*}(\mm_\infty)$, therefore by Lemma \ref{lem:mixrecov} there exists  a sequence $u_n \in W^{1,2}(\X_n)$ that converge $W^{1,2}$-strong and $L^{2^*}$-strong to $u$. By definition of $\lambda_{\sca_n}(\X_n)$, we have 
		\[ \| u_n\|_{L^{2^*}(\mm_n)}^2 \lambda_{\sca_n}(\X_n) \le \int |Du_n|^2\, \d \mm_n + \int \sca_n |u_n|^2\, \d \mm_n, \qquad \forall n \in \N.\]
	From the assumption that $p>N/2$, we have that its conjugate exponent $p'$ satisfies $2p'<2^*$, therefore from $(vii)$, $(viii)$ in Proposition \ref{prop:lp prop} we have that $|u_n|^2$ $L^{p'}$-strongly converges to $u^2$. Recalling Proposition \ref{prop:coupling}, we get that all the above quantities pass to the limit and thus we reach
		\[  \| u\|_{L^{2^*}(\mm_\infty)}^2 \limsup_{n \to \infty}\lambda_{\sca_n}(\X_n) \le \int |Du|^2\, \d \mm_\infty + \int \sca |u|^2\, \d \mm_\infty.\]
		By arbitrariness of $u$, we conclude.
\end{proof}
We shall now come to the main continuity result.
\begin{theorem}[mGH-continuity of $\lambda_\sca$]\label{thm:GHstablelambda}
    Let $(\X_n,\sfd_n,\mm_n)$ be a sequence of compact $\rcd KN$-spaces with $\mm(\X_n)=1$, $n \in \bar \N$, for some $K \in \R, N \in (2,\infty)$ satisfying $\X_n\overset{mGH}{\to}\X_\infty$. Let also $\sca_n \in L^p(\mm_n)$ be $L^p$-weak convergent to $\sca \in L^p(\mm_\infty)$, for a given for $p>N/2$. Then, 
    \[ \lim_{n\to \infty}\lambda_{\sca_n}(\X_n) = \lambda_{\sca}(\X_\infty).\]
\end{theorem}
\begin{proof}
    In light of Lemma \ref{lem:upper lambda}, we only have to prove that $$ \liminf_{n\to \infty}\lambda_{\sca_n}(\X_n) \ge \lambda_{\sca}(\X_\infty).$$ It is not restrictive to assume that the $\liminf$ is actually a limit.	For every $n \in \N$, we take $u_n \in W^{1,2}(\X_n)$ non-zero so that $Q_{\sca_n}(u_n) - \lambda_{\sca_n}(\X_n) \le n^{-1}$. In other words
		\begin{equation}
			\|u_n\|_{L^{2^*}(\mm_n)}^2 \big( \lambda_{\sca_n}(\X_n)+ \tfrac 1n\big) \ge \int |Du_n|^2\, \d \mm_n + \int  \sca_n |u_n|^2\, \d \mm_n. \label{eq:lambdan}
		\end{equation}
	It is also clearly not restrictive to suppose that $u_n \in \LIP_c(\X_n)$ are non-negative and such  that $\|u_n\|_{L^{2^*}(\mm_n)}\equiv 1$. Hence, arguing as in the proof of Theorem \ref{thm:4.3} (using also \eqref{eq:lambdalimsup}), we get that $u_n$ is uniformly bounded in $W^{1,2}$. Then, by compactness (see Proposition \ref{prop:W12compact}), up to a not relabeled subsequence, we have that $u_n$ converge $L^2$-strong and $W^{1,2}$-weak to some $u_\infty \in W^{1,2}(\X_\infty)$. From $\|u_n\|_{L^{2^*}(\mm_n)}\equiv 1$ and the assumption $p>N/2$,  Proposition \ref{prop:lp prop} implies that $u_n^2$ converges $L^{p/(p-1)}$-strongly to $u_\infty^2$ and that $u_n$ converges $L^{2p/(p-1)}$-strongly to $u_\infty$. From this point we subdivide the proof in three cases to be handled separately.

		\noindent\textsc{Case 1: }$\lim_n \lambda_{\sca_n}(\X_n) < 0$. In this case, by \eqref{eq:lambdan} we know by lower semicontinuity of the $2$-Cheeger energy and Proposition \ref{prop:coupling}, we have that
		\[ 0> \lim_n \lambda_{\sca_n}(\X_n) \ge \int |Du_\infty|^2\, \d \mm_\infty + \int \sca u_\infty^2\,\d \mm_\infty. \]
		In particular, $u_\infty$ is not $\mm_\infty$-a.e.\ equal to zero and by weak-lower semicontinuity, we have that $0<\|u_\infty\|_{L^{2^*}(\mm_\infty)} \le 1$. Therefore
		\[ \|u_\infty\|_{L^{2^*}(\mm_\infty)} \lim_n \lambda_{\sca_n}(\X_n) \ge \lim_n \lambda_{\sca_n}(\X_n) \ge \int |Du_\infty|^2\, \d \mm_\infty + \int \sca u_\infty^2\,\d \mm_\infty\ge \lambda_{S}(\X_\infty) \|u_\infty\|_{L^{2^*}(\mm_\infty)} , \]
	which concludes the proof in this case.
	
		\noindent\textsc{Case 2: }$\lim_n \lambda_{\sca_n}(\X_n) >0 $. Before starting, notice that by using the H\"older inequality, for any $n \in \N$ and any $u \in W^{1,2}(\X_n)$ we have by the definition of $\lambda_{\sca_n}(\X_n) $  that
		\begin{equation}\label{eq:disturbed 2} 
			\|u\|_{L^{2^*}(\mm_n)}^2  \le \lambda_{\sca_n}(\X_n)^{-1}\int |Du|^2\, \d \mm_n + \lambda_{\sca_n}(\X_n)^{-1}\| \sca_n\|_{L^p(\mm_n)} \| u\|^2_{L^{2p/p-1}(\mm_n)}. \end{equation}
		Moreover, since all $\X_n$ are compact and renormalized, there are $\mu \in \MM_b^+(\Z), \nu \in \PP(\Z)$ so that, up to a not relabeled subsequence, $|Du_n|^2\mm_n \weakto \mu$ and $|u_n|^{2^*}\mm_n \weakto \nu$ in duality with $C(\Z)$ as $n$ goes to infinity, where $(\Z,\sfd_\Z)$ is a (compact) space realizing the convergences via extrinsic approach. Since we are assuming that $\lim_n \lambda_{\sca_n}(\X_n) >0$, the constant in \eqref{eq:disturbed 2} are uniformly bounded (for $n$ big enough) and  we are in position to apply Lemma \ref{lem:conccompII}. In particular we  get the existence of an at most countable set $J$, points $(x_j)_{j \in\ J}\subset \X_\infty$ and weights $ (\mu_j),(\nu_j) \subset \R ^+$, so that $\mu_j \ge \lim_n \lambda_{\sca_n}(\X_n) \nu_j^{2/2^*}$ with $j \in J$ and
		\begin{align*}
			&\nu=|u_\infty|^{2^*}\mm+\sum_{j\in J} \nu_j \delta_{x_j}, \quad \mu \ge |D u_\infty |^2\mm+\sum_{j \in J}\mu_j\delta_{x_j}.
		\end{align*}
		Moreover, recalling Proposition \ref{prop:coupling} we have
		\begin{equation} \mu(\X) + \int \sca u_\infty^2 \, \d \mm_\infty = \lim_{n\to \infty} Q_{\sca_n}(u_n) \overset{\eqref{eq:lambdan}}{\le} \lim_{n\to \infty}\lambda_{\sca_n}(\X_n),\label{eq:item1}
		\end{equation}
		and, arguing as in the proof of \eqref{eq:imporant observation}, $u_\infty$ is so that $\|u_\infty\|_{L^{2^*}(\mm_\infty)} \lambda_\sca(\X_\infty) \le \int |Du_\infty|^2\, \d \mm_\infty + \int \sca|u_\infty|^2 \, \d \mm_\infty$.
		Finally, we can perform the chain of estimates
		\begin{align*}
			\lim_{n\to \infty}\lambda_{\sca_n} (\X_n )  &\overset{\eqref{eq:item1}}{\ge} \mu(\X)+\int \sca u_\infty^2\, \d \mm_\infty \ge \int |D u_\infty |^2\d \mm_\infty +\lim_{n\to \infty} \lambda_{\sca_n}(\X_n) \sum_{j \in J}\nu_j^{2/2^*} + \int   \sca u_\infty ^2\, \d \mm_\infty \\
			&\ge  \lambda_{\sca}(\X_\infty)\|u_\infty\|_{L^{2^*}(\mm_\infty)}^2+\lim_{n\to \infty} \lambda_{\sca_n}(\X_n) \sum_{j \in J}\nu_j^{2/2^*}\\
			 &\overset{\eqref{eq:lambdalimsup}}{\ge} \lim_{n\to \infty} \lambda_{\sca_n}(\X_n)\Big(\|u_\infty\|_{L^{2^*}(\mm_\infty)}^2+\sum_{j \in J}\nu_j^{2/2^*}\Big)\\
			&\ge\lim_{n\to \infty} \lambda_{\sca_n}(\X_n) \Big( \int |u_\infty|^{2^*}\,\d \mm_\infty+\sum_{j \in J}\nu_j \Big)^{2/2^*} \ge\lim_{n\to \infty}\lambda_{\sca_n}(\X_n),
		\end{align*}
		where in the last line, we used the concavity of $t^{2/2^*}$ and the fact that $\nu \in \PP(\X)$. In particular, all inequalities must be equalities and by the strict concavity of $t^{2/2^*}$  either $\|u_\infty\|_{L^{2^*}(\mm_\infty)}=1$ and all $\nu_j=0$, or $u_\infty=0$ $\mm_\infty$-a.e.\  and all the weights are zero  except one $\nu_j=1$. The first situation is the easiest one, as in this case the above inequalities which are actually equalities imply that $ \lambda_{\sca}(\X_\infty)= \lim_n \lambda_{\sca_n}(\X_n)$, which is what we wanted. Therefore we suppose that we are in the second case, i.e. that there exists a point $y_0 \in \X_\infty$ so that $|u_n|^{2^*}\mm_n \weakto \delta_{y_0}$ in duality with $C(\Z)$ and that $u_n$ converges in  $L^2$-strong to zero. Moreover, from \eqref{eq:lambdan} and H\"older inequality we get
		\[ \|u_n\|_{L^{2^*}(\mm_n)}^2 \ge  \big( \lambda_{\sca_n}(\X_n)+ \tfrac 1n\big)^{-1}\Big( \int |Du_n|^2\, \d \mm_n - \|\sca_n\|_{L^p(\mm_n)}\| u_n\|^2_{L^{2p/(p-1)}(\mm_n)}\Big) ,\qquad \forall n \in \N.\]
		We can therefore apply Lemma \ref{lem:concpoint} to get that $\theta_N(y_0) \le \eucl(N,2)^N \lim_n\lambda_{\sca_n}(\X_n)^{N/2}$. Finally, we can rearrange and invoke Proposition \ref{prop:lambdaalpha} to get
		\[\lim_n\lambda_{\sca_n}(\X_n) \ge \frac{\theta_N(y_0)^{2/N}}{ \eucl(N,2)^2} \ge \lambda_\sca(\X_\infty). \]
		\noindent\textsc{Case 3: }$\lim_n \lambda_{\sca_n}(\X_n) = 0$. The argument is the same as in the previous case, only that we replace \eqref{eq:disturbed 2}  with the Sobolev inequality given in Proposition \ref{prop:sobolev embedding general}:
		\begin{equation}\label{eq:case 3}
			\|u\|_{L^q(\mm)}^2\le A(K,N,D)\||D u|\|_{L^2(\mm)}^2+\|u\|_{L^2(\mm_n)}^2, \qquad \forall u \in W^{1,2}(\X_n),
		\end{equation}
	where $D>0$ is constant such that $\diam(\X_n)\le D$. Then we can apply exactly as in the previous case Lemma \ref{lem:conccompII}, except that in this case we obtain $\mu_j \ge  A(K,N,D)^{-1} \nu_j^{2/2^*}$ for every $j \in J$. Then the above chain of estimates becomes
	\begin{align*}
	0=\lim_{n\to \infty}\lambda_{\sca_n} (\X_n )  &\overset{\eqref{eq:item1}}{\ge} \mu(\X)+\int \sca u_\infty^2\, \d \mm_\infty \\
	&\ge \int |D u_\infty |^2\d \mm_\infty +  A(K,N,D)^{-1}\sum_{j \in J}\nu_j^{2/2^*} + \int   \sca u_\infty ^2\, \d \mm_\infty \\
	&\ge  \lambda_{\sca}(\X_\infty)\|u_\infty\|_{L^{2^*}(\mm_\infty)}^2+  A(K,N,D)^{-1} \sum_{j \in J}\nu_j^{2/2^*}\\
	&\overset{\eqref{eq:lambdalimsup}}{\ge} \lim_{n\to \infty} \lambda_{\sca_n}(\X_n)\|u_\infty\|_{L^{2^*}(\mm_\infty)}^2+A(K,N,D)^{-1}\sum_{j \in J}\nu_j^{2/2^*}\ge0.
\end{align*}
Therefore we must have that $\nu_j=0$ for every $j \in J$. This forces $\|u_\infty\|_{L^{2^*}(\mm_\infty)}^2=1$ giving in turn that $\lambda_{\sca}(\X_\infty)=0$. Having examined all the three cases, the proof is now concluded.
\end{proof}

\medskip

\textbf{Acknowledgements.} We wish to thank A. Mondino and D. Semola for their useful comments and suggestions on a preliminary version of this manuscript.

%#######################################################################################

%\medskip

    %\addcontentsline{toc}{chapter}{Bibliography}
   %\parskip0pt 
	%\itemsep0pt
   \bibliographystyle{siam}

\end{document}